\documentclass[11pt]{amsart}

\usepackage[normalem]{ulem} %temporary

\usepackage{fullpage}
\usepackage{amssymb,amsmath,amsfonts,amsthm,mathrsfs}
\usepackage[all]{xy} \SelectTips{cm}{10}
\usepackage{color}
\usepackage{comment}
\definecolor{webcolor}{rgb}{0.8,0,0.2}
\definecolor{webbrown}{rgb}{.6,0,0}
\usepackage[
        colorlinks,
        linkcolor=webbrown,  filecolor=webcolor,  citecolor=webbrown, 
        backref,
        pdfauthor={Samuel Le Fourn},
        pdfauthor={Davide Lombardo},
        pdfauthor={David Zywina}, 
       pdftitle={Torsion bounds for a fixed abelian variety and varying number field},
]{hyperref}
\usepackage[alphabetic,backrefs,lite]{amsrefs} % for bibliography

\numberwithin{equation}{section}

\renewcommand{\AA}{\mathbb A}

\newcommand{\CC}{\mathbb C}
\newcommand{\C}{\mathbb C}

\newcommand{\FF}{\mathbb F}
\newcommand{\GG}{\mathbb G}

\newcommand{\PP}{\mathbb P}
\newcommand{\QQ}{\mathbb Q}

\newcommand{\ZZ}{\mathbb Z}

\newcommand{\calG}{\mathcal G}  \newcommand{\calF}{\mathcal F}
\newcommand{\calH}{\mathcal H}

\newcommand{\calS}{\mathcal S}

\newcommand{\calC}{\mathcal C}

\newcommand{\calW}{\mathcal W}

\newcommand{\calV}{\mathcal V}

\newcommand{\gG}{\mathfrak g} 
\newcommand{\hH}{\mathfrak h}

\newcommand{\boldG}{\mathbf G}

\renewcommand{\P}{\mathbb{P}}

\def\Hom{{\operatorname{Hom}}}

\def\un{{\operatorname{un}}}

\def\Spec{\operatorname{Spec}} 
\def\Gal{\operatorname{Gal}}
 
\def \GL {\operatorname{GL}}  
\def \gl {\mathfrak{gl}}

\def \GSp {\operatorname{GSp}}  

\def \SL {\operatorname{SL}}

\def\Aut{\operatorname{Aut}} 
\def\End{\operatorname{End}}

\newcommand{\tors}{{\operatorname{tors}}}

\newcommand{\defi}[1]{\textsf{#1}} % for defined terms

\newcommand\blank[1]{}

\def\bbar#1{\setbox0=\hbox{$#1$}\dimen0=.2\ht0 \kern\dimen0 
\overline{\kern-\dimen0 #1}}
\newcommand{\Qbar}{{\overline{\mathbb Q}}} 
\newcommand{\Kbar}{{\bbar{K}}}

\newcommand{\Fbar}{\bbar{F}} 
\newcommand{\FFbar}{\overline{\FF}} 

\newtheorem{thm}{Theorem}[section]
\newtheorem{lemma}[thm]{Lemma}
\newtheorem{cor}[thm]{Corollary}
\newtheorem{prop}[thm]{Proposition}
\newtheorem{conj}[thm]{Conjecture}

\theoremstyle{definition}
\newtheorem{definition}[thm]{Definition}

\theoremstyle{remark}
\newtheorem{remark}[thm]{Remark}

\usepackage[usenames,dvipsnames]{xcolor}
\usepackage{csquotes}
\usepackage{enumitem}

\newenvironment{alphenum}{\hfill \begin{enumerate}[label=\emph{(\alph*)}] 
}{\end{enumerate}}

\newenvironment{romanenum}{\hfill \begin{enumerate}[label=\emph{(\roman*)}] 
}{\end{enumerate}}

\newcommand{\Fcal}{{\mathcal F}}
\newcommand{\Gcal}{{\mathcal G}}
\newcommand{\Ocal}{{\mathcal O}}
\newcommand{\Ucal}{{\mathcal U}}
\newcommand{\Xcal}{{\mathcal X}}
\newcommand{\Wcal}{{\mathcal W}}

\newcommand{\A}{{\mathbb{A}}}
\newcommand{\F}{{\mathbb{F}}}
\newcommand{\N}{{\mathbb{N}}}
\newcommand{\Q}{{\mathbb{Q}}}
\newcommand{\Z}{{\mathbb{Z}}}

\newcommand{\dd}{\mathrm{d}}

\newcommand{\Ker}{\operatorname{ker}}

\newcommand{\Grass}{\operatorname{Grass}}
\newcommand{\Grassdn}{\operatorname{Grass}_{d, n}}

\newcommand{\mt}{\mapsto}	
\newcommand{\ra}{\rightarrow}

\theoremstyle{definition}

\newcommand{\Fix}{\operatorname{Fix}}
\newcommand{\fix}{\operatorname{fix}}

\begin{document}
\title[Torsion bounds for a fixed abelian variety and varying number field]{Torsion bounds for a fixed abelian variety and varying number field}

\author{Samuel Le Fourn}
\address{Institut Fourier, Université Grenoble Alpes, 38610 Gières, France}
\email{Samuel.Le-Fourn@univ-grenoble-alpes.fr}
\author{Davide Lombardo}
\address{Dipartimento di Matematica, Università di Pisa, Largo Bruno Pontecorvo 5, 56127 Pisa, Italy}
\email{davide.lombardo@unipi.it}
\author{David Zywina}
\address{Department of Mathematics, Cornell University, Ithaca, NY 14853, USA}
\email{zywina@math.cornell.edu}

\subjclass[2020]{Primary 11G10; Secondary 14K15}

%% MSC-class: 11G10 (Primary) 14K15 (Secondary)

\begin{abstract}
Let $A$ be an abelian variety defined over a number field $K$.  For a finite extension $L/K$, the cardinality of the group $A(L)_{\tors}$ of torsion points in $A(L)$ can be bounded in terms of the degree $[L:K]$.   We study the smallest real number $\beta_A$ such that for any finite extension $L/K$ and $\varepsilon>0$, we have $|A(L)_{\tors}| \leq C \cdot [L:K]^{\beta_A+\varepsilon}$, where the constant $C$ depends only on $A$ and $\varepsilon$ (and not $L$).  Assuming the Mumford--Tate conjecture for $A$, we will show that $\beta_A$ agrees with the conjectured value of Hindry and Ratazzi.  We also give a similar bound for the maximal order of a torsion point in $A(L)$.
\end{abstract}

\maketitle

\section{Introduction} \label{S:intro}

Let $A$ be a nonzero abelian variety defined over a number field $K$.  For every finite extension $L$ of $K$, the group $A(L)_{\tors}$ of torsion points in $A(L)$ is finite. We are interested in finding upper bounds for the cardinality of $A(L)_{\tors}$ that depend only on $A$ and the degree $[L:K]$.    A theorem of Masser \cite{masser-lettre} implies that for any real number $\beta>\dim A$ and any finite extension $L/K$, we have $|A(L)_{\tors}|\leq C\cdot [L:K]^\beta$, where $C$ is a constant depending only on $A$ and $\beta$.   We will see that Masser's bound  remains true if $\dim A$ is replaced by some smaller value whenever $A$ is not isogenous over $\Kbar$ to a power of a CM elliptic curve, cf.~\S\ref{S:Masser improvement}.  

Let $\beta_A$ be the infimum of the set of real numbers $\beta$ for which the inequality 
\[
|A(L)_{\tors}|\leq C\cdot [L:K]^\beta
\]
holds for all finite extensions $L/K$, where $C$ is a constant that depends only on $A$ and $\beta$ (and in particular not $L$).   From Masser, we have $\beta_A \leq \dim A$.  This is made totally explicit in a recent preprint of Gaudron and Rémond \cite{GaudronRemond}: namely, they prove that
\[
|A(L)_{\tors}| \leq (6g)^{8g} D^g \max(1,h_\Fcal(A),\log D)^g,
\]
where $g = \dim A$, $D = [L:\Q]$ and $h_\Fcal(A)$ is the stable Faltings height of $A$.

Hindry and Ratazzi have made a precise conjecture for the value of $\beta_A$ which we now recall.     Fix an embedding $K\subseteq \CC$.
The abelian variety $A_\CC$, obtained by base extending $A$ to $\CC$, is isogenous to a product $\prod_{i=1}^n A_i^{m_i}$, where the $A_i$ are abelian varieties over $\CC$ that are simple and pairwise nonisogenous.    For each subset $I \subseteq \{1,\ldots, n\}$, define the abelian variety $A_I:= \prod_{i\in I} A_i^{m_i}$ over $\CC$.    
Associated to each abelian variety $B$ defined over $K$ or $\CC$ is a \emph{Mumford--Tate group} $G_{B}$ whose definition we recall in \S\ref{SS:MT def}; it is a linear algebraic group defined over $\QQ$.   Define the real number
\[
\gamma_A := \max_{\emptyset \neq I \subseteq \{1,\ldots, n\}}  \frac{ 2 \dim A_I}{\dim G_{A_I} }.
\]

Hindry and Ratazzi have conjectured that $\beta_A=\gamma_A$, cf.~\cite[Conjecture~1.1]{MR2862374}; note that $A_I$ and $\prod_{i\in I} A_i$ have isomorphic Mumford--Tate groups.    They  proved the inequality $\beta_A\geq \gamma_A$ \cite[Proposition~1.5]{MR2643390}.

Hindry and Ratazzi have proved their conjecture in various situations where the Mumford--Tate conjecture is known and the Mumford-Tate group $G_A$ is of a very special form  \cites{MR2419854,MR2643390,MR2862374,MR3576113}. For example, if $G_A$ is isomorphic to $\GSp_{2g}$ (with $g=\dim A$) and the Mumford--Tate conjecture holds for $A$, then $\beta_A$ equals $\gamma_A = 2g/(2g^2+g+1)$, cf.~\cite{MR2862374}.    Cantoral Farf\'an has  proved several additional cases, see~\cite{VCF}.    A statement of the Mumford--Tate conjecture can be found in \S\ref{SS:MT conjecture}.  

The following is our main result; we prove the conjecture of Hindry and Ratazzi assuming the Mumford--Tate conjecture.
\begin{thm} \label{T:main}
Let $A$ be a nonzero abelian variety defined over a number field $K$ for which the Mumford--Tate conjecture holds.  Then $\beta_A=\gamma_A$.  Equivalently, $\gamma_A$ is the smallest real value such that  for any finite extension $L/K$ and real number $\varepsilon>0$, we have
\[
|A(L)_{\tors}| \leq C\cdot [L:K]^{\gamma_A + \varepsilon},
\]
where $C$ is a constant that depends only on $A$ and $\varepsilon$.  
\end{thm}

This result is in fact implied by an unconditional one, namely Theorem \ref{T:main unconditional}, relating $\beta_A$ to invariants associated to all the $\ell$-adic monodromy groups of $A$.

\begin{remark}
In addition to recovering all previous results on the Hindry--Ratazzi conjecture, Theorem \ref{T:main} proves many new cases of it: for example, the Mumford-Tate conjecture is known to hold for all geometrically simple abelian varieties of prime dimension \cite{MR701568}, while the Hindry--Ratazzi conjecture had not been previously proven for such varieties.
\end{remark}

In the case where $A$ is geometrically simple, the following shows that the converse of Theorem~\ref{T:main} holds.  So the Mumford--Tate conjecture assumption in Theorem~\ref{T:main} is reasonable and the value $\beta_A$ is an interesting arithmetic invariant of $A$. 

\begin{thm}  \label{T:equivalent version}
Let $A$ be a geometrically simple abelian variety defined over a number field $K$.   Then the Mumford--Tate conjecture for $A$ holds if and only if $\beta_A=\gamma_A$.
\end{thm}

We also show that the Hindry-Ratazzi conjecture (for all abelian varieties) is in fact equivalent to the Mumford--Tate conjecture (for all abelian varieties).

\begin{thm} \label{T:equivalent version 2}
The following are equivalent:
\begin{alphenum}
\item \label{T:equivalent version 2 a}
the Mumford--Tate conjecture holds for all abelian varieties $A$ defined over a number field,
\item \label{T:equivalent version 2 b}
$\beta_A=\gamma_A$ for all nonzero abelian varieties $A$ defined over a number field,
\item \label{T:equivalent version 2 c}
$\beta_A=\gamma_A$ for all geometrically simple abelian varieties $A$ defined over a number field.
\end{alphenum}
\end{thm}

\begin{remark}
Take any $\varepsilon>0$ and any torsion point $P \in A(\Kbar)$ whose order we denote by $n$.  A direct consequence of Theorem~\ref{T:main} is that, assuming the Mumford--Tate conjecture for $A$, we have 
\[
[K(P): K] \geq C \cdot n^{1/\gamma_A - \varepsilon},
\]
where $C$ is a positive constant depending only on $A$ and $\varepsilon$. Here, the constant $1/\gamma_A$ is often not best possible.  The next theorem describes an optimal version assuming the Mumford--Tate conjecture for $A$; an unconditional version can be found in \S\ref{S:one torsion point}. 
\end{remark}

With a fixed embedding $K\subseteq \CC$, the Mumford--Tate group $G_A$ of $A$ comes with a faithful action on the $\QQ$-vector space $V_A:=H_1(A(\CC),\QQ)$, cf.~\S\ref{SS:MT def}.

\begin{thm} \label{T:main d1}
Let $A$ be a nonzero abelian variety defined over a number field $K$ for which the Mumford--Tate conjecture holds.  Let $d$ be the minimal dimension of the orbits of the action of $G_A(\CC)$ on the nonzero vectors in $V_A \otimes_\QQ \CC$.    Take any real number $\varepsilon>0$.  Then for any integer $n\geq 1$ and any point $P \in A(\Kbar)$ of order $n$, we have 
\[
[K(P):K] \geq C\cdot n^{d - \varepsilon},
\]
where $C$ is a positive constant depending only on $A$ and $\varepsilon$.   Moreover, $d$ is the minimal real number that has this property.
\end{thm}

\begin{remark}\label{rmk: computing d over Q and over C}
It is easy to construct examples for which the minimal dimension $d_{\Q}$ of the orbits of the action of $G_A$ on the nonzero vectors in $V_A$ is strictly greater than the quantity $d$ appearing in the previous theorem. The simplest such case arises when $A$ is an elliptic curve with complex multiplication: $G_A$ is then a nonsplit torus of dimension 2, and one sees easily that $d_\Q=2$ while $d=1$.
\end{remark}

\subsection{Notation} \label{SS:notation}

\begin{itemize}
\item 
$\ell$ will always denote a rational prime.

 \item 
For a scheme $X$ over a commutative ring $R$ and a (commutative) $R$-algebra $S$, we will denote by $X_S$ the $S$-scheme $X \times_{\Spec R} \Spec S$.

\item 
Fix a commutative ring $R$ and a free $R$-module $M$ of finite rank.  We define $\GL_M$ to be the group scheme over $R$ such that for each (commutative) $R$-algebra $B$, we have $\GL_M(B)=\Aut_{B}(B\otimes_R M)$ with the obvious functoriality.   A choice of basis of the $R$-module $M$ induces an isomorphism $\GL_M \cong \GL_{d, R}$, where $d$ is the rank of $M$.  

\item
Let $G$ be an algebraic subgroup of $\GL_V$, where $V$ is a nonzero vector space over a field $k$.   For a subspace $W$ of $V$, we let $G_W$ be the algebraic subgroup of $G$ that fixes $W$; more precisely, we have $G_W(B)=\{g\in G(B): g w = w \text{ for all } w\in B\otimes_k W\}$ for all $k$-algebras $B$.   

\item  Let $R$ be a commutative ring, $\calG$ be a group subscheme of $\GL_{n, R}$, and $\Wcal$ be a submodule of $R^n$ such that $R^n/\Wcal$ is locally free of rank $n-d$. We define a group subscheme $\Fix_\Gcal(\Wcal)$ of $\calG$ by the functorial property
\[
 \Fix_\Gcal(\Wcal)(S) = \{ \varphi \in \Gcal(S) \, | \, \varphi_{|\Wcal \otimes_R S}  \textrm{ is the identity on }\Wcal \otimes_R S\}
\]
for every $R$-algebra $S$. When $R=k$ is a field and $\Wcal=W$ is a vector subspace of $k^n$ we have $\Fix_{\Gcal}(\Wcal) = \Gcal_W$.

\item  We define the Mumford--Tate group $G_A$ of an abelian variety $A$ in \S\ref{SS:MT def}.

\item We denote by $T_\ell(A)$ the $\ell$-adic Tate module of $A$ and set $V_\ell(A) = T_\ell(A) \otimes_{\Z_\ell} \Q_\ell$. The precise definitions of these objects are recalled in \S\ref{SS:monodromy groups}. We define the $\ell$-adic monodromy groups $G_{A,\ell} \subseteq \GL_{V_\ell(A)}$ over $\Q_\ell$ and   $\Gcal_{A,\ell} \subseteq \GL_{T_\ell(A)}$ over $\Z_\ell$ in \S\ref{SS:monodromy groups}.

\item 
For a field $k$, a finite-dimensional $k$-vector space $V$ and a linear algebraic subgroup $G$ of $\GL_{V}$, define the \defi{slope}
\[
\gamma_G := \max_{\substack{0 \neq W \subseteq V}} \frac{\dim W}{\dim G - \dim G_W},
\]
where we vary over the nonzero subspaces $W$ of $V$.   We interpret this as $\gamma_G=+\infty$, when for some nonzero subspace $W$, the stabilizer $G_W$ has finite index in $G$. If $\gamma_G$ is finite, we say that $G$ has \defi{finite slope}. Note that if $\gamma_G$ is finite, then $\gamma_G \leq \dim V$.

\item 
For two positive real numbers $a$ and $b$, by $a\ll b$ (or $b\gg a$), we mean that $a \leq C b$ for a positive constant $C$; the dependencies of the constant $C$ will always be indicated by subscripts.  For example, Masser's result mentioned above says that for any finite extension $L/K$ and number $\beta>\dim A$, we have $|A(L)_{\tors}| \ll_{A,\beta} [L:K]^{\beta}$.   We will write $a \asymp b$ to denote that  $a\ll b$ and $a \gg b$ both hold, where the dependencies in the implicit constants will be indicated by subscripts. 
\end{itemize}

\subsection{Overview}

In \S\ref{S:background}, we give some background on the $\ell$-adic representations associated to an abelian variety, review some of their uniformity properties, and recall the Mumford--Tate conjecture.

The group $G:=(G_A)_\CC$ acts on the complex vector space $V_\CC:=H_1(A(\CC),\QQ)\otimes_\QQ \CC$.    For each subspace $W$ of $V_\CC$, we have an algebraic subgroup $G_W$ of $G$ as defined in \S\ref{SS:notation}.  In \S\ref{S:codimension bounds new}, we prove that the inequality
\begin{align} \label{E:inequality intro}
\gamma_A\cdot (\dim G - \dim G_W) \geq \dim W
\end{align}
always holds and that $\gamma_A$ is the smallest real number with this property. 

Let $\ell$ be any prime. In \S\ref{S:prime power version} and \S\ref{S: small primes}, we prove a version of Theorem~\ref{T:main} for the subgroup $A(L)[\ell^\infty]$ consisting of the points of $A(L)$ whose order is a power of $\ell$.   More precisely, we prove that if the Mumford--Tate conjecture for $A$ holds, then for any finite extension $L/K$, we have $|A(L)[\ell^\infty]| \ll_{A, \ell}  [L:K]^{\gamma_A}$ (this follows from Theorem~\ref{T: small primes} and Lemma~\ref{L:gamma connection}(\ref{L:gamma connection ii})). Theorem~\ref{T:ell-adic version, large ell}, proved in \S\ref{S:prime power version}, further shows that the implicit constant can be taken to be independent of $\ell$. 

In \S\ref{S:beta bounds}, we obtain upper and lower bounds for $\beta_A$, which agree under the assumption of the Mumford--Tate conjecture. This establishes our main theorems. In \S\ref{S:some remarks}, we make some remarks on a conjectural expression for $\beta_A$ not involving Mumford--Tate groups.  In \S\ref{sec: Masser bound} we show that $\beta_A \leq \dim A$, with equality only for powers of CM elliptic curves. This shows that our results improve on Masser's bound except for those cases. Finally, in \S\ref{S:one torsion point}, we prove Theorem~\ref{T:main d1} and also give an unconditional version.

\subsection{Acknowledgments} The first author is supported by the IRS grant QUAD (Labex Persyval), a PEPS JCJC grant 2022 and the ANR project JINVARIANT. The second author gratefully acknowledges funding from MIUR (Italy, grant PRIN 2017 ``Geometric, algebraic and analytic
methods in arithmetic") and from the University of Pisa (grant PRA 2018-19 ``Spazi di moduli, rappresentazioni e strutture
combinatorie"). The first two authors thank Ga\"el Rémond for many inspiring discussions and for his insightful comments and ideas, which greatly helped them with writing parts of this paper. We are grateful to Marc Hindry and Nicolas Ratazzi for their interest in this work and for asking the questions that led to the results in \S\ref{S:Masser improvement}.

\section{Abelian varieties background}  \label{S:background}

In this section, except for \S\ref{SS:MT def}, we fix an abelian variety $A$ of dimension $g\geq 1$ defined over a number field $K$.   We review some of the theory of the $\ell$-adic representations associated to $A$ and the Mumford--Tate conjecture.

\subsection{Mumford--Tate groups} \label{SS:MT def}

Let $A$ be a nonzero abelian variety defined over $\CC$.  We now recall the definition of the Mumford--Tate group $G_A$ of $A$.    If instead $A$ is defined over a number field $K$, then with a fixed embedding $K\subseteq \CC$, we define $G_A$ to be the Mumford--Tate group of $A_\CC$. The choice of embedding $K$ into $\C$ does not affect any of the following constructions by \cite[Theorem 2.11]{MR654325}.

We view $A(\CC)$ as a topological space with its usual complex topology.  The first homology group $V:=H_1(A(\CC),\QQ)$ is a vector space of dimension $2\dim A$ over $\QQ$.   It is endowed with a $\QQ$-Hodge structure of type $\{(-1,0),(0,-1)\}$ from the Hodge decomposition, so
\[
V\otimes_\QQ\CC = H_1(A(\CC),\CC)=V^{-1,0} \oplus V^{0,-1}
\]
with $V^{0,-1}=\bbar{V^{-1,0}}$.  Let $\mu\colon \GG_{m,\CC} \to \GL_{V\otimes_\QQ\CC}$ be the cocharacter such that $\mu(z)$ is the automorphism of $V\otimes_\QQ\CC$ which is multiplication by $z$ on $V^{-1,0}$ and the identity on $V^{0,-1}$ for each $z\in\CC^\times=\GG_{m}(\CC)$.   The \defi{Mumford--Tate group} of $A$ is the smallest algebraic subgroup $G_A$ of $\GL_V$, defined over $\QQ$, which contains $\mu(\GG_{m,\CC})$.  

The ring of endomorphisms $\End(A)$ of the abelian variety $A/\CC$ acts on $V$, which induces an injective homomorphism $\End(A)\otimes_\ZZ \QQ \hookrightarrow \End_{\QQ}(V)$.  Denote by $\End_{\QQ}(V)^{G_A}$ the subring of $\End_{\QQ}(V)$ consisting of those elements that commute with $G_A$.
  
\begin{lemma} \label{L:MT basics}
\begin{romanenum}
\item  \label{L:MT basics i}
The group $G_A$ is connected and reductive.
\item  \label{L:MT basics ii}
The image of $\End(A)\otimes_\ZZ \QQ \hookrightarrow \End_{\QQ}(V)$ is $\End_{\QQ}(V)^{G_A}$.
\end{romanenum}
\end{lemma}
\begin{proof}
This follows from Propositions~17.3.4 and 17.3.6 of \cite{MR2062673};  note that the Mumford--Tate group $G_A$ is generated by the \emph{Hodge group} $\operatorname{Hg}(A)$ and the group $\GG_m$ of homotheties.
\end{proof}

The abelian variety $A$ is isogenous to a product $\prod_{i=1}^n A_i^{m_i}$, where the $A_i$ are simple abelian varieties over $\CC$ that are pairwise nonisogenous and the $m_i$ are positive integers.   A fixed isogeny induces an isomorphism
\begin{align} \label{E:V decomposition}
V=\bigoplus_{i=1}^n V_i
\end{align}
of $\QQ$-vector spaces, where each $V_i:= H_1(A_i^{m_i}(\CC),\QQ)$ is a representation of $G_A$.

For each subset $I \subseteq \{1,\ldots, n\}$, define the subspace $V_I:=\bigoplus_{i\in I}V_i$ of $V$ and the abelian variety $A_I:=\prod_{i\in I}A_i^{m_i}$.  We can identify $H_1(A_I(\CC),\QQ)$ with $V_I$.  The group $G_A$ acts on $V_I$ since it acts on each $V_i$.   The representation $G_A \to \GL_{V_I}$ gives a dominant homomorphism $G_A\to G_{A_I}$ of linear algebraic groups.  The kernel of $G_A\to G_{A_I}$ is $(G_A)_{V_I}$ and hence
\begin{align} \label{E:codim key}
\dim G_{A_I}= \dim G_A - \dim (G_A)_{V_I}.
\end{align}

\begin{lemma} \label{L:GA isotypical}
The direct sum (\ref{E:V decomposition}) is the decomposition of the representation $V$ of $G_A$ into isotypical components.
\end{lemma}
\begin{proof}
Take any $i\in \{1,\ldots, n\}$ and set $I:=\{i\}$.   The subspace $V_I=V_i$ of $V$ is a representation of $G_{A}$ via the homomorphism $G_A \to G_{A_I}$.   We thus have
\[
\prod_{i=1}^n \End(A_i^{m_i}) \otimes_\ZZ\QQ = \prod_{i=1}^n \End_\QQ(V_i)^{G_A} \subseteq \End_\QQ(V)^{G_A} = \End(A) \otimes_\ZZ \QQ = \prod_{i=1}^n \End(A_i^{m_i}) \otimes_\ZZ \QQ,
\]
where the first two equalities follow from Lemma~\ref{L:MT basics}(\ref{L:MT basics ii}) and the last equality uses that the simple abelian varieties $A_i$ are pairwise nonisogenous.   Therefore, we have  $\End_\QQ(V)^{G_A}= \prod_{i=1}^n \End_\QQ(V_i)^{G_A}$ and each $\End_\QQ(V_i)^{G_A}$ is isomorphic to $M_{e_i}(D_i)$ for some integer $e_i\geq 1$ and division algebra $D_i$ (the ring $\End(A_i^{m_i}) \otimes_\ZZ \QQ$ is of this form since $A_i$ is simple).   The lemma is now a consequence of $V$ being a semisimple representation of $G_A$; this is true since $G_A$ is reductive by Lemma~\ref{L:MT basics}(\ref{L:MT basics i}).  
\end{proof}

\subsection{$\ell$-adic monodromy groups} \label{SS:monodromy groups}
Take any prime $\ell$.   For each integer $i \geq 1$, let $A[\ell^i]$ be the $\ell^i$-torsion subgroup of $A(\Kbar)$, where $\Kbar$ is a fixed algebraic closure of $K$.   The group $A[\ell^i]$ is a free $\ZZ/\ell^i\ZZ$-module of rank $2g$.  The \defi{$\ell$-adic Tate module} is 
\[
T_\ell(A):=\varprojlim_{i} A[\ell^i],
\] 
where the inverse limit is with respect to the multiplication by $\ell$ maps $A[\ell^{i+1}] \to A[\ell^i]$.   The Tate module $T_\ell(A)$ is a free $\ZZ_\ell$-module of rank $2g$.  Define $V_\ell(A):=T_\ell(A)\otimes_{\ZZ_\ell} \QQ_\ell$; it is a $\QQ_\ell$-vector space of dimension $2g$. 
We can identify $\GL_{V_\ell(A)}$ with the generic fiber of $\GL_{T_\ell(A)}$.

The Galois group $\Gal_K:=\Gal(\Kbar/K)$ acts on each $A[\ell^i]$ and respects the group structure.  This induces an action of $\Gal_K$ on $T_\ell(A)$ and $V_\ell(A)$.  The action of $\Gal_K$ on $V_\ell(A)$ respects the vector space structure and can thus be expressed in terms of a representation
\[
\rho_{A,\ell^\infty} \colon \Gal_K \to \Aut_{\Z_\ell}(T_\ell(A)) = \GL_{T_\ell(A)}(\Z_\ell) \subseteq \Aut_{\QQ_\ell}(V_\ell(A)) = \GL_{V_\ell(A)}(\QQ_\ell).
\]
The \defi{$\ell$-adic monodromy group} of $A$, which we denote by $G_{A,\ell}$, is the algebraic subgroup of $\GL_{V_\ell(A)}$ obtained by taking the Zariski closure of $\rho_{A,\ell^\infty}(\Gal_K)$.   The group $\rho_{A,\ell^\infty}(\Gal_K)$ is open in $G_{A,\ell}(\QQ_\ell)$, cf.~\cite{MR574307}.   Therefore, $G_{A,\ell}$ determines the group $\rho_{A,\ell^\infty}(\Gal_K)$ up to commensurability. 

We define $\calG_{A,\ell}$ to be the group subscheme of $\GL_{T_\ell(A)}$ obtained by taking the Zariski closure of $\rho_{A,\ell^\infty}(\Gal_K) \subseteq \GL_{T_\ell(A)}(\ZZ_\ell)$.     We can also describe $\calG_{A,\ell}$ as the Zariski closure of $G_{A,\ell}$ in $\GL_{T_\ell(A)}$.  The monodromy group $G_{A,\ell}$ is the generic fiber of the $\ZZ_\ell$-group scheme $\calG_{A,\ell}$.  

Denote by $G_{A,\ell}^\circ$ the \defi{neutral component} of $G_{A,\ell}$, i.e., the connected component of $G_{A,\ell}$ containing the identity element; it is an algebraic subgroup of $G_{A,\ell}$.  Denote by $K_{A,\ell}$ the finite extension of $K$ such that the kernel of the homomorphism 
\[
\Gal_K \xrightarrow{\rho_{A,\ell^\infty}} G_{A,\ell}(\QQ_\ell)\to G_{A,\ell}(\QQ_\ell)/G_{A,\ell}^\circ(\QQ_\ell)
\] 
is $\Gal(\Kbar/K_{A,\ell})$, where the second homomorphism is the obvious quotient map.   The following proposition was proved by Serre \cite{MR1730973}*{133}; see also \cite{MR1441234}. 

\begin{prop} \label{P:connected}
The extension $K_{A,\ell}/K$ is independent of $\ell$.   In particular, there is a finite extension $K'/K$ such that $G_{A_{K'},\ell}$ is connected for all $\ell$.
\end{prop}

We will make extensive use of the following fundamental results of Faltings.

\begin{prop}[Faltings] \label{P:Faltings}
\begin{romanenum}
\item \label{P:Faltings i}
The representation $V_\ell(A)$ of $\Gal_K$ is semisimple.
\item \label{P:Faltings ii}
The map $\End(A)\otimes_\ZZ \QQ_\ell \to \End_{\QQ_\ell[\Gal_K]}(V_\ell(A))$ is an isomorphism of $\QQ_\ell$-algebras.
\item \label{P:Faltings iii}
For any abelian variety $B$ over $K$, the map $\Hom(A,B)\otimes_\ZZ \QQ_\ell \to \Hom_{\QQ_\ell[\Gal_K]}(V_\ell(A),V_\ell(B))$ is an isomorphism of $\QQ_\ell$-vector spaces.   
\end{romanenum}
\end{prop}
\begin{proof}
These are proven in \S5 of \cite{MR861971}.
\end{proof}

\begin{prop} \label{P:reductive}
Assume that the groups $G_{A,\ell}$ are connected for all $\ell$.
\begin{romanenum}
\item \label{P:reductive i}
The algebraic group $G_{A,\ell}$ is reductive for all $\ell$.
\item \label{P:reductive ii}
 For $\ell\gg_A 1$, the $\ZZ_\ell$-group scheme $\calG_{A,\ell}$ is reductive.
\end{romanenum}
\end{prop}
\begin{proof}
Part (\ref{P:reductive i}) is a direct consequence of Proposition~\ref{P:Faltings}(\ref{P:Faltings i}).  Part (\ref{P:reductive ii}) is proved in \cite{MR1339927} though also see \cite{MR1944805}*{\S1.3}.
\end{proof}

\subsection{The Mumford--Tate conjecture} \label{SS:MT conjecture}

After fixing an embedding $\Kbar\subseteq \CC$, the \defi{comparison isomorphism} $V_\ell(A)\cong V \otimes_\QQ \QQ_\ell$ induces an isomorphism $\GL_{V_\ell(A)} \cong \GL_{V,\,\QQ_\ell}$.   The following conjecture says that $G_{A,\ell}^\circ$ and $(G_A)_{\QQ_\ell}$ are the same algebraic group when we use the comparison isomorphism as an identification, cf.~\cite[{\S3}]{MR0476753}.

\begin{conj}[Mumford--Tate conjecture for $A$]   \label{C:MT}
For each prime $\ell$, we have $G_{A,\ell}^\circ= (G_A)_{\QQ_\ell}$.
\end{conj}
 
The Mumford--Tate conjecture is still open, however significant progress has been made in showing that several general classes of abelian varieties satisfy the conjecture; we simply refer the reader to \cite[{\S1.4}]{MR2400251} for a partial list of references.  

\begin{prop} \label{P:partial MTC}
\begin{romanenum}
\item \label{P:partial MTC i}
For each prime $\ell$, we have $G_{A,\ell}^\circ\subseteq (G_A)_{\QQ_\ell}$.
\item \label{P:partial MTC ii}

The Mumford--Tate conjecture for $A$ holds if and only if the common rank of the groups $G_{A,\ell}^\circ$ equals the rank of $G_{A}$; in particular, the conjecture holds for one prime $\ell$ if and only if it holds for all $\ell$.  
\end{romanenum}
\end{prop}
\begin{proof}
For a proof of (\ref{P:partial MTC i}) see \cite[{I, Prop.~6.2}]{MR654325}.  Part (\ref{P:partial MTC ii}) follows from  \cite[{Theorem~4.3}]{MR1339927}.
\end{proof}

\subsection{Bounded index and independence}
We can identify $\calG_{A,\ell}(\ZZ_\ell)$ with a (compact) subgroup of $G_{A,\ell}(\QQ_\ell)$.  We thus have a Galois representation
\[
\rho_{A,\ell^\infty} \colon \Gal_K \to \calG_{A,\ell}(\ZZ_\ell) \subseteq G_{A,\ell}(\QQ_\ell).
\]
As noted in \S\ref{SS:monodromy groups}, the group $\rho_{A,\ell^\infty}(\Gal_K)$ is open in $G_{A,\ell}(\QQ_\ell)$.     So $\rho_{A,\ell^\infty}(\Gal_K)$ is open, and hence of finite index, in $\calG_{A,\ell}(\ZZ_\ell)$.  The following theorem says that this index can in fact be bounded independent of $\ell$.

\begin{thm} \label{T:finite index}
There is a constant $C$, depending only on $A$, such that $[\calG_{A,\ell}(\ZZ_\ell): \rho_{A,\ell^\infty}(\Gal_K)] \leq C$ for all primes $\ell$.  
\end{thm}
\begin{proof}
This is proven assuming the Mumford-Tate conjecture for $A$ in \cite[\S10]{MR3576113} and unconditionally in \cite[Theorem 1.2$(b)$ and \S6]{Zywinaopenimage}. 
\end{proof}

\begin{prop} \label{P:independence}
There is a finite extension $K'/K$ such that the representations $\{\rho_{A,\ell^\infty}|_{\Gal_{K'}}\}_\ell$ are \defi{independent}, i.e., we have
\[
\bigg(\prod_\ell \rho_{A,\ell^\infty}\bigg)(\Gal_{K'}) = \prod_\ell \rho_{A,\ell^\infty}(\Gal_{K'})
\]
in $\prod_\ell G_{A,\ell}(\QQ_\ell)$, where the products are over all primes $\ell$.
\end{prop}
\begin{proof}
This was proved by Serre \cite{MR3093502}; see also \cite{MR1730973}*{138}.
\end{proof}

\subsection{Finiteness after base extension}

Throughout this section, we assume that all the $\ell$-adic monodromy groups $G_{A,\ell}$ are connected.  The goal of this section is to prove the following finiteness result which can often serve as a substitute to assuming the Mumford-Tate conjecture for $A$.

For a prime $\ell$, let $\QQ_\ell^\un$ be the maximal unramified extension of $\QQ_\ell$ in some fixed algebraic closure $\Qbar_\ell$.   Denote by $\ZZ_\ell^\un$ the integral closure of $\ZZ_\ell$ in $\QQ_\ell^\un$; it is a DVR with uniformizer $\ell$.   The residue field of $\ZZ_\ell^\un$ is an algebraic closure $\FFbar_\ell$ of $\FF_\ell$.

\begin{prop} \label{P:redux finiteness}
There is a finite collection $\{\calG_i \}_{i\in I}$ of split reductive subgroups of $\GL_{2g,\ZZ}$ such that the following hold:
\begin{alphenum}
\item \label{P:redux finiteness a}
For $\ell\gg_A 1$, there is an $i\in I$ such that 
\[
(\calG_{A,\ell})_{\ZZ_\ell^\un}=(\calG_i)_{\ZZ_\ell^\un}
\]
with respect to an appropriate choice of basis of the free $\ZZ_\ell^\un$-module $T_\ell(A) \otimes_{\ZZ_\ell} \ZZ_\ell^\un$.
\item \label{P:redux finiteness b}
For each $i\in I$, there is a direct sum $\ZZ^{2g} = \oplus_{j=1}^s W_j$ of representations of $\calG_i$ so that if $F=\QQ$ or if $F=\FF_\ell$ with $\ell\gg_A 1$, then $F^{2g}=\oplus_{j=1}^s (W_j \otimes_{\ZZ} F)$ is the decomposition of the tautological representations of $(\calG_i)_F$ into isotypical components.
\end{alphenum}
\end{prop}

Before starting the proof of the proposition, we need a better understanding of $T_\ell(A)\otimes_{\ZZ_\ell} \ZZ_\ell^\un$ as a representation of $(\calG_{A,\ell})_{\ZZ_\ell^\un}$.

\begin{lemma} \label{L:integrally semisimple}
For $\ell\gg_A 1$, there is a direct sum $T_\ell(A)\otimes_{\ZZ_\ell} \ZZ_\ell^\un = \bigoplus_{j=1}^m M_j$ of representations of $(\calG_{A,\ell})_{\ZZ_\ell^\un}$ over $\ZZ_\ell^\un$ such that each $M_j \otimes_{\ZZ_\ell^\un} \Qbar_\ell$ is an irreducible representation of $(\calG_{A,\ell})_{\Qbar_\ell}$.
\end{lemma}
\begin{proof}
To ease notation, we set $\calG:=(\calG_{A,\ell})_{\ZZ_\ell^\un}
$.   The abelian variety $A$ is isogenous to a product $\prod_{i=1}^s A_i$ of simple abelian varieties over $K$.   A fixed isogeny $A\to \prod_{i=1}^s A_i$ defines an isomorphism $T_\ell(A)\cong  \oplus_{i=1}^s T_\ell(A_i)$ of $\ZZ_\ell[\Gal_K]$-modules for all sufficiently large $\ell$.  Moreover, this is an isomorphism of representations of $\calG_{A,\ell}$ where $\calG_{A,\ell}$ acts on $T_\ell(A_i)$ through a dominant morphism $\calG_{A,\ell}\to \calG_{A_i,\ell}$.   So by taking $\ell \gg_A 1$, we may assume without loss of generality that $A$ is simple.  Define the ring $E:=\End(A)$.   Since $A$ is simple, the ring $D:=E\otimes_\ZZ \QQ$ is a division algebra.  

Let $R_\ell$ be the $\ZZ_\ell^\un$-submodule of $\End_{\ZZ_\ell^\un}(T_\ell(A) \otimes_{\ZZ_\ell} \ZZ_\ell^\un)$ generated by $\rho_{A,\ell^\infty}(\Gal_K)$; it is a $\ZZ_\ell^\un$-algebra.   We claim that by taking $\ell\gg_A 1$, every finitely generated and torsion-free $R_\ell$-module is  projective.    Fix an embedding $\Kbar\subseteq \CC$.     The ring $E$ of endomorphisms acts faithfully on $\Lambda:=H_1(A(\CC),\ZZ)$ so we may view $E$ as a subring of $\End_\ZZ(\Lambda) \cong M_{2g}(\ZZ)$.    Let $R$ be the centralizer of $E$ in $\End_\ZZ(\Lambda)$.    Using the comparison isomorphism $T_\ell(A)=\Lambda\otimes_{\ZZ} \ZZ_\ell$, we may view $R\otimes_\ZZ \ZZ_\ell$ as a subalgebra of $\End_{\ZZ_\ell}(T_\ell(A))$.   Since the actions of $\Gal_K$ and $E$ on $T_\ell(A)$ commute, $\rho_{A,\ell^\infty}(\Gal_K)$ is a subset of $R\otimes_\ZZ \ZZ_\ell$.   After taking $\ell$ sufficiently large, $R\otimes_\ZZ \ZZ_\ell$ will be the $\ZZ_\ell$-submodule of $\End_{\ZZ_\ell}(T_\ell(A))$ generated by $\rho_{A,\ell^\infty}(\Gal_K)$, see the last remark in \cite{MR861971}.  So in $\End_{\ZZ_\ell^\un}(T_\ell(A)\otimes_{\ZZ_\ell}\ZZ_\ell^\un)$, we have $R_\ell=R\otimes_{\ZZ} \ZZ_\ell^\un$.       Since $D=E\otimes_\ZZ \QQ$ is a division algebra, $R\otimes_\ZZ \QQ$ is a simple algebra \cite[\href{https://stacks.math.columbia.edu/tag/074T}{Theorem 074T}]{stacks-project}.   Since $R$ is a $\ZZ$-order in $R \otimes_\ZZ \QQ$, we find that $R_\ell=R\otimes_\ZZ \ZZ_\ell^\un$ is a maximal $\ZZ_\ell^\un$-order in $R \otimes_\ZZ \QQ_\ell^\un$ after taking $\ell$ sufficiently large, cf.~\cite[\S23]{MR632548} for background on orders.  The claim is then a consequence of  Theorem~26.12 of \cite{MR632548}.  

Let $M$ be any $\ZZ_\ell^\un$-submodule of $T_\ell(A)\otimes_{\ZZ_\ell} \ZZ_\ell^\un$ that is a representation of $\calG$.  Since $\ZZ_\ell^\un$ is a DVR, $M$ is a free $\ZZ_\ell^\un$-module of finite rank.  We will now show that $M=\oplus_{j=1}^m M_j$ where the $M_j$ are representations of $\calG$ over $\ZZ_\ell^\un$ for which the $M_j \otimes_{\ZZ_\ell^\un} \QQ_\ell^\un$ are irreducible representations of $\calG_{\QQ_\ell^\un}$.  We may assume that $M$ is nonzero since the result is trivial otherwise.   Since $\calG_{\QQ_\ell^\un}$ is reductive, there is a direct sum $M \otimes_{\ZZ_\ell^\un} \QQ_\ell^\un= \oplus_{i=1}^t V_i$ of irreducible representations of $\calG_{\QQ_\ell^\un}$.   Define $M_1:= M \cap V_1$; it is a representation of $\calG$ over $\ZZ_\ell^\un$ and $M_1 \otimes_{\ZZ_\ell^\un} \QQ_\ell^\un = V_1$.    Note that the $\ZZ_\ell^\un$-module $M/M_1$ is torsion free.  Consider the short exact sequence 
\[
0 \to M_1 \to M \to M/M_1 \to 0
\]
of representations of $\calG$.   This can also be viewed as a short exact sequence of $R_\ell$-modules with $M/M_1$ being finitely generated and torsion-free.   The claim we proved above says that $M/M_1$ is a projective $R_\ell$-module.  So our short exact sequence splits and we obtain a direct sum $M=M_1 \oplus M'$ of $R_\ell$-modules and hence also of representations of $\calG=(\calG_{A,\ell})_{\ZZ_\ell^\un}$ over $\ZZ_\ell^\un$.  Now proceeding by induction on the rank over $\ZZ_\ell^\un$, we deduce the existence of the desired direct sum $M=\oplus_{j=1}^m M_j$.   

By considering the special case $M=T_\ell(A)\otimes_{\ZZ_\ell} \ZZ_\ell^\un$,  we obtain a direct sum $T_\ell(A)=\oplus_{j=1}^m M_j$, where  $M_j$ is a representation of $\calG$ over $\ZZ_\ell^\un$ such that $M_j \otimes_{\ZZ_\ell^\un} \QQ_\ell^\un$ is an irreducible representation of $\calG_{\QQ_\ell^\un}$.

For a fixed $1\leq j \leq m$, it remains to show that $M_j\otimes_{\ZZ_\ell^\un} \Qbar_\ell$ is an irreducible representation of  $\calG_{\Qbar_\ell}$.  The ring $D=E\otimes_\ZZ \QQ$ is a division algebra whose center $L$ is a number field.  Define the integers $d=[L:\QQ]$ and $n=[D:L]^{1/2}$.   Choose a finite extension $L_1$ of $L$ that splits $D$, i.e., $D \otimes_L L_1 \cong M_n(L_1)$. By choosing $\ell$ large enough, we may assume that $\ell$ is unramified in the number field $L_1$.   Since $\ell$ is unramified in $L$, there are exactly $d$ embeddings $L\hookrightarrow \QQ_\ell^\un$.   Therefore, we have isomorphisms
\[
D \otimes_\QQ \QQ_\ell^\un = D\otimes_L (L \otimes_\QQ \QQ_\ell^\un) = D \otimes_L (\prod_\sigma \QQ_\ell^\un) = \prod_\sigma (D \otimes_{L,\sigma} \QQ_\ell^\un),
\]
where the $\sigma$ vary over the embeddings $L\hookrightarrow \QQ_\ell^\un$.  Since $\ell$ is unramified in $L_1$, any embedding $\sigma\colon L\hookrightarrow \QQ_\ell^\un$ extends to an embedding of $L_1$. So we isomorphisms 
\[
D \otimes_{L,\sigma} \QQ_\ell^\un= (D \otimes_L L_1)\otimes_{L_1,\sigma} \QQ_\ell^\un\cong M_n(L_1)\otimes_{L_1,\sigma} \QQ_\ell^\un\cong M_n(\QQ_\ell^\un)
\] 
of $\QQ_\ell^\un$-algebras.   Therefore, $D \otimes_\QQ \QQ_\ell^\un\cong M_n(\QQ_\ell^\un)^d$.

From Proposition~\ref{P:Faltings}(\ref{P:Faltings ii}), the commutant of $\rho_{A,\ell^\infty}(\Gal_K)$ in $\End_{\QQ_\ell}(V_\ell(A))$ is isomorphic to $E\otimes_\ZZ \QQ_\ell = D \otimes_\QQ \QQ_\ell$.   Therefore, the commutant of $\calG_{\QQ_\ell^\un}$ in $\End_{\QQ_\ell^\un}(V_\ell(A)\otimes_{\QQ_\ell} \QQ_\ell^\un)$ is isomorphic to $D\otimes_\QQ \QQ_\ell^\un \cong M_n(\QQ_\ell^\un)^d$.   Since $\calG_{\QQ_\ell^\un}$ is reductive, we can read from this that $T_\ell(A)\otimes_{\ZZ_\ell} \QQ_\ell^\un = V_\ell(A)\otimes_{\QQ_\ell} \QQ_\ell^\un$ has $d$ isotypical components with each irreducible representation occuring with multiplicity $n$.  The same holds true for the representation $T_\ell(A)\otimes_{\ZZ_\ell} \Qbar_\ell$ since $D \otimes_\QQ \Qbar_\ell \cong M_n(\Qbar_\ell)^d$.  So for any irreducible representation $V\subseteq T_\ell(A)\otimes_{\ZZ_\ell} \QQ_\ell^\un$ of $\calG_{\QQ_\ell^\un}$, $V\otimes_{\QQ_\ell^\un} \Qbar_\ell$ is also an irreducible representation of $(G_{A,\ell})_{\Qbar_\ell}$.  The lemma is now immediate.
\end{proof}

\begin{proof}[Proof of Proposition~\ref{P:redux finiteness}]
By taking $\ell\gg_A 1$, we may assume by Proposition~\ref{P:reductive}(\ref{P:reductive ii}) that the $\ZZ_\ell$-group scheme $\calG_{A,\ell} \subseteq \GL_{T_\ell(A)}$ is reductive.     By Lemma~\ref{L:integrally semisimple} and taking $\ell \gg_A 1$, we may assume that there is a direct sum
\[
T_\ell(A) \otimes_{\ZZ_\ell} \ZZ_\ell^\un = \oplus_{j=1}^m M_j 
\]
 of representations of $(\calG_{A,\ell})_{\ZZ_\ell^\un}$ over $\ZZ_\ell^\un$ so that each representation $M_j$ is geometrically irreducible.  

We claim that the reductive group $(G_{A,\ell})_{\QQ_\ell^\un}$ is split for all $\ell \gg_A 1$.    There is a torus $T$ defined over $\QQ$ such that $T_{\QQ_\ell}$ is isomorphic to a maximal torus of $G_{A,\ell}$ for all $\ell$, cf.~\cite[\S3b]{MR1156568}.   The torus $T_{\QQ_\ell^\un}$ is split for all large enough $\ell$ (for example, it holds for all $\ell$ that are unramified in some number field over which $T$ splits).  This proves the claim and we may now assume that $(G_{A,\ell})_{\QQ_\ell^\un}$ is split.

Let $\calC_\ell$ and $\calS_\ell$ be the central torus and the derived subgroup, respectively, of the reductive $\ZZ_\ell$-group scheme $\calG_{A,\ell}$.   The groups $(\calC_\ell)_{\ZZ_\ell^\un}$ and $(\calS_\ell)_{\ZZ_\ell^\un}$ are both split since $(\calG_{A,\ell})_{\ZZ_\ell^\un}$ is split.  There is a split semisimple group scheme $\calS$ over $\ZZ$ for which there is an isomorphism 
\[
f_\ell\colon\calS_{\ZZ_\ell^\un} \xrightarrow{\sim} (\calS_\ell)_{\ZZ_\ell^\un}, 
\]
cf.~\cite[Exposé XXV, Théorème 1.1 for the existence of $\mathcal{S}$ and Exposé XXIII, Corollaire 5.1 for the isomorphism]{SGA3-III}. We can construct $\calS$ directly from the root datum of $(\calS_\ell)_{\ZZ_\ell^\un}$.  There are only finitely many possibilities for this root datum, up to isomorphism, since our semisimple groups have dimension bounded independent of $\ell$.   So we may assume that only finitely many $\ZZ$-group schemes $\calS$ arise as we vary $\ell$.

Let $C$ be the central torus of the Mumford--Tate group $G_A \subseteq \GL_{V_A}$, where $V_A=H_1(A(\CC),\QQ)$.  Fix a minimal field extension $L$ of $\QQ$ for which the torus $C_L$ is split and let $X(C_L)$ be the group of characters $C_L \to \GG_{m,L}$.  Let $\calC$ be the split torus over $\ZZ$ isomorphic to $C_L$ over $L$ for which we can identify its character group $X(\calC)$ with $X(C_L)$.   Let $\Omega \subseteq X(C_L)=X(\calC)$ be the (finite) set of weights of the action of the torus $C_L$ on $V_A \otimes_\QQ L$. 

We have an inclusion of groups $G_{A,\ell} \subseteq (G_A)_{\QQ_\ell}$ by Proposition~\ref{P:partial MTC}(\ref{P:partial MTC i}).  By \cite[Corollary 2.11]{UllmoYafaev}, their central tori agree, i.e., $(\calC_{\ell})_{\QQ_\ell}=C_{\QQ_\ell}$.   By taking $\ell\gg_A 1$, we may assume that $\ell$ is unramified in $L$ and hence there is an embedding $\sigma\colon L \hookrightarrow \QQ_\ell^\un$.   Using the embedding $\sigma$, we obtain an isomorphism between $(\calC_{\ell})_{\QQ_\ell^\un}=C_{\QQ_\ell^\un} = (C_L)_{\QQ_\ell^\un}
$ and $(\calC_L)_{\QQ_\ell^\un}=\calC_{\QQ_\ell^\un}$.  Since the tori $\calC_{\ell}$ and $\calC$ are split over $\ZZ_\ell^\un$, this gives rise to an isomorphism
\[
f_\ell'\colon \calC_{\ZZ_\ell^\un} \xrightarrow{\sim} (\calC_\ell)_{\ZZ_\ell^\un}.
\]

Using our isomorphisms $f_\ell$ and $f_\ell'$, we may view $T_\ell(A)\otimes_{\ZZ_\ell} \ZZ_\ell^\un$ and each $M_j$ as $\ZZ_\ell^\un$-representations of $\calS_{\ZZ_\ell^\un}$ and $\calC_{\ZZ_\ell^\un}$.  Note that these actions of $\calS_{\ZZ_\ell^\un}$ and $\calC_{\ZZ_\ell^\un}$ commute with each other.  \\

Now take any $1\leq j \leq m$.   Since $M_j \otimes_{\ZZ_\ell^\un} \Qbar_\ell$ is an irreducible representation of $(\calG_{A,\ell})_{\Qbar_\ell}$, it must also be an irreducible representation of $(\calS_{\ell})_{\Qbar_\ell}$ with $(\calC_{\ell})_{\Qbar_\ell}$ acting via a single character $\alpha_j \in X(\calC)$.  The character $\alpha_j$ lies in $\Omega$ since with respect to the comparison isomorphism $T_\ell(A)\otimes_{\ZZ_\ell} \Qbar_\ell  = V_A \otimes_\QQ \Qbar_\ell$ it corresponds to a weight of $C_{\Qbar_\ell}$.

Let $\mu_j$ be the dominant weight of $\calS$ that arises as the highest weight of the irreducible representation $M_j \otimes_{\ZZ_\ell^\un} \QQ_\ell$ of $\calS_{\QQ_\ell}$ (this lies in the character group of a fixed split maximal torus of $\calS$ and is with respect to a choice of a positive system of roots).   The minuscule weight conjecture, proved by Pink \cite[Corollary~5.11]{PinkMumfordTate}, says that all the weights of $T_\ell(A)\otimes_{\ZZ_\ell} \Qbar_\ell$ as a representation of $(S_\ell)_{\Qbar_\ell}$ are minuscule.   Since a semisimple group has only finitely many minuscule weights, we find that for our given $\calS$, there are only finitely many possibilities for $\mu_j$.

Following Jantzen \cite{Jantzen}*{II.8.3}, there is a $\ZZ$-representation $V(\mu_j)$ of $\calS$ such that for any field $F$ of characteristic $0$, $V(\mu_j)\otimes_\ZZ F$ is the unique irreducible representation of $\calS_F$, up to isomorphism, with highest weight $\mu_j$.   In particular, $M_j \otimes_{\ZZ_\ell^\un} \QQ_\ell^\un$ and $V(\mu_j)\otimes_{\ZZ} \QQ_\ell^\un$ are isomorphic representations of $\calS_{\QQ_\ell^\un}$.

Since $\mu_j$ is a minuscule weight, all the weights of the action of $\mathcal{S}_{\FFbar_\ell}$ on $V(\mu_j) \otimes_{\mathbb{Z}} \FFbar_\ell$ lie in a single orbit for the Weyl group, and therefore $V(\mu_j) \otimes_{\mathbb{Z}} \FFbar_\ell$ is an irreducible representation of $\mathcal{S}_{\FFbar_\ell}$, see also \cite{GaribaldiGuralnickNakano} (and in particular \cite[p.~73, Generalization of Theorem 1.1 to split reductive groups]{GaribaldiGuralnickNakano}).

We now claim that $M_j$ and $V(\mu_j)\otimes_{\ZZ} \ZZ_\ell^\un$ are isomorphic $\ZZ_\ell^\un$-representations of $\calS_{\ZZ_\ell^\un}$.   We know there is an isomorphism $M_j \otimes_{\ZZ_\ell^\un} \QQ_\ell^\un=V(\mu_j)\otimes_{\ZZ} \QQ_\ell^\un$ of $\calS_{\QQ_\ell^\un}$ representations that we will view as an identification; this allows us to compare $M_j$ and $\calV:=V(\mu_j)\otimes_{\ZZ}\ZZ_\ell^\un$ which are $\ZZ_\ell^\un$-modules of the same rank.  Since $\ZZ_\ell^\un$ is a DVR with uniformizer $\ell$, we can scale our isomorphism by an appropriate power of $\ell$ so that $M_j \subseteq \calV$ and $M_j \not\subseteq \ell \calV$.  To prove that $\calV= M_j$, it suffices to show that the quotient map $\varphi\colon M_j \to \calV/\ell \calV$ is surjective.   Define $W:=\varphi(M_j)$.  Since $\FFbar_\ell$ is the residue field of $\ZZ_\ell^\un$, $\calV/\ell \calV$ is an $\FFbar_\ell$-vector space.  Moreover, $\calV/\ell \calV$ is a representation of $\calS_{\FFbar_\ell}$ that is isomorphic to $\calV \otimes_{\ZZ_\ell^\un} \FFbar_\ell$.    By our previous claim, $\calV/\ell \calV$ is an irreducible representation of $\calS_{\FFbar_\ell}$.  Since $W \subseteq \calV/\ell \calV$ is also a representation of $\calS_{\FFbar_\ell}$, we have $W=0$ or $W=\calV/\ell \calV$.   We have $W\neq 0$ since $M_j \not\subseteq \ell \calV$ and thus $\varphi$ is surjective. 

Define $N_j:=V(\mu_j)$; it is a free $\ZZ$-module of finite rank that has an action of $\calS$.   We let $\calC$ act on $N_j$ via the character $\alpha_j$.     Now define the $\ZZ$-module $N:=\oplus_{j=1}^m N_j$; it is a representation over $\ZZ$ of $\calS$ and $\calC$.       We have shown above that there is an isomorphism $N_j \otimes_\ZZ\ZZ_\ell^\un  = V(\mu_j)\otimes_\ZZ \ZZ_\ell^\un \xrightarrow{\sim} M_j$ of $\ZZ_\ell^\un$-modules that respects the action of $\calS_{\ZZ_\ell^\un}$; it also respects the $\calC_{\ZZ_\ell^\un}$ -action since the torus acts on both by the character $\alpha_j$.  Combining these together, we obtain an isomorphism 
\[
\psi_\ell\colon N\otimes_\ZZ \ZZ_\ell^\un  \xrightarrow{\sim} \oplus_{j=1}^m M_j = T_\ell(A) \otimes_{\ZZ_\ell} \ZZ_\ell^\un
\]
of $\ZZ_\ell^\un$-modules that respects the actions of $\calS_{\ZZ_\ell^\un}$ and $\calC_{\ZZ_\ell^\un}$.  Using that $\calG_{A,\ell}$ acts faithfully on $T_\ell(A)$, we find that $\calS$ and $\calC$ act faithfully on $N$ and hence we may view them as subgroups of $\GL_N$.   Since $\calS$ and $\calC$ commute, the subgroup of $\GL_N$ generated by $\calS$ and $\calC$ is a reductive $\ZZ$-group scheme $\calG$.    Taking $f_\ell$ and $f_\ell'$ into account, the isomorphism
\begin{align*}
\GL_{N\otimes_\ZZ \ZZ_\ell^\un} \xrightarrow{\sim} \GL_{T_\ell(A)\otimes_{\ZZ_\ell} \ZZ_\ell^\un}
\end{align*}
of $\ZZ_\ell^\un$-groups coming from $\psi_\ell$ gives isomorphisms $\calS_{\ZZ_\ell^\un}\xrightarrow{\sim} (\calS_\ell)_{\ZZ_\ell^\un}$ and $\calC_{\ZZ_\ell^\un}\xrightarrow{\sim} (\calC_\ell)_{\ZZ_\ell^\un}$, and hence also an isomorphism $\calG_{\ZZ_\ell^\un}\xrightarrow{\sim} (\calG_{A,\ell})_{\ZZ_\ell^\un}$.   After choosing a basis of $N$ over $\ZZ$, we can view $\calG$ as a subgroup of $\GL_{2g,\ZZ}$.  To complete our proof of (\ref{P:redux finiteness a}), we need to explain why only finitely many such groups $\calG$ need arise as we vary $\ell$.  In our construction, $\calC$ is independent of $\ell$ and only finitely many $\calS$, up to isomorphism, will occur.  There are only finitely many representations $N_j$ of $\calS$ and $\calG$ that will occur since only finitely many dominant weights $\mu_j$ show up and there are only finitely many $\alpha_j \in \Omega$.   This implies that only a finite number of representations $N$ of $\calS$ and $\calC$ will occur since the rank of $N$ must be $2g$.   Thus our construction gives rise to only finitely many possible groups $\calG \subseteq \GL_N$ for $\ell\gg_A 1$.  

We now prove (\ref{P:redux finiteness b}).  We keep notation as in the above construction of $\calG \subseteq \GL_N$.    Each representation $N_j$ of $\calG$ was constructed using the dominant weight $\mu_j$ of $\calS$  and the character $\alpha_j$ of $\calC$.  Let $F$ be the field $\QQ$ or $\FF_\ell$.     From the irreducible representation $N_j\otimes_\ZZ F$ of $\calG_F$, we can recover $\mu_j$ and $\alpha_j$  since we know the central torus acts via a character and the action of the derived subgroup is given by its highest weight (for $F=\FF_\ell$, the irreducibility was shown above for $\ell\gg_A 1$ when we proved that $V(\mu_j) \otimes_\ZZ \FFbar_\ell$ is an irreducible representation of $\calS_{\FFbar_\ell}$).   Therefore, the isomorphism class of the representation $N_j\otimes_\ZZ F$ of $\calG_F$ is determined by $(\mu_j,\alpha_j)$.  So let $\mathcal{I}$ be the set of pairs $(\mu_j,\alpha_j)$ corresponding to the $(\calS \times \calC)$-modules $N_j$, with $1\leq j \leq m$. For each $(\mu,\alpha)\in \mathcal{I}$, let $W_{(\mu,\alpha)}$ be the direct sum of the $N_j$ with $1\leq j \leq m$ for which $(\mu_j,\alpha_j)=(\mu,\alpha)$.   Therefore, $N=\oplus_{(\mu,\alpha) \in \mathcal{I}} W_{(\mu,\alpha)}$ and $N\otimes_\ZZ F = \oplus_{(\mu,\alpha) \in \mathcal{I}} (W_{(\mu,\alpha)}\otimes_\ZZ F)$ is the decomposition into isotypical components as a representation of $\calG_F$. 
\end{proof}

\subsection{Points modulo $\ell$}
Throughout this section, we will assume that all the $\ell$-adic monodromy groups $G_{A,\ell}$ are connected.  Choosing a $\ZZ_\ell$-basis for $T_\ell(A)$, we can identify $\calG_{A,\ell}$ with an algebraic subgroup of $\GL_{2g,\ZZ_\ell}$ and hence identify the special fiber $G:=(\calG_{A,\ell})_{\FF_\ell}$ with an algebraic subgroup of $\GL_{2g,\FF_\ell}$.  For each subspace $W \subseteq \FF_\ell^{2g}$, let $G_W$ be the algebraic subgroup of $G$ as defined in \S\ref{SS:notation}.  The goal of this section is to give bounds on the cardinality of $G_W(\FF_\ell)$ that do not depend on $W$.

For $W\subseteq \FF_\ell^{2g}$, let $m_W$ be the number of irreducible components of $(G_W)_{\FFbar_\ell}$.  Since $(G_W)_{\FFbar_\ell}$ is an algebraic group, all its irreducible components are disjoint and have the same dimension.  Our next lemma bounds $m_W$ uniformly in $\ell$ and $W$; see also \cite[Lemma 6]{MR3649016} for a closely related result. 

\begin{lemma} \label{L:mW bound}
We have $m_W \ll_A 1$ for all primes $\ell$ and subspaces $W\subseteq \FF_\ell^{2g}$.
\end{lemma}
\begin{proof}
 Fix a prime $\ell$ and a subspace $W\subseteq \FF_\ell^{2g}$.   In our proof, we can always take $\ell$ to be sufficiently large since this only excludes a finite number of groups $G_W$ from consideration.

Set $n:=4g^2+1$. 	We can view $\GL_{2g,\FF_\ell}$ as a closed subvariety of $\AA_{\FF_\ell}^{n}$ by identifying a matrix with its $(2g)^2$ entries and its determinant.  Let $\iota$ be the embedding given by the inclusions
\[
\GL_{2g,\FF_\ell} \subset \A^{n}_{\FF_\ell} \subset  \P^{n}_{\FF_\ell}.
\]	
By definition, $G_W = G \cap (\GL_{2g,\FF_\ell})_W$ with $(\GL_{2g,\FF_\ell})_W$ determined by homogenenous linear equations in the entries of the matrices in $\GL_{2g,\FF_\ell}$.  So there is a linear subspace $L$ of $\PP_{\FF_\ell}^n$ such that $\iota(G) \cap L = \iota(G_W)$.  We can choose a linear subspace $L'$ in $\PP^n_{\FFbar_\ell}$ of codimension $\dim G_W$ that intersects every irreducible component of  $\iota(G_W)_{\FFbar_\ell}$ and $\iota(G_W)_{\FFbar_\ell} \cap L'$ is finite.   Let $X$ be the Zariski closure of $\iota(G)$ in $\PP^n_{\F_\ell}$.    We find that $X_{\FFbar_\ell} \cap L_{\FFbar_\ell} \cap L'$ is zero dimensional and that the number of its $\FFbar_\ell$-points gives an upper bound  for $m_W$.   By a suitable version of  B\'ezout's theorem (for example, \cite[Example 8.4.6]{Fulton}), we deduce that
\[
m_W \leq \deg(X_{\FFbar_\ell}) \cdot  \deg(L_{\FFbar_\ell}) \cdot \deg(L').
\]
Therefore, $m_W \leq \deg(X_{\FFbar_\ell})$ since linear spaces have degree $1$. See \cite[\S8.4]{Fulton} for background on B\'ezout's theorem and the degree of a closed subvariety of projective space.  

 Proposition~\ref{P:redux finiteness} implies that there is a finite collection $\{\calG_i\}_{i\in I}$ of reductive subgroups of $\GL_{2g,\ZZ}$ such that for all sufficiently large $\ell$, $G_{\FFbar_\ell}$ and $(\calG_i)_{\FFbar_\ell}$ are conjugate in $\GL_{2g,\FFbar_\ell}$ for some $i\in I$.  Using that $I$ is finite, this implies that the subvariety $X_{\FFbar_\ell}$  of $\PP^n_{\FFbar_\ell}$ can be defined by a finite set $S$ of homogenous polynomials with $|S|$ and the degree of the $f\in S$ bounded in terms of a constant depending only on $A$.   By B\'ezout's theorem (for example, \cite[Example 8.4.6]{Fulton}), we deduce that $\deg(X_{\FFbar_\ell}) \ll_A 1$.  Therefore, $m_W \ll_A 1$.  
\end{proof}

\begin{prop} \label{P:HW count}
Define $G:=(\calG_{A,\ell})_{\FF_\ell}$.  We have $|G_W(\FF_\ell)| \asymp_A \ell^{\dim G_W}$ for all subspaces $W\subseteq \FF_\ell^{2g}$.
\end{prop}
\begin{proof}
We have inequalities $|G_W^\circ(\FF_\ell)| \leq |G_W(\FF_\ell)| \leq m_W \cdot |G_W^\circ(\FF_\ell)|$.  Since we have $m_W \ll_A 1$ by Lemma~\ref{L:mW bound}, it suffices to show that $|G_W^\circ(\FF_\ell)| \asymp_A \ell^{\dim G_W}$.   So the proposition is a consequence of \cite[Lemma 3.5]{MR880952} which shows that $(\ell-1)^{\dim G_W} \leq |G_W^\circ (\F_\ell)| \leq  (\ell+1)^{\dim G_W}$.
\end{proof}

\section{Codimension bounds} \label{S:codimension bounds new}

\subsection{Lie algebras} \label{SS:codimension bounds lie algebras}
Let $V$ be a nonzero finite-dimensional vector space over a field $F$.    Let $\gl(V)$ be the Lie algebra consisting of the $F$-linear endomorphisms of $V$ with the commutator serving as the Lie bracket.

\begin{definition}\label{D:alpha gW}
    Fix a Lie subalgebra $\gG$ of $\gl(V)$.   For each subspace $W$ of $V$, let $\gG_W$ be the subspace of $\gG$ consisting of $B\in \gG$ such that $Bw=0$ for all $w\in W$.  Observe that $\gG_W$ is a Lie subalgebra of $\gG$.    For each nonzero subspace $W$ of $V$, we define the nonnegative rational number 
\[
\alpha(\gG,W) := \frac{\dim \gG - \dim \gG_W}{\dim W}.
\]
There are only finitely many possibilities for the numerator and denominator of $\alpha(\gG,W)$, so we can define 
\[
\alpha(\gG):= \min_{W\neq 0} \alpha(\gG,W),
\] 
where the minimum is over all nonzero subspaces $W$ of $V$.   Note that the notation $\alpha(\gG)$ suppresses the dependence on the ambient algebra $\gl(V)$.  We define $\alpha(\gG,0)=0$. 
\end{definition}

The values of $\alpha(\gG,W)$ satisfy the main axiom of \enquote{slope theory}, which is given in the following lemma.

\begin{lemma}
\label{L: lemma slope}
 For any nonzero subspaces $W$ and $W'$ of $V$, 
 \[
  \alpha(\gG,W+W') \dim(W+W') + \alpha(\gG,W \cap W') \dim(W \cap W') \leq \alpha(\gG, W) \dim W + \alpha(\gG, W') \dim W'.
 \]
\end{lemma}

\begin{proof}
 The linear map $\gG_W \to \gG_{W\cap W'}/\gG_{W'}$ has kernel $\gG_W\cap \gG_{W'} = \gG_{W+W'}$, so
\[
 \dim (\gG_W/\gG_{W+W'}) \leq \dim (\gG_{W\cap W'} /\gG_{W'}).
\]
Equivalently, 
\[
\alpha(\gG,W+W')\, \dim(W+W') - \alpha(\gG,W)\,  \dim W \leq \alpha(\gG,W')\cdot \dim W' - \alpha(\gG,W\cap W')\, \dim (W\cap W'),
\]  
which gives the desired inequality.
\end{proof}

\begin{lemma} \label{L:key lemma new}
\begin{romanenum}
\item \label{L:key lemma new i}
There exists a unique maximal subspace $U$ of $V$ satisfying $\alpha(\gG,U)= \alpha(\gG)$.
\item \label{L:key lemma new ii}
If $\alpha(\gG,W)=\alpha(\gG)$ and $\alpha(\gG,W')=\alpha(\gG)$ for nonzero subspaces $W$ and $W'$ of $V$, then $\alpha(\gG,W+W')=\alpha(\gG)$.

\end{romanenum}
\end{lemma}
\begin{proof}
We first prove (\ref{L:key lemma new ii}).   By Lemma~\ref{L: lemma slope} and our assumptions, we have 
 \[
  \alpha(\gG,W+W') \dim(W+W') + \alpha(\gG,W \cap W') \dim(W \cap W') \leq \alpha(\gG) (\dim W + \dim W').
 \]
Since $\alpha(\gG) \leq \alpha(\gG,W\cap W')$ when $W\cap W'\neq 0$,  we have
\[
\alpha(\gG,W+W') \dim (W+W') \leq \alpha(\gG) (\dim W + \dim W' - \dim W\cap W' ) = \alpha(\gG) \dim (W+W')
\]
and hence $\alpha(\gG,W+W') \leq \alpha(\gG)$. Note that the same inequality also holds if $W \cap W' = 0$, since in this case the term corresponding to $W \cap W'$ in Lemma \ref{L: lemma slope} vanishes. By the minimality of $\alpha(\gG)$, this implies that $\alpha(\gG,W+W')=\alpha(\gG)$.  This completes the proof of (\ref{L:key lemma new ii}).

We have $\alpha(\gG,W)=\alpha(\gG)$ for some nonzero subspace $W$ of $V$ by the definition of $\alpha(\gG)$.   By taking $U$ to be the subspace generated by all subspaces $W$ of $V$ with $\alpha(\gG,W)=\alpha(\gG)$, we have $\alpha(\gG,U)=\alpha(\gG)$ by part (\ref{L:key lemma new ii}).   Part (\ref{L:key lemma new i}) is now clear.
\end{proof}

Fix a finite Galois extension $L$ of $F$.  Then $\gG\otimes_F L$ is a Lie subalgebra of $\gl(V\otimes_F L)$ and we can define $\alpha(\gG\otimes_F L)$ as before.  By Lemma~\ref{L:key lemma new}(\ref{L:key lemma new i}), there are  unique subspaces $U\subseteq V$ and $U' \subseteq V\otimes_F L$ that are maximal amongst those subspaces that satisfy $\alpha(\gG,U)=\alpha(\gG)$ and $\alpha(\gG\otimes_F L,U')=\alpha(\gG\otimes_F L)$, respectively.

\begin{lemma}\label{L:U functoriality}
With notation as above, the inclusion $U \subseteq V$ induces an isomorphism $U\otimes_F L \xrightarrow{\sim} U'$.  In particular, we have $\alpha(\gG\otimes_F L )=\alpha(\gG)$.
\end{lemma}
\begin{proof}
We have $\alpha(\gG)=\alpha(\gG, U) = \alpha(\gG\otimes_F L, U\otimes_F L)$ and hence $\alpha(\gG \otimes_F L)\leq \alpha(\gG)$.

Take any $\sigma\in\Gal(L/F)$.  We have $\sigma( (\gG\otimes_F L)_{U'} ) = (\gG\otimes_F L)_{\sigma(U')}$.   Therefore, $(\gG\otimes_F L)_{U'}$ and $(\gG\otimes_F L)_{\sigma(U')}$ have the same dimension over $F$ and hence also $L$.  Therefore, $\alpha(\gG\otimes_F L, \sigma(U'))= \alpha(\gG\otimes_F L,U')=\alpha(\gG\otimes_F L)$.   By the maximality of $U'$, we have $\sigma(U')=U'$ for all $\sigma\in \Gal(L/K)$.  By Galois descent for vectors spaces, there is a subspace $U_0$ of $V$ such that the inclusion $U_0 \subseteq V$ induces an isomorphism $U_0 \otimes_F L \xrightarrow{\sim} U'$.   Therefore, $\alpha(\gG,U_0) = \alpha(\gG\otimes_F L, U')=\alpha(\gG\otimes_F L)$.  Since  $\alpha(\gG \otimes_F L)\leq \alpha(\gG) $, this proves that $\alpha(\gG,U_0)=\alpha(\gG)$ and hence $U_0 \subseteq U$ by the maximality of $U$.  In particular, $\alpha(\gG)=\alpha(\gG\otimes_F L)$.

We have $U_0 \otimes_F L = U \otimes_F L$ since otherwise $U\otimes_F L$ would be a subspace of $V\otimes_F L$ strictly larger than $U'$ satisfying $\alpha(\gG\otimes_F L, U\otimes_F L)=\alpha(\gG\otimes_F L)$ which would contradict the maximality of $U'$.   Therefore, $U_0=U$ and the lemma follows.
\end{proof}

\subsection{Reductive groups} \label{SS:codimension bounds reductive}

Let $V$ be a nonzero finite dimensional vector space over a perfect field $F$.  Consider a reductive group $G\subseteq \GL_V$.   Assume that there exists a decomposition 
\[
V=\bigoplus_{i=1}^n V_i
\]
of the representation $V$ of $G$ into isotypic components, i.e., $V_i\neq 0$ is a $G$-invariant subspace of $V$ that is isomorphic to $M_i^{n_i}$ for an irreducible representation $M_i$ of $G$ and the representations $M_i$ are pairwise nonisomorphic.  Such a decomposition will always exist when $F$ has  characteristic $0$.   

For each subspace $W$ of $V$, let $G_W$ be the algebraic subgroup of $G$ that fixes $W$ pointwise. Define 
\[
\alpha(G):= \min_{W\neq 0} \frac{\dim G - \dim G_W}{\dim W},
\] 
where the minimum is over all nonzero subspaces $W\subseteq V$.   We obviously have $\gamma_G = 1/\alpha(G)$ (the former defined in \S\ref{SS:notation}).

Let $\gG \subseteq \gl(V)$ be the Lie algebra of $G$.   For each subspace $W\subseteq V$, the Lie algebra of $G_W$ agrees with the Lie algebra $\gG_W$ (see Definition \ref{D:alpha gW}); for a proof of the analogous statement in the case of stabilizers, see \cite{Milne}*{Proposition 10.31}.

   The following shows that under certain conditions, $\alpha(G)$ agrees with $\alpha(\gG)$ and can be computed using only a finite number of special subspaces $W$.   For each subset $I \subseteq \{1,\ldots, n\}$, define $V_I:= \oplus_{i\in I} V_i$.

\begin{prop}
\label{P:newer approach}
Assume that the algebraic group $G_{V_I}$ is smooth for all nonempty subsets $I \subseteq \{1,\ldots, n\}$; this always holds when $F$ has characteristic $0$.
\begin{romanenum}
\item \label{P:newer approach i}
We have $\alpha(G)=\alpha(\gG)$.
\item \label{P:newer approach ii}
We have  
\[
\alpha(G)= \min_{\emptyset \neq I \subseteq \{1,\ldots,n\}} \frac{\dim G - \dim G_{V_I}}{\dim V_I}.
\]
\item  \label{P:newer approach iii}
We have $\alpha(G_L)=\alpha(G)$ for all field extensions $L/F$, where we are using the embedding $G_{L} \subseteq \GL_{V\otimes_F L}$. 

\end{romanenum}
\end{prop}
\begin{proof}
First note that the groups $G_{V_I}$ are always smooth when $F$ has characteristic $0$ since every affine algebraic group over $F$ is smooth by Cartier's theorem \cite[Theorem 3.23]{MilneAlgebraicGroups}.

For any subspace $W\subseteq V$, we have $\dim G_W \leq \dim \gG_W$  since the tangent space of a variety at a point has dimension at least the dimension of the variety.  We have $\dim G= \dim \gG$ since $G$ is reductive and hence smooth.   From the definitions of $\alpha(G)$ and $\alpha(\gG)$, we deduce that $\alpha(\gG) \leq \alpha(G)$.   We need to prove the other inequality.

Let $U$ be the maximal subspace of $V$ such that $\alpha(\gG, U) = \alpha(\gG)$, see Lemma~\ref{L:key lemma new}(\ref{L:key lemma new i}).   We claim that $U$ is a representation of $G$.   Since $F$ is perfect, to prove the claim it suffices to show that $U\otimes_F L$ is stable under the action of $G(L)$ on $V\otimes_F L$ for all finite Galois extensions $L/F$.    Using that $\gG\otimes_F L$ is the Lie algebra of $G_L \subseteq \GL_{V\otimes_FL}$ and Lemma~\ref{L:U functoriality}, we need only consider the case $L=F$, i.e., we need only show that $U$ is stable under the action of $G(F)$.  Take any $h\in G(F)$.   Using that $h\gG h^{-1}=\gG$, we find that $\gG_{h(U)} = h  \gG_U h^{-1}$ and hence $\alpha(\gG,h(U))=\alpha(\gG,U)=\alpha(\gG)$.  By the maximality of $U$, we have $h(U)=U$ and the claim follows.
 
 Let $I$ be the subset of $i\in \{1,\ldots, n\}$ for which $U\cap V_i \neq 0$.  The set $I$ is nonempty and we have $U=\bigoplus_{i\in I} U\cap V_i$. Note that $U\cap V_i$ is a sum of copies of the irreducible representation of which $V_i$ is a multiple.
 We have 
 \[
 \gG_U=  \bigcap_{i \in I} \gG_{U \cap V_i} = \bigcap_{i \in I} \gG_{V_i} =\gG_{V_I}
 \]
 and hence $\alpha(\gG,V_I) = (\dim U)/(\dim V_I) \alpha(\gG,U) \leq \alpha(\gG)$.  Therefore, $\alpha(\gG,V_I)=\alpha(\gG)$ and $V_I=U$ by the maximality of $U$.   Since $G_{V_I}$ is smooth by assumption, we have $\dim G_{V_I} = \dim \gG_{V_I}$ and hence
 \begin{align} \label{E:alpha comp}
 \alpha(\gG)= \alpha(\gG,V_I) = \frac{\dim G - \dim G_{V_I}}{\dim V_I}.
 \end{align}
 Therefore, $\alpha(\gG) \geq \alpha(G)$.   We have already proved the other inequality so $\alpha(\gG)=\alpha(G)$.   We have thus proved (\ref{P:newer approach i}), and (\ref{P:newer approach ii}) then follows from (\ref{E:alpha comp}).

It remains to prove (\ref{P:newer approach iii}).  Take any field extension $L/F$.  For each $1\leq i \leq n$, define the $L$-vector space $V_i':=V_i \otimes_{F} L$.   Then  $\bigoplus_{i=1}^n V_i'$ is the decomposition
of the representation $V\otimes_F L$ of $G_{L}$ into isotypic components.  Take any nonempty subset $I \subseteq \{1,\ldots, n\}$.  Define $V_I':=\oplus_{i\in I} V_i'$.    We have 
\[
(G_{L})_{V_I'} = (G_{V_I})_{L}.
\]
The group $(G_{L})_{V_I'}$ is smooth since $G_{V_I}$ is smooth by assumption.  In particular, the assumptions of the proposition hold for $G_{L} \subseteq \GL_{V\otimes_F L}$.  We have $\dim (G_{L})_{V_I'} = \dim G_{V_I}$, so 
\[
\frac{\dim G_{L} - \dim (G_L)_{V_I'}}{\dim V_I'} = \frac{\dim G - \dim G_{V_I}}{\dim V_I}.
\]
Since $I$ was an arbitrary nonempty subset of $\{1,\ldots, n\}$, we deduce that $\alpha(G_L)=\alpha(G)$ by (\ref{P:newer approach ii}).
\end{proof}

\begin{remark}
We now give an example where we do not have an equality $\alpha(\gG)=\alpha(G)$ over a field of positive characteristic $p$. Let $G$ be the algebraic subgroup of $\GL_{2,\FF_p}$ given by matrices of the form $\left\{ \left(\begin{smallmatrix} x^p & 0 \\ 0 & x \end{smallmatrix} \right) \right\}$. This group $G$ is isomorphic to $\GG_{m,\FF_p}$ and in particular is smooth.  The Lie algebra of $G$ is $\gG=\left\{ \left(\begin{smallmatrix} 0 & 0 \\ 0 & \ast \end{smallmatrix} \right)\right\} \subseteq \gl(\FF_p^2)$.  However, for the subspace $W := \FF_p \cdot \left(\begin{smallmatrix}
    1 \\ 0
\end{smallmatrix}\right)$ the stabilizer $G_W = \mu_p$ is of dimension 0, whereas $\gG_W = \gG $ is of dimension 1. So in this case we have $\alpha(G) = 1$ while $\alpha(\gG)=0$.
\end{remark}

\subsection{The values $\gamma_A$ and $\gamma_{A,\ell}$} \label{SS:gamma info}

Fix a nonzero abelian variety $A$ over a number field $K$ and choose an embedding $\Kbar\subseteq \CC$.  Choose a finite extension $K'$ of $K$ in $\Kbar$ so that the abelian variety $A_{K'}$ is isogenous to a product $\prod_{i=1}^n A_i^{m_i}$, where the $A_i$ are abelian varieties over $K'$ that are simple and pairwise nonisogenous over $\CC$.    For each subset $I\subseteq \{1,\ldots, n\}$, define $A_I:=\prod_{i\in I} A_i^{m_i}$. \\

Define $V=H_1(A(\CC),\QQ)$ and $V_i=H_1(A_i^{m_i}(\CC),\QQ)$.   An isogeny between $A_{K'}$ and $\prod_{i=1}^n A_i^{m_i}$ induces an isomorphism 
\begin{align} \label{E:V decomp new}
V=\bigoplus_{i=1}^n V_i.
\end{align}
By Lemma~\ref{L:GA isotypical}, the direct sum (\ref{E:V decomp  new}) is the decomposition of the representation $V$ of $G_A$ into isotypic components.    So Proposition~\ref{P:newer approach}(\ref{P:newer approach ii}), applied to $G_A \subseteq \GL_{V}$, implies that
\[
\alpha(G_A) = \min_{\emptyset\neq I \subseteq \{1,\ldots. n\}} \frac{\dim G_A - \dim (G_A)_{V_I}}{\dim V_I},
\]
where $V_I := \bigoplus_{i\in I} V_i$.     By (\ref{E:codim key}), we have $\dim G_A - \dim (G_A)_{V_I} =\dim G_{A_I}$.  The vector space $V_I$ is isomorphic to $H_1(A_I(\CC),\QQ)$ and hence has dimension $2\dim A_I$.    Therefore,
\begin{align}\label{E:alpha eq gamma inv}
\alpha(G_A) = \min_{\emptyset\neq I \subseteq \{1,\ldots, n\}} \frac{\dim G_{A_I}}{2\dim A_I} = \frac{1}{\gamma_{A}}.
\end{align}

In particular, we have $\gamma_{G_A} = \gamma_A$.

\begin{prop} \label{P:codim bound}
Take any field extension $F/\QQ$ and set $G:=(G_A)_F$.  For any subspace $W$ of the $F$-vector space $V\otimes_\QQ F= H_1(A(\CC),F)$, we have
\begin{align} \label{E:codim bound}
\gamma_A\cdot (\dim G - \dim G_W) \geq \dim W.
\end{align}
Moreover, $\gamma_A$ is the smallest number for which this holds.  
\end{prop}
\begin{proof}
This is just a reformulation of $\alpha((G_A)_F)=\gamma_A^{-1}$ which follows from (\ref{E:alpha eq gamma inv}) and Proposition~\ref{P:newer approach}(\ref{P:newer approach ii}).
\end{proof}

\begin{remark}
We proved Proposition~\ref{P:codim bound} without computing the integers $\dim G_W$.  In the work of Hindry and Ratazzi, for example see \cite{MR2862374,MR3576113}, they explicitly compute $\dim G_W$ to give a direct proof of Proposition~\ref{P:codim bound} in various special cases.
\end{remark}

Now take any prime $\ell$.  The $\ell$-adic monodromy group $G_{A,\ell}^\circ \subseteq \GL_{V_\ell(A)}$ is reductive, so we can define $\alpha(G_{A,\ell}^\circ)$.  We claim that $\alpha(G_{A,\ell}^\circ)$ is nonzero.  If $\alpha(G_{A,\ell}^\circ)=0$, then there would be a finite extension $K'/K$ such that $V_\ell(A)$ has a nonzero subspace fixed pointwise by the action of $\Gal_{K'}$.  However, this is impossible since it would imply that $A(K')[\ell^\infty]$ is infinite, contradicting the Mordell--Weil theorem.   Therefore, $\alpha(G_{A,\ell}^\circ)$ is nonzero as claimed.    We can now define 
\begin{equation}
\label{eq:defgammaAell}
\gamma_{A,\ell}:=1/\alpha(G_{A,\ell}^\circ).
\end{equation}
Note that, by definition, for every prime $\ell$ we have $\gamma_{G^{\circ}_{A,\ell}} = \gamma_{A,\ell}$.

\begin{lemma}  \label{L:gamma connection}
Take any prime $\ell$.
\begin{romanenum}
\item \label{L:gamma connection i}
We have inequalities
\[
\gamma_A\leq \max_{\emptyset \neq I \subseteq \{1,\ldots,n\}} \frac{2\dim A_I}{\dim G_{A_I,\ell}^\circ } \leq \gamma_{A,\ell}.
\]
\item \label{L:gamma connection ii}
If the Mumford--Tate conjecture for $A$ holds, then $\gamma_{A}=\gamma_{A,\ell}$.
\end{romanenum}
\end{lemma}
\begin{proof}
Recall that there is a finite extension $K'/K$ in $\Kbar$ for which there is an isogeny  $A_{K'}\to \prod_{i=1}^n A_i^{m_i}$, where the $A_i$ are abelian varieties over $K'$ that are simple and pairwise nonisogenous over $\CC$.    We can assume that $K'$ is chosen large enough so that $G_{A_{K'},\ell}$ is connected.

 The isogeny induces an isomorphism
\[
V_\ell(A) = \bigoplus_{i=1}^n V_\ell(A_i^{m_i})
\]
of $\QQ_\ell[\Gal_{K'}]$-modules.  In particular, $V_\ell(A)$ is a representation of $G_{A,\ell}^\circ$ and the subspaces $V_\ell(A_i^{m_i})$ are invariant under the action of $G_{A,\ell}^\circ$.  Take any distinct $1\leq i, j \leq n$.  We have
\begin{align} \label{E:Homs}
\Hom_{\QQ_\ell[\Gal_{K'}]}(V_\ell(A_i^{m_i}),V_\ell(A_j^{m_j})) \cong \Hom(A_i^{m_i},A_j^{m_j}) \otimes_\ZZ \QQ_\ell
\end{align}
by Proposition~\ref{P:Faltings}(\ref{P:Faltings iii}).
Since $A_i$ and $A_j$ are simple and nonisogenous, we deduce that (\ref{E:Homs}) is $0$ and hence $V_\ell(A_i^{m_i})$ and $V_\ell(A_j^{m_j})$ contain no isomorphic irreducible representations of $G_{A,\ell}^\circ$.   By Proposition~\ref{P:newer approach}(\ref{P:newer approach ii}),
we have  
\[
\alpha(G_{A,\ell}^\circ)\leq \min_{\emptyset \neq I \subseteq \{1,\ldots,n\}} \frac{\dim G_{A,\ell}^\circ - \dim (G_{A,\ell}^\circ)_{\calV_I}}{\dim \calV_I}, 
\]
where $\calV_I:= \bigoplus_{i\in I} V_\ell(A_i^{m_i})$.   Note that we only have an inequality here since $V_\ell(A_i^{m_i})$ need not be isotypic for $G_{A,\ell}^\circ$ (though it is a direct sum of isotypic components).    Since $(G_{A,\ell})_{\calV_I}$ is the kernel of the projection homomorphism $G_{A,\ell}^\circ \to G_{A_I,\ell}$ and $\calV_I$ has dimension $2\dim A_I$, we have
\[
\alpha(G_{A,\ell}^\circ)\leq \min_{\emptyset \neq I \subseteq \{1,\ldots,n\}} \frac{\dim G_{A_I,\ell}^\circ }{2\dim A_I}, 
\]
This proves that $\gamma_{A,\ell} \geq \max_{\emptyset \neq I \subseteq \{1,\ldots,n\}} {2\dim A_I}/(\dim G_{A_I,\ell}^\circ )$.    By Proposition~\ref{P:partial MTC}(\ref{P:partial MTC i}), we deduce that
\[
\gamma_{A,\ell} \geq \max_{\emptyset \neq I \subseteq \{1,\ldots,n\}} \frac{2\dim A_I}{\dim G_{A_I}} = \gamma_A.
\]
This completes the proof of (\ref{L:gamma connection i}).  We now prove (\ref{L:gamma connection ii}).   Assuming the Mumford--Tate conjecture for $A$, we have $\gamma_{A,\ell}^{-1}= \alpha(G_{A,\ell}^\circ)=\alpha((G_A)_{\QQ_\ell})$.   Therefore, $\gamma_{A,\ell}^{-1}= \alpha((G_A)_{\QQ_\ell}) = \alpha(G_A)=\gamma_A^{-1}$ by Proposition~\ref{P:newer approach}(\ref{P:newer approach ii}).
\end{proof}

\begin{prop}\label{L:uniformness of alpha}
Assume that all the $\ell$-adic monodromy groups $G_{A,\ell}$ are connected.  For $\ell \gg_A 1$ we have $\alpha(G_{A,\ell}) = \alpha((\Gcal_{A,\ell})_{\F_\ell}) = \alpha(\gG_{A,\ell})$ where $\gG_{A,\ell}$ is the Lie algebra of $(\Gcal_{A,\ell})_{\F_\ell}$.
\end{prop}
\begin{proof} 
With notation as in Proposition~\ref{P:redux finiteness}, we fix an $i\in I$ and set $\calG:=\calG_i$.    Let $\ZZ^{2g}=\oplus_{j=1}^{s} W_{j}$ be the direct sum of representations of $\calG$ as in Proposition~\ref{P:redux finiteness}(\ref{P:redux finiteness b}).    Since $W_{j} \otimes_\ZZ \QQ$, with $1\leq j \leq s$, are the isotypical components of the action of $\calG_\QQ$ on $\QQ^{2g}$,
Proposition~\ref{P:newer approach}(\ref{P:newer approach ii}) implies that
\[
\alpha(\calG_\QQ)  = \min_{\emptyset \neq J \subseteq \{1,\ldots,s\}} \frac{\dim \calG_\QQ - \dim (\calG_\QQ)_{V_J}}{\dim V_J},
\]
where $V_J:=\oplus_{j\in J} (W_{j}\otimes_\ZZ \QQ)$.

For each subset $J\subseteq \{1,\ldots, s\}$, define the $\ZZ$-submodule $\calV_J = \oplus_{j\in J} W_{j}$ of $\ZZ^{2g}$.  Recall that $\Fix_{\calG}(\calV_{J})$  is a subgroup scheme of $\calG$, and hence its generic fibre (which is an algebraic group over a field of characteristic zero) is automatically smooth.  So there is a positive constant $c$, such that the $\ZZ_\ell$-group scheme $(\Fix_{\calG}(\calV_{J}))_{\ZZ_\ell}$ is smooth for all primes $\ell \geq c$ and all $J\subseteq \{1,\ldots, s\}$.   %We can choose $c$ large enough so that it does not depend on the choice of $i\in I$ or $J$ (since the set $I$ is finite).     
%Note that we can choose $c$ depending only on $A$  since there are only finitely many possibilities for $i$ and $J$.  

Take any prime $\ell \geq c$ and any nonempy set  $J\subseteq \{1,\ldots, s\}$.  
Since $(\Fix_{\calG}(\calV_{J}))_{\ZZ_\ell}$ is smooth, its generic fiber $(\calG_{\QQ})_{\calV_{J}\otimes_\ZZ \QQ}=(\calG_{\QQ})_{V_J}$ and special fiber $(\calG_{\FF_\ell})_{\calV_{J}\otimes_\ZZ \FF_\ell}$ are both smooth of the same dimension.   Since $\calG_{\ZZ_\ell}$ is smooth as well, we have
\[
\frac{\dim \calG_\QQ - \dim (\calG_{\QQ})_{V_{J}}}{\dim \calV_J}= \frac{\dim \calG_{\FF_\ell} - \dim (\calG_{\FF_\ell})_{\calV_{J}\otimes_\ZZ \FF_\ell}}{\dim (\calV_J\otimes_\ZZ \FF_\ell)}
\]
and hence
\begin{align} \label{E:alpha transition}
\alpha(\calG_\QQ)=\min_{\emptyset \neq J \subseteq \{1,\ldots,s\}}  \frac{\dim \calG_{\FF_\ell} - \dim (\calG_{\FF_\ell})_{\calV_{J}\otimes_\ZZ \FF_\ell}}{\dim (\calV_J\otimes_\ZZ \FF_\ell)}. 
\end{align}
After increasing our constant $c$ appropriately first, Proposition~\ref{P:redux finiteness}(\ref{P:redux finiteness b}) tells us that the $\FF_\ell$-vector spaces $W_{j}\otimes_\ZZ \FF_\ell$, with $1\leq j \leq s$, are the isotypical components of $\calG_{\FF_\ell} \subseteq \GL_{2g,\FF_\ell}$.   Since $\calV_J\otimes_\ZZ \FF_\ell = \oplus_{j\in J} W_{j}\otimes_\ZZ \FF_\ell$, we deduce from (\ref{E:alpha transition}) and Proposition~\ref{P:newer approach}(\ref{P:newer approach ii}) that $\alpha(\calG_{\QQ})=\alpha(\calG_{\FF_\ell})$;  this uses that the groups $(\calG_{\FF_\ell})_{\calV_{J}\otimes_\ZZ \FF_\ell}$ are smooth.   By Proposition~\ref{P:newer approach}(\ref{P:newer approach iii}), we have $\alpha(\calG_{\Qbar_\ell})=\alpha(\calG_{\FFbar_\ell})$.  

Therefore, for primes $\ell\gg_A 1$ and $i\in I$, we have $\alpha((\calG_i)_{\Qbar_\ell})=\alpha((\calG_i)_{\FFbar_\ell})$; 
to get that the bound on $\ell$ depends only on $A$, we use that $I$ is finite and that the chosen groups $\calG_i$ and the corresponding representations $W_j$ depend only on $A$.   By taking $\ell\gg_A 1$, Proposition~\ref{P:redux finiteness}(\ref{P:redux finiteness a}) implies that there is an $i\in I$ such that $\alpha((\calG_{A,\ell})_{\Qbar_\ell})=\alpha((\calG_i)_{\Qbar_\ell})$ and $\alpha((\calG_{A,\ell})_{\FFbar_\ell})=\alpha((\calG_i)_{\FFbar_\ell})$.  Therefore, $\alpha((G_{A,\ell})_{\Qbar_\ell})=\alpha((\calG_{A,\ell})_{\Qbar_\ell})$ is equal to $\alpha((\calG_{A,\ell})_{\FFbar_\ell})$ for $\ell\gg_A 1$.  
By Proposition~\ref{P:newer approach}(\ref{P:newer approach iii}) we have $\alpha((G_{A,\ell})_{\Qbar_\ell}) = \alpha(G_{A,\ell})$ and $\alpha((\calG_{A,\ell})_{\FFbar_\ell}) = \alpha((\calG_{A,\ell})_{\FF_\ell})$, hence we obtain $\alpha(G_{A,\ell})=\alpha((\calG_{A,\ell})_{\FF_\ell})$.

Note that for $\ell\gg_A 1$, the group $(\calG_{A,\ell})_{\FF_\ell}$ satisfies the assumption of Proposition~\ref{P:newer approach}; it suffices to show this after base extending to $\FFbar_\ell$ and we already proved it holds for the groups $(\calG_i)_{\FFbar_\ell}$ with $i\in I$.   We thus have $\alpha((\calG_{A,\ell})_{\FFbar_\ell})= \alpha(\gG_{A,\ell})$ for $\ell\gg_A 1$ by Proposition~\ref{P:newer approach}(\ref{P:newer approach i}).
\end{proof}

\section{Prime power version: large primes} \label{S:prime power version}
Fix a nonzero abelian variety $A$ of dimension $g$ defined over a number field $K$.   
Take any prime $\ell$.  In \S\ref{SS:gamma info}, we defined a positive rational number $\gamma_{A,\ell}$; it is the smallest number for which
\[
\gamma_{A,\ell}\cdot (\dim G_{A,\ell}^\circ - \dim (G_{A,\ell}^\circ)_W) \geq \dim W
\]
holds for all nonzero subspaces $W\subseteq V_\ell(A)$.   In this section, we prove the following bound for the $\ell$-power torsion of $A(L)$ for a finite extension $L/K$.

\begin{thm} \label{T:ell-adic version, large ell}
For primes $\ell\gg_A 1$, we have
\[
|A(L)[\ell^\infty]| \ll_A  [L:K]^{\gamma_{A,\ell}}
\]
for all finite extensions $L/K$.
\end{thm}

The above theorem actually holds for all primes $\ell$ and we will give a different argument for the finitely many excluded primes in \S\ref{S: small primes}. The proofs (both for a single $\ell$ and the uniform argument) are ineffective, and it would be an interesting problem to obtain explicit estimates in terms of the Faltings height of $A$ and the degree of the field extensions.

\subsection{Lie algebras and filtrations}   \label{SS:lie algebras and filtrations}
For a fixed prime $\ell$, set $\calG:=\calG_{A,\ell}$.  We shall assume that $G_{A,\ell}$ is connected and that $\calG$ is a reductive group scheme over $\ZZ_\ell$.  

By choosing a $\ZZ_\ell$-basis of $T_\ell(A)$, we may assume that $\calG \subseteq \GL_{2g,\ZZ_\ell}$.  Take any commutative $\ZZ_\ell$-algebra $R$.   Define the ring $R[\varepsilon]:=R[x]/(x^2)$, where $\varepsilon$ is the image of $x$ and hence satisfies $\varepsilon^2=0$.   The $R$-algebra homomorphism $R[\varepsilon]\to R$ mapping $\varepsilon$ to $0$ induces a homomorphism 
\begin{align}\label{E:epsilon to 0}
\calG(R[\varepsilon]) \to \calG(R).
\end{align}
Let $L(R)$ be the set of $B\in M_{2g}(R)$ for which $I+\varepsilon B$ lies in the kernel of (\ref{E:epsilon to 0}).   Observe that $L(R)$ is a \emph{Lie algebra} over $R$; it is an $R$-submodule of $M_{2g}(R)$ that is closed under the pairing $[B_1,B_2]=B_1 B_2 - B_2 B_1$.   

The Lie algebra of $G_{A,\ell}$ is $L(\QQ_\ell)$; its dimension as a $\QQ_\ell$-vector space is $\dim G_{A,\ell}$.  Since $\calG$ is the Zariski closure of $G_{A,\ell}$ in $\GL_{2g,\ZZ_\ell}$, we find that $L(\ZZ_\ell) = L(\QQ_\ell) \cap M_{2g}(\ZZ_\ell)$; it is a free $\ZZ_\ell$-module of rank $\dim G_{A,\ell}$. 

 Let $\gG_\ell$ be the image of the reduction modulo $\ell$ homomorphism $L(\ZZ_\ell)\to L(\FF_\ell)$; it is a Lie algebra over $\FF_\ell$ of dimension $\dim G_{A,\ell}$ (this uses that $\Gcal$ is smooth over $\Z_\ell$). 

\begin{lemma} \label{L:Lie algebra gG}
The Lie algebra of $\calG_{\FF_\ell}$ is $\gG_\ell$.
\end{lemma}
\begin{proof}
Since $\calG$ is smooth over $\ZZ_\ell$, the Lie algebra of $\calG_{\FF_\ell}$ is $L(\FF_\ell)$ and has dimension equal to $\dim \calG_{\FF_\ell} = \dim \calG_{\QQ_\ell} = \dim G_{A,\ell}$.  We thus have $L(\FF_\ell)=\gG_\ell$ since we have an inclusion $\gG_\ell\subseteq L(\FF_\ell)$ of $\FF_\ell$-vector spaces of the same dimension.
\end{proof}

Now consider a closed subgroup $H$ of $\calG(\ZZ_\ell)$.  For each integer $i\geq 1$, let $H(\ell^i)$ and $H_i$ be the image and kernel, respectively, of the reduction modulo $\ell^i$ homomorphism $H\to \calG(\ZZ/\ell^i\ZZ)$. The map
\[
\varphi_i\colon H_i \to  M_{2g}(\FF_\ell),\quad I+\ell^i B \mapsto B \bmod{\ell}
\]
is a group homomorphism with kernel $H_{i+1}$ whose image we will denote by $\hH_i$. 

 The group $\hH_i$ is an $\FF_\ell$-subspace of $M_{2g}(\FF_\ell)$ and we have
\begin{align} \label{E:G filtration}
|H(\ell^i)| = [H:H_i] = [H:H_1]\cdot \prod_{1\leq j<i}[H_{j}: H_{j+1}] = |H(\ell)| \prod_{1\leq j<i} |\hH_j| = |H(\ell)| \cdot \ell^{\sum_{j=1}^{i-1} \dim \hH_j}
\end{align}

\begin{lemma} \label{L:new h=g}
Let $H$ be a closed subgroup of $\calG(\ZZ_\ell)$.  Take any $i\geq 1$ with $i\geq 2$ if $\ell=2$.
\begin{romanenum}
\item \label{L:new h=g i}
We have $\hH_i \subseteq \gG_\ell$.
\item  \label{L:new h=g ii}
If $H=\calG(\ZZ_\ell)$, then $\hH_i=\gG_\ell$.
\end{romanenum}
\end{lemma}
\begin{proof}
Part (\ref{L:new h=g i}) follows from (\ref{L:new h=g ii}), so we may assume that $H=\calG(\ZZ_\ell)$.  Fix an $i\geq 1$ and take any $B\in L(\ZZ_\ell)$.   We have a homomorphism $\calG(\ZZ_\ell[\varepsilon])\to \calG(\ZZ/\ell^{i+1}\ZZ)$ arising from the ring homomorphism $\ZZ_\ell[\varepsilon]\to \ZZ/\ell^{i+1} \ZZ$ that reduces modulo $\ell^{i+1}$ and sends $\varepsilon$ to $\ell^i$.  In particular, $I+\ell^i B$ modulo $\ell^{i+1}$ lies in $\calG(\ZZ/\ell^{i+1}\ZZ)$.   The reduction map $\calG(\ZZ_\ell)\to \calG(\ZZ/\ell^{i+1}\ZZ)$ is surjective since $\calG$ is smooth.  So there is an element $I+\ell^i B' \in \calG(\ZZ_\ell)$ such that $B\equiv B' \pmod{\ell}$.  Therefore, $B$ modulo $\ell$ lies in $\hH_i$.   Since $B$ was an arbitrary element of $L(\ZZ_\ell)$, we deduce that $\gG_\ell \subseteq \hH_i$ for all $i\geq 1$.

Now suppose that $\gG_\ell \subsetneq \hH_j$ for some $j \geq 1$ with $j\geq 2$ if $\ell=2$.   There is an element $I+\ell^i B \in H_j$ such that $B$ modulo $\ell$ does not lie in $\gG_\ell$.   Raising $I+\ell^j B$ to the $\ell$-th power gives
\[
(I+\ell^j B)^\ell = I +\ell^{j+1} B + \sum_{k=2}^\ell \binom{\ell}{k} \ell^{jk} B^k.
\]
Observe that $\binom{\ell}{k} \ell^{jk} \equiv 0 \pmod{\ell^{j+2}}$ for all $2\leq k \leq \ell$; this uses that $\binom{\ell}{k} \equiv 0 \pmod{\ell}$ when $2\leq k <\ell$ (we have also used $j\geq 2$ when $\ell=2$).  Since $(I+\ell^j B)^\ell \in H_{j+1}$, this proves that $B$ modulo $\ell$ lies in $\hH_{j+1}$ and hence $\gG_\ell \subsetneq \hH_{j+1}$.     Therefore, $\gG_\ell \subsetneq \hH_{i}$ for all sufficiently large $i$.  By (\ref{E:G filtration}), this implies that $|H(\ell^i)| \gg_{A,\ell}  \ell^{i(\dim \gG_\ell +1)} = \ell^{i(\dim G_{A,\ell} +1)}$ for all $i\geq 1$.  However, we have $|H(\ell^i)| \ll_{A,\ell} \ell^{i\dim G_{A,\ell}}$, see Th\'eor\`eme~8 of \cite{MR644559}.   These inequalities contradict for $i$ large enough, so we conclude that $\gG_\ell = \hH_j$ for $j\geq 1$ with $(j,\ell)\neq (1,2)$.
\end{proof}

\begin{lemma} \label{L:elli image size}
We have $|\rho_{A,\ell^i}(\Gal_K)| \asymp_A \ell^{i \dim \gG_\ell }= \ell^{i \dim G_{A,\ell} }$.
\end{lemma}
\begin{proof}
Define the group $H=\calG(\ZZ_\ell)$.  
By Lemma~\ref{L:new h=g}(\ref{L:new h=g ii}) and (\ref{E:G filtration}), we have $|H(\ell^i)| \asymp_A |H(\ell)| \cdot \ell^{(i-1) \dim \gG_\ell}$.    Since $\calG$ is a smooth group scheme, we have $H(\ell)=\calG(\FF_\ell)$ by Hensel's lemma.    By Proposition~\ref{P:HW count} (with $W=0$), we have $|H(\ell)|=|\calG(\FF_\ell)|\asymp_A \ell^{\dim \calG_{\FF_\ell}}$.  We have equalities $\dim \gG_\ell = \dim \calG_{\FF_\ell} = \dim \calG_{\QQ_\ell} = \dim G_{A,\ell}$, where we have used the smoothness of $\calG$ along with Lemma~\ref{L:Lie algebra gG}.   Combining everything together, we find that $|H(\ell^i)| \asymp_A \ell^{i \dim \gG_\ell} =\ell^{i \dim G_{A,\ell}}$.

The image of $\rho_{A,\ell^\infty}(\Gal_K)$ modulo $\ell^i$ is the group $\rho_{A,\ell^i}(\Gal_K)$.   Therefore, \[
[H(\ell^i): \rho_{A,\ell^i}(\Gal_K)] \leq [H:\rho_{A,\ell^\infty}(\Gal_K)] \ll_A 1,
\]
where the last inequality uses Theorem~\ref{T:finite index}.   So $|\rho_{A,\ell^i}(\Gal_K)| \asymp_A |H(\ell^i)|$ and the lemma follows from the estimates of $|H(\ell^i)|$ we have  computed above.
\end{proof}

\subsection{Proof of Theorem~\ref{T:ell-adic version, large ell}} \label{SS:proof of large ell}

There is no harm in replacing $K$ by a finite extension and $A$ with its base extension by this field.  Indeed, suppose that $K'/K$ is a finite extension.   For a finite extension $L/K$, set $L'=L\cdot K'$.  We have $|A(L)_{\tors}|\leq |A(L')_{\tors}|$ and 
\[
[L':K']^{\gamma_{A,\ell}} \leq [K':K]^{\gamma_{A,\ell}} [L:K]^{\gamma_{A,\ell}}\leq [K':K]^{2\dim A} [L:K]^{\gamma_{A,\ell}}\ll_{A,K'} [L:K]^{\gamma_{A,\ell}}.
\]  
Also $\gamma_{A,\ell}=\gamma_{A_{K'},\ell}$.  So Theorem~\ref{T:ell-adic version, large ell} for $A_{K'}/K'$ implies the theorem for $A/K$.  So after first replacing $K$ by a finite extension, we may assume by Proposition~\ref{P:connected} that the algebraic groups $G_{A,\ell}$ are connected for the rest of the section.    By taking $\ell\gg_A 1$, we may assume by Proposition~\ref{P:reductive} that the $\ZZ_\ell$-group scheme $\calG_{A,\ell}\subseteq \GL_{T_\ell(A)}$ is reductive and that $\ell$ is odd.   

We now make some identifications that will hold for the rest of the proof.  By choosing a basis for $T_\ell(A)$ as a $\ZZ_\ell$-module, we will identify $\calG_{A,\ell}$ with an algebraic subgroup of $\GL_{2g,\ZZ_\ell}$.   In particular, we have $(\calG_{A,\ell})_{\FF_\ell} \subseteq \GL_{2g,\FF_\ell}$.     Define $\gG:=\gG_\ell \subseteq M_{2g}(\FF_\ell)$ as in \S\ref{SS:lie algebras and filtrations}.\\

Take any finite extension $L/K$ in $\Kbar$.   Define the group $U:=A(L)[\ell^\infty]$, i.e., the group of torsion points in $A(L)$ whose order is a power of $\ell$.  The group $U$ is finite by the Mordell--Weil theorem.    For each $i\geq 0$, let $U[\ell^i]$ be the group of $P\in U$ for which $\ell^i P =0$.   For each $i\geq 0$, let
\[
\psi_i \colon A[\ell^{i+1}] \xrightarrow{\sim} (\ZZ/\ell^{i+1}\ZZ)^{2g} 
\]
be the isomorphism obtained by our choice of $\ZZ_\ell$-basis for $T_\ell(A)$.  Composing $\psi_i$ with the reduction modulo $\ell$ map induces an isomorphism $\bbar{\psi}_i \colon A[\ell^{i+1}]/A[\ell^i] \xrightarrow{\sim} \FF_\ell^{2g}$.   We can identify $U[\ell^{i+1}]/U[\ell^i]$ with a subgroup of $A[\ell^{i+1}]/A[\ell^i]$, so we can define
\[
W_i := \bbar{\psi}_i(U[\ell^{i+1}]/U[\ell^i]);
\]
it is a subspace of $\FF_\ell^{2g}$. 

\begin{lemma} \label{L:lie sum}
For each $i\geq 1$, we have $|\rho_{A,\ell^i}(\Gal_L)| \ll_A \ell^{\sum_{j=0}^{i-1} \dim \gG_{W_j}}$.
\end{lemma}
\begin{proof}
Define the group $H:=\rho_{A,\ell^\infty}(\Gal_L)$.   For each $i\geq 1$, define $H(\ell^i)$, $H_i$ and $\hH_i$ as in \S\ref{SS:lie algebras and filtrations}.   

We claim that $\hH_i \subseteq\gG_{W_i}$ for all $i\geq 1$.     Take any $I+\ell^i B \in H_i$.  Since $\hH_i \subseteq \gG$ by Lemma~\ref{L:new h=g}(\ref{L:new h=g i}), to prove the claim, we need only show that $Bw=0$ for all $w\in W_i$.  
Choose an element $\sigma\in \Gal_L$ for which $\rho_{A,\ell^\infty}(\sigma)=I+\ell^i B$.     We have $\sigma(P)=P$ for all $P\in U[\ell^{i+1}]$ since $U\subseteq A(L)$.   Therefore, $I+\ell^i B$ fixes each element of $\psi_i(U[\ell^{i+1}])$.   So for each $w\in \psi_i(U[\ell^{i+1}])$, we have $(I+\ell^i B) w = w$ and hence $\ell^i B w = 0$.  Therefore, $Bw\equiv 0 \pmod{\ell}$ for all $w\in \psi_i(U[\ell^{i+1}])$.   The claim is now immediate since $W_i$ is the image of $\psi_i(U[\ell^{i+1}])$ modulo $\ell$.  

Take any $i\geq 1$.  By (\ref{E:G filtration}) and the above claim, we have 
\begin{align*}
|\rho_{A,\ell^i}(\Gal_L)|=|H(\ell^i)| \ll_A  |H(\ell)| \cdot \ell^{\sum_{j=1}^{i-1} \dim \gG_{W_j}}.
\end{align*}
The group $\Gal_L$ fixes $U[\ell] \subseteq A(L)$, so $H(\ell)$ fixes each element of $W_0$.  Therefore, $H(\ell)\subseteq G_{W_0}(\FF_\ell)$, where $G:=\calG_{\FF_\ell}$.  By Proposition~\ref{P:HW count}, we have $|H(\ell)|\leq |G_{W_0}(\FF_\ell)| \ll_A \ell^{\dim G_{W_0}}$.  Therefore,
\[
|\rho_{A,\ell^i}(\Gal_L)| \ll_A  \ell^{\dim G_{W_0}} \cdot \ell^{\sum_{j=1}^{i-1} \dim \gG_{W_j}}.
\]
Since $G$ has Lie algebra $\gG$ by Lemma~\ref{L:Lie algebra gG}, $G_{W_0}$ will have Lie algebra $\gG_{W_0}$ and hence $\dim G_{W_0}\leq  \dim \gG_{W_0}$. 
Therefore, $|\rho_{A,\ell^i}(\Gal_L)| \ll_A   \ell^{\sum_{j=0}^{i-1} \dim \gG_{W_j}}$.
\end{proof}

Take any $i\geq 1$ large enough so that $W_j=0$ for all $j\geq i$.  By Lemmas~\ref{L:elli image size} and \ref{L:lie sum}, we have
\begin{align*}
[L:K] \geq [\rho_{A,\ell^i}(\Gal_K):\rho_{A,\ell^i}(\Gal_L)] \gg_A \ell^{\sum_{j=0}^{i-1}(\dim \gG - \dim \gG_{W_j})}.
\end{align*}
With notation as in \S\ref{SS:codimension bounds lie algebras}, we have 
\[
[L:K] \gg_A \ell^{\sum_{j=0}^{i-1} \alpha(\gG) \dim W_j} = \left( {\prod}_{j=0}^{i-1} |W_j| \right)^{\alpha(\gG)} = |U|^{\alpha(\gG)}.
\]
We have $\alpha(\gG)=\alpha((\calG_{A,\ell})_{\FF_\ell}) = \alpha(G_{A,\ell}) = \gamma_{A,\ell}^{-1}$ for $\ell$ large enough by Proposition~\ref{L:uniformness of alpha}.  The numerator and denominator of $\gamma_{A,\ell}$ can be bounded in terms of the dimension of $A$, so we have $|A(L)[\ell^\infty]|= |U| \ll_A [L:K]^{\gamma_{A,\ell}}$ for $\ell \gg_A 1$.

\section{Prime power version: small primes}
\label{S: small primes}

In this section, we prove the following version of Theorem \ref{T:ell-adic version, large ell}, valid for all primes $\ell$.
\begin{thm}\label{T: small primes}
For every nonzero abelian variety $A$ of dimension $g$ over a number field $K$ and every prime $\ell$, we have \[ |A(L)[\ell^\infty]| \ll_{A,\ell}   [L : K]^{\gamma_{A,\ell}} \] for all finite extensions $L/K$.
\end{thm} 

Since Theorem \ref{T:ell-adic version, large ell} has an implicit constant depending only on $A$ for all sufficiently large primes $\ell$, the constant can actually be taken to be independent of $\ell$ by Theorem~\ref{T: small primes}. In fact, the two results are complementary: the proof of Theorem \ref{T:ell-adic version, large ell} relies on remarking that, for $\ell$ sufficiently large, we only need to consider a finite family of pointwise stabilizers, which are all smooth (again, when $\ell$ is large), as we see in the proof of Lemma \ref{L:uniformness of alpha}. On the other hand, for small $\ell$ the group schemes $\Gcal_{A,\ell}$ can lack several desirable properties (including smoothness or reductivity), which explains why we need to adopt a more general point of view. In fact, as can be seen for example from Proposition \ref{propmain}, the proof of Theorem \ref{T: small primes} has little to do with $\ell$-adic monodromy groups and more with general linear group schemes over $\Z_\ell$. We remark that the question of smoothness (or flatness) for stabilizers in reductive group schemes is a current topic of research: see for example \cite{Cotner22}, where -- even under strong assumptions on the fibrewise dimensions of centralizers -- the proofs are quite delicate. 

Finally, we note that the proof of Theorem \ref{T: small primes} that we give below can be made uniform in $\ell$ (at the cost of more technical statements): this can be achieved easily by assuming the Mumford--Tate conjecture, and also unconditionally, with some more work, by relying on the finiteness statement given by Proposition \ref{P:redux finiteness}.

 \subsection{Grassmannians and stabilizers}\label{sect:NotationSchemesStabilisers}
 
In order to prove Theorem \ref{T: small primes} we need to control the behaviour of the subgroups of $\Gcal_{A, \ell}$ that arise as pointwise stabilizers of certain (saturated) submodules $\Wcal$ of $T_\ell A$. As it turns out, the questions we are interested in are more easily studied in families, by letting $\Wcal$ vary among all saturated submodules of a given rank. This can be achieved by considering a suitable universal stabilizer group scheme over the Grassmannian. We now introduce the necessary definitions, starting with the Grassmannian itself; for its basic properties, we refer the reader to \cite[Section 8.4]{GortzWedhorn}.

 \begin{definition}[Grassmannian]
Let $n$ be a positive integer and fix $d \in \{1, \ldots, n\}$. The \defi{Grassmannian of $d$-dimensional submodules in $n$-dimensional space}, denoted by $\Grassdn$, is the scheme which represents the contravariant functor in schemes
\[
 S \mt \{ \Ocal_S\text{-submodule } \Ucal \subseteq \Ocal_S^n \, \bigm\vert \, \Ocal_S^n / \Ucal \textrm{ is a locally free $\Ocal_S$-module of rank } n-d\}.
\]
 \end{definition}

\begin{remark}\label{rmk:OpenCoverGrassmannian}
The scheme $\Grassdn$ has a finite open covering by schemes isomorphic to $\mathbb{A}^{d(n-d)}$, see \cite{GortzWedhorn}*{Corollary 8.15}. Moreover, for every PID $R$, $\Grassdn(R)$ is the set of saturated free submodules of $R^n$ of rank $d$.
\end{remark}

We can now define the desired stabilizer scheme over the Grassmannian:
 \begin{prop}\label{P: Fix exists}
The functor on algebras 
 \[
  R \mt \{ (\varphi,\Wcal) \in \GL_n(R) \times \Grassdn(R) \, \bigm\vert \, \varphi_{|\Wcal} \textrm{ is the identity on }\Wcal \}
 \]
 is represented by a subscheme of $\GL_n \times \Grassdn$.
 \end{prop}
 \begin{proof}
Consider the canonical open covering of $\Grassdn$ by copies of $\mathbb{A}^{d(n-d)}$ given in \cite{GortzWedhorn}*{Corollary 8.15}. On each such affine piece, the condition that $\varphi$ be the identity on $\Wcal$ amounts to a finite number of equations involving the generators of $\Wcal$, which in turn may be expressed in terms of the coordinates of $\mathbb{A}^{d(n-d)}$. These equations glue to give the desired subscheme.
 \end{proof}
 
\begin{definition}
We denote by $\Fix$ the scheme representing the functor of Proposition \ref{P: Fix exists}. Explicitly, $\Fix$ is a subgroup scheme of $\GL_{n,\Grassdn}$ with the following property: for every ring $R$ and every $\Wcal \in \Grassdn(R)$, the pullback group scheme of $\Fix$ by $\Wcal : \Spec R \ra \Grassdn$, denoted by $\Fix(\Wcal)$, satisfies
\[
 \Fix(\Wcal)(R) = \{ \varphi \in \GL_n(R) \, \bigm\vert \, \varphi_{|\Wcal} \textrm{ is the identity on }\Wcal\}.
\]
\end{definition}

Finally, we introduce the following definition for arbitrary linear algebraic groups:

\begin{definition}
Let $R_0$ be a ring and let $\Gcal$ be a linear algebraic subgroup of $\GL_{n, R_0}$. We denote by $\Fix_\Gcal$ the subgroup scheme of $\Gcal_{\Grass_{d,n,R_0}}$ given by the (scheme-theoretic) intersection of $\Gcal_{\Grass_{d,n,R_0}}$ with $\Fix_{R_0}$ inside $\GL_{n, \Grass_{d, n, R_0}}$.   
\end{definition}

By definition, for every $R_0$-algebra $R$ and every submodule $\Wcal$ of $R^n$ such that $R^n/\Wcal$ is locally free of rank $n-d$ we have 
\[
 \Fix_\Gcal(\Wcal)(R) = \{ \varphi \in \Gcal(R) \, | \, \varphi_{|\Wcal}  \textrm{ is the identity on }\Wcal\}.
\]
We thus recover the definition given in \S\ref{SS:notation}.
As already pointed out in that section, when $R_0=k$ is a field, $G$ is an algebraic subgroup of $\GL_{n, k}$ and $W$ is a subspace of $k^n$, the group $\Fix_G(W)$ is simply the group $G_W$.
When $R$ is a PID, the condition on $\Wcal$ amounts to saying that $\Wcal$ is a saturated submodule of $R^n$.

Our main objective in this section is to understand the lack of smoothness of groups of the form $\Fix_\Gcal(\Wcal)$. Notice that, even when $\Gcal$ is smooth, $\Fix_\Gcal(\Wcal)$ can easily fail to be smooth (the problem usually being its lack of flatness).

\subsection{Reduction to point-counting on group schemes}\label{ss: reduction to point counting}

We will deduce Theorem \ref{T: small primes} from the group-theoretic statement below, whose proof will occupy the rest of the section. 

In order to ease the notation, for any scheme $X$ over $\Z_\ell$ we write $X(\ell^i) := X(\Z/\ell^i \Z)$. We also write $\pi_i : X(\Z_\ell) \rightarrow X(\ell^i)$ and $\pi_{j,i} : X(\ell^j) \ra X(\ell^i)$ for the natural reduction maps modulo $\ell^i$, for any $0 \leq i \leq j$. Notice that $X(\ell^i)$ does not have the same meaning here as in \S \ref{S:prime power version}.

\begin{prop}
	\label{propmain}
Let $\Gcal \subseteq \GL_{n,\Z_\ell}$ be a linear group scheme such that $G = \Gcal_{\mathbb{Q}_\ell}$ has finite slope. Then, there is a positive constant $C(\Gcal)$ such that for every $m \geq 1$ and every subgroup $H \subseteq (\Z/\ell^m \Z)^n = \A^n(\ell^m)$ we have
 \[
  [\Gcal(\Z_\ell) : \fix_{\Gcal(\Z_\ell)}(H)] \geq \frac{1}{C(\Gcal)} |H|^{1 / \gamma_G},
 \]
where
\[
 \fix_{\Gcal(\Z_\ell)}(H) := \{ M \in \Gcal(\Z_\ell) \, | \, \pi_m(M) \textrm{ is the identity on }H\}.
\]
\end{prop}

\begin{remark}
	\begin{romanenum}
		\item It would be better to  state the result only in terms of $\Z/\ell^m \Z$-points, but this is not always possible when $\Gcal$ is not smooth, as the proof will make clear, hence this somewhat inelegant inequality. In the smooth case, $\pi_m : \Gcal(\Z_\ell) \ra \Gcal (\ell^m)$ is surjective, hence it is enough to prove that 
		\[
		[\Gcal(\ell^m) : \{ M \in \Gcal(\ell^m) \, | \, M \textrm{ is the identity on }H\}]  \geq   \frac{1}{C(\Gcal)} |H|^{1 / \gamma_G}.
		\]
		\item The subgroup \enquote{$\fix$} does not have a scheme-theoretic interpretation in itself, as $H$ is not always a direct factor of $(\Z/\ell^m \Z)^n$. We will see that when $H$ is a direct factor, it does relate to the definition of $\Fix$ as above, and in fact the proof works by reduction to this case.
		
	\end{romanenum}

\end{remark}

We begin by showing that Proposition \ref{propmain} implies Theorem \ref{T: small primes}.

\begin{proof}[Proof of Proposition \ref{propmain} $\Rightarrow$ Theorem \ref{T: small primes}]
Up to replacing $K$ by a finite extension we can and do assume that $\mathcal{G}_{A, \ell}$ is connected for all primes $\ell$. 
	Let $L/K$ be a finite extension and let $\ell$ be a prime number. Define $H := A(L)[\ell^\infty]$. By the Mordell-Weil theorem, there is an integer $m \geq 1$ such that $H \subseteq A(L)[\ell^m]$. Every $\sigma \in \Gal(\overline{K}/L)$ fixes $H$ pointwise, and hence
	\[
	\rho_{A,\ell^\infty}( \Gal(\overline{K}/L) ) \subseteq \fix_{\Gcal_{A,\ell}(\Z_\ell)}(H).
	\]
We apply Proposition \ref{propmain} to $\Gcal_{A, \ell}$ and obtain the inequality
	\[
	[\Gcal_{A,\ell}(\Z_\ell) : \fix_{\Gcal_{A,\ell}(\Z_\ell)}(H)] \gg_{A,\ell} |H|^{1/\gamma_{A,\ell}}.
	\]
Therefore, $[\Gcal_{A,\ell}(\Z_\ell) : 	\rho_{A,\ell^\infty}( \Gal(\overline{K}/L) )] \gg_{A,\ell} |H|^{1/\gamma_{A,\ell}}$.   Since $\rho_{A,\ell^\infty}(\Gal_K)$ is a finite-index subgroup of $\Gcal_{A,\ell}(\Z_\ell)$ with index bounded independent of $\ell$ by Theorem~\ref{T:finite index}, we have 
\[
[L:K] \geq [\rho_{A,\ell^\infty}(\Gal_K): \rho_{A,\ell^\infty}( \Gal(\overline{K}/L) )] \gg_{A,\ell} [\Gcal_{A,\ell}(\Z_\ell) : 	\rho_{A,\ell^\infty}( \Gal(\overline{K}/L) )] \gg_{A,\ell} |H|^{1/\gamma_{A,\ell}}.
\]
Raising both sides to the power $\gamma_{A,\ell}$ gives $|A(L)[\ell^\infty]|=|H| \ll_{A,\ell} [L:K]^{\gamma_{A,\ell}}$.
\end{proof}

Even though Proposition \ref{propmain} is not explicitly formulated in terms of $\Fix_\Gcal(\Wcal)$, we now show that it may be deduced from sufficiently strong estimates on the number of $\Z/\ell^n\Z$-points of groups of the form $\Fix_\Gcal(\Wcal)$. The precise result we need is Proposition \ref{propkey} below. To state it, we introduce a notion of partial slope. Let $V$ be an $n$-dimensional vector space over a field $k$ and let $G$ be an algebraic subgroup of $\GL_V$. For every $r=1, \ldots, n$ we define 
\begin{equation}\label{eq: dr}
 d_r(G) := \max_{\substack{W \subseteq V \\ \dim W = r}} \dim \Fix_G(W).
\end{equation}
If $G$ is of finite slope, we have
\begin{equation}
 \label{eqpartslope}
  \dim G - d_r(G) \geq r/\gamma_G.
\end{equation}
The key auxiliary result is the following.
\begin{prop}[Key estimate]
\label{propkey}
 Let $\ell$ be a prime number, $M$ a free $\Z_\ell$-module of rank $n$ and $\Gcal \subseteq \GL_{M}$ a linear group scheme whose generic fiber $G$ has finite slope. There exists a constant $C(\Gcal)$ such that for every $r =1, \ldots, n$, every saturated $\Z_\ell$-submodule $\Wcal \subseteq M$ of rank $r$ and all integers $m \geq m' \geq 0$ we have 
 \begin{equation}
 \label{eqkey}
  \left|  \Ker \left( \pi_{m,m'}: \Fix_\Gcal(\Wcal)(\ell^m) \ra \Fix_\Gcal(\Wcal)(\ell^{m'} )\right) \right| \leq C(\Gcal) \ell^{d_r(G)(m-m')}.  
 \end{equation}

\end{prop}

\begin{remark}
 Such a result is easy to establish if we replace $C(\Gcal)$ by something possibly depending on $\Gcal$ and $\Wcal$, as $d_r(G)$ is found to be an upper bound for the dimension of the generic fiber of $\Fix_\Gcal(\Wcal)$. The main difficulty is to give bounds that are uniform in $\Wcal$.  Furthermore, if $\Fix_\Gcal(\Wcal)$ is known to be smooth of relative dimension $d$, the left-hand side of \eqref{eqkey} becomes exactly $\ell^{d(m-m')}$ (for $m' \geq 1$), so we see that all the complications come from not being able to make this assumption.
\end{remark}

Before diving into the technical lemmas required to prove this Proposition, let us first see how it implies Proposition \ref{propmain} (in short: by dévissage).
\begin{proof}[Proof of Proposition \ref{propkey} $\Rightarrow$ Proposition \ref{propmain}]
We take the notation of Proposition \ref{propmain}. Consider a subgroup $H$ of $(\Z/\ell^m \Z)^n$. There exists a basis $(e_1, \ldots, e_n)$ of $\Z_\ell^ n$ and integers $m \geq m_1 \geq \ldots \geq m_r$ with $1 \leq r \leq n$ such that 
\[
 H = \bigoplus_{i=1}^r \langle \ell^{m-m_i} \pi_m (e_i) \rangle \subseteq (\Z/\ell^m \Z)^n.
\]
In particular, notice that 
\begin{equation}
 \label{eqcardH}
 |H| = \prod_{i=1}^r \ell^{m_i}.
\end{equation}
We define, for every $j=1, \ldots, r$, 
\[
 \Wcal_j := \bigoplus_{i=1}^j \Z_\ell e_i \quad \text{ and } \quad \Gcal_j := \Fix_{\Gcal_{\Z_\ell}}(\Wcal_j).
\]
For every $M \in \Gcal(\Z_\ell)$, 
\begin{eqnarray*}
 \pi_m(M) \textrm{ fixes }H & \iff & \pi_m(M) \textrm{ fixes } \ell^{m-m_i} \pi_m(e_i) \text{ for all } 1 \leq i \leq r \\
 & \iff &  \pi_{m_i}(M) \in \Gcal_i (\ell^{m_i})  \text{ for all } 1 \leq i \leq r
\end{eqnarray*}
as the $\Wcal_j$ form a strictly increasing sequence of saturated submodules of $\Z_\ell^n$. Consequently, 
\[
 \fix_{\Gcal(\Z_\ell)}(H) = \bigcap_{i=1}^r \pi_{m_i}^{-1} (\Gcal_i (\ell^{m_i})).
\]
As $m \geq m_1 \geq m_2 \geq \ldots \geq m_r$, the index we want to bound from below is
\begin{equation}
 \label{eqmtom1}
 [\Gcal(\Z_\ell) : \fix_{\Gcal(\Z_\ell)}(H)] \geq [\pi_{m_1}(\Gcal(\Z_\ell)) : \pi_{m_1} (\fix_{\Gcal(\Z_\ell)}(H))]  = \frac{|\pi_{m_1}(\Gcal(\Z_\ell))|}{|\pi_{m_1}( \fix_{\Gcal(\Z_\ell)}(H'))| },
\end{equation}
where we have identified $H \subseteq (\ell^{m-m_1} \Z/\ell^{m} \Z)^n$ to a subgroup $H' \subseteq (\Z/\ell^{m_1} \Z)^n \cong (\ell^{m-m_1} \Z/\ell^{m} \Z)^n$. We can now assume $m=m_1$ and $H=H'$, since the right-hand side does not depend on $m$ anymore. We now bound the ratio ${|\pi_{m_1}(\Gcal(\Z_\ell))|}/{|\pi_{m_1}( \fix_{\Gcal(\Z_\ell)}(H))| }$.

For the numerator, we have a (non-effective) lower bound of the form $C_1(\Gcal) \ell^{m \dim G}$ with $C_1(\Gcal)>0$, see Lemma \ref{lemlowbound} and the comments following it.  For the denominator, we have 
\[
\left| \pi_m( \fix_{\Gcal(\Z_\ell)}(H)) \right| \leq  \left| \{ M \in \Gcal(\ell^m) \, : M \textrm{ fixes } H\} \right| \leq \left| \bigcap_{i=1}^r \pi_{m,m_i}^{-1} (\Gcal_i (\ell^{m_i})) \right|.
\]

To bound this cardinality we proceed as follows. To obtain an element of this intersection, first we choose a matrix $M_r \in \Gcal_r(\ell^{m_r})$, then we choose a lift $M_{r-1}$ of $M_r$ in $\Gcal_{r-1} (\ell^{m_{r-1}})$, and notice that two different lifts are multiplicatively related by a matrix of $\Gcal_{r-1}(\ell^{m_{r-1}})$ whose reduction modulo $m_r$ is the identity (in particular, there are at most $|\Ker \Gcal_{r-1} (\ell^{m_{r-1}}) \ra \Gcal_{r-1} (\ell^{m_{r}})|$ such lifts). The same holds until we have lifted back to $\Gcal(\ell^m)$. Setting by convention $m_{r+1} = 0$, we thus have 
\[
\left| \bigcap_{i=1}^r \pi_{m,m_i}^{-1} (\Gcal_i (\ell^{m_i})) \right| \leq \prod_{i=1}^r \left| \Ker (\Gcal_i(\ell^{m_i}) \ra \Gcal_i (\ell^{m_{i+1}}) ) \right|.
\]
Using Proposition \ref{propkey} for every $\Gcal_i$ we then get 
\begin{eqnarray*}
 \frac{|\pi_{m_1}(\Gcal(\Z_\ell))|}{|\pi_{m_1}( \fix_{\Gcal(\Z_\ell)}(H))|} & \geq & \frac{C_1(\Gcal)\ell^{m \dim G} }{C(\Gcal)^r \prod_{i=1}^r  \ell^{(m_i - m_{i+1}) d_i(G)}} \\
 & \geq & \frac{C_1(\Gcal)}{C(\Gcal)^r} \prod_{i=1}^r \ell^{(m_i - m_{i+1}) (\dim G - d_i(G))} \\
 & \geq & \frac{C_1(\Gcal)}{C(\Gcal)^r} \prod_{i=1}^r \ell^{(m_i - m_{i+1})i/\gamma_G} \\
 & = & \frac{C_1(\Gcal)}{C(\Gcal)^r} \left(\ell^{\sum_{i=1}^r m_i} \right)^{1/\gamma_G} \\
 & = & \frac{C_1(\Gcal)}{C(\Gcal)^r} |H|^{1/\gamma_G}.
\end{eqnarray*}
Here we used the equality $m = m_1 = \sum_{i=1}^r (m_i - m_{i+1})$ in the second line, Inequality \eqref{eqpartslope} in the third line, and Equation \eqref{eqcardH} in the last line. Combined with \eqref{eqmtom1}, this implies the proposition.
\end{proof}

\begin{remark}
 The last sequence of inequalities is the counterpart in the present setting of Lemma \ref{L:lie sum}.
\end{remark}

\subsection{Uniformizing the behaviour of the pointwise stabilizers}

Although the group schemes $\Fix_\Gcal(\Wcal)$ that appear in Proposition \ref{propkey} are not a priori smooth, the use of the pointwise stabilizer scheme on the Grassmannian allows us to obtain the following \enquote{uniform} statement, which suffices to prove Proposition \ref{propkey} when combined with the counting lemmas in the next section.

\begin{prop}\label{prop:er}
	Let $\ell$ be a prime number, let $M$ be a free $\Z_\ell$-module of rank $n$, and let $\Gcal \subseteq \GL_{M}$ be a linear group scheme whose generic fiber $G$ has finite slope. Denote by $\varepsilon : \Spec \Z_\ell \ra \Gcal$ the unit section.
	For every $r =1, \ldots, n$ there is an integer $e_r \geq 0$ such that for all saturated submodules $\Wcal \subseteq M$ of rank $r$, the $\Z_\ell$-module $\ell^{e_r} \cdot \varepsilon^* \Omega^1_{\Fix_\Gcal(\Wcal)/\mathbb{Z}_\ell}$ can be generated by a set of cardinality at most $d_r(G)$. Equivalently, for $N = \ell^{e_r} \cdot \varepsilon^* \Omega^1_{\Fix_\Gcal(\Wcal)/\mathbb{Z}_\ell}$, we have $\dim_{\FF_\ell} N/\ell N \leq d_r(G)$.
\end{prop}

\begin{remark}\label{rmk:FxVanishesAtKPoints}
For any $W \subseteq M \otimes_{\Z_\ell} \Q_\ell$ of dimension $r$, by definition of $d_r$ the group scheme $\Fix_G(W)$ has dimension at most $d_r(G)$, and being an affine algebraic group in characteristic 0 it is smooth by Cartier's theorem. It follows that $\varepsilon^* \Omega^1_{\Fix_G(W)/\Q_\ell}$ is a vector space of dimension at most $d_r(G)$. The difficulty lies in extending this statement to $\Z_\ell$, up to allowing multiplication by a nonzero element $\ell^{e_r}$.
\end{remark}

\begin{proof}
As explained in Section \ref{sect:NotationSchemesStabilisers}, for every $r=1, \ldots, n$, we can view $\Fix_\Gcal$ as a subgroup scheme of $\Gcal_{\Grass_{r, n, \Z_\ell}}$, with unit section again denoted by $\varepsilon$. We now define the sheaf 
	\[
	\Fcal := \bigwedge^{d_r(G)+1} \varepsilon^* \Omega^1_{\Fix_\Gcal /  \Grass_{r, n, \Z_\ell}}.
	\]
It is a coherent sheaf over $\Grass_{r, n, \Z_\ell}$, and for every $x \in \Grass_{r, n}(\Z_\ell)$, corresponding to a saturated submodule $\Wcal$ of $M$ of rank $r$, compatibility of pullbacks and exterior powers with base change gives
	\[
	 \bigwedge^{d_r(G)+1} \varepsilon^* \Omega^1_{\Fix_\Gcal(\Wcal)/\Z_\ell} = x^* \Fcal.
	\]
	In particular, for the $\Q_\ell$-points $x \in \Grass_{r, n}(\Q_\ell)$, by definition of $d_r(G)$ we have $x^* \Fcal \cong \Fcal_x = 0$ (see Remark \ref{rmk:FxVanishesAtKPoints}). 
{By Corollary \ref{cor:SheafUniformlyAnnihilated} below}, there exists an integer $e_r \geq 0$ such that $\ell^{e_r} \cdot x^*\calF=0$ for all $x \in \Grass_{r, n, \Z_\ell}(\Z_\ell)$.
	This finally gives 
	\[
	\bigwedge^{d_r(G)+1} (\ell^{e_r} \varepsilon^* \Omega^1_{\Fix_\Gcal(\Wcal)/\Z_\ell}) = \ell^{e_r(d_r(G)+1)} x^*\Fcal = 0,
	\]
	which proves that $\ell^{e_r} \cdot \varepsilon^* \Omega^1_{\Fix_\Gcal(\Wcal)/\Z_\ell}$ is generated by a set of cardinality at most $d_r(G)$.
\end{proof}
 {
 It remains to show that, as claimed in the previous proof, we can find an integer $e_r$ such that $\ell^{e_r} \cdot x^* \calF=0$ for all $\Z_\ell$-sections $x$. This is achieved in the next lemma and corollary.
 \begin{lemma}\label{lemma:SheafOnAffineSpaceZell}
 Let $\mathcal{F}$ be a coherent sheaf over $\A^d_{\Z_\ell}$. If $x^*\mathcal{F} \cong \mathcal{F}_x$ vanishes at every $\Q_\ell$-point $x$ of $\A^d_{\Z_\ell}$, there exists an integer $e \geq 0$ such that $\ell^e \cdot y^*\mathcal{F}=0$ for every $\Z_\ell$-point $y$ of $\A^d_{\Z_\ell}$.
 \end{lemma}
 \begin{proof}
 Write $\A^d_{\Z_\ell} = \operatorname{Spec} A$, where $A=\Z_\ell[t_1, \ldots, t_d]$, and let $M$ be the $A$-module corresponding to $\mathcal{F}$. As $\mathcal{F}$ is coherent, $M$ is the cokernel of a certain $A$-linear map $A^m \to A^k$, so that we can write $M = A^k / RA^m$ for a suitable matrix $R \in \operatorname{M}_{k \times m}(A)$. The condition $\mathcal{F}_x=0$ for $x \in \A^d_{\Z_\ell}(\Q_\ell)=\Q_\ell^d$ translates to the following. A $\Q_\ell$-point $x$ can be identified with a $d$-tuple $(x_1,\ldots,x_d) \in \Q_\ell^d$. For such a point, denote by $R(x)$ the evaluation of $R$ (which is a matrix of polynomials in the variables $t_1, \ldots, t_d$) at $(t_1, \ldots, t_d)=(x_1, \ldots, x_d)$. Note that $R(x)$ is a $k \times m$ matrix with coefficients in $\Q_\ell$. The stalk of $M$ at $x$ is then $\Q_\ell^k / R(x)\Q_\ell^m$, so it vanishes if and only if $\operatorname{rk} R(x)=k$. This is equivalent to the fact that at least one $k \times k$ minor of $R(x)$ is nonzero, and by assumption this holds for all $x \in \A^d_\ell(\Q_\ell)$.
 
Let $I$ be the set of all $k \times k$ submatrices of $R$. For each $i \in I$, let $r_i(t_1, \ldots, t_d) \in A$ be the minor of $R$ indexed by $i$. As we already argued, for every $x=(x_1, \ldots, x_d) \in \Q_\ell^d$, at least one of the $r_i(t_1, \ldots, t_d)$ is nonvanishing at $x$, so that the function $f(x) := \min_{i \in I} v_\ell(r_i(x))$ is everywhere finite on $\Z_\ell^d$. The function $f(x)$ is also trivially continuous, so $f(\Z_\ell^d) \subseteq \Z$ is compact, and hence bounded. In particular, there exists an integer $e \geq 0$ such that, for each $y \in \Z_\ell^d$, there exists $i \in I$ for which $v_\ell(r_i(y)) \leq e$. This implies that the $\Z_\ell$-span of the columns of the sub-matrix of $R(y)$ indexed by $i$ contains $(\ell^e\Z_\ell)^k$, hence, a fortiori, the same holds for the $\Z_\ell$-span of all the columns of $R(y)$, and we have shown that this holds for all $y$.
Thus, $\Z_\ell^k / R(y)\Z_\ell^m$ is killed by $\ell^e$ for all $y \in \Z_\ell^d = \A^d_{\Z_\ell}(\Z_\ell)$. As $\Z_\ell^k / R(y)\Z_\ell^m \cong y^*\mathcal{F}$, the claim follows.
 \end{proof}
 
 \begin{cor}\label{cor:SheafUniformlyAnnihilated}
 Let $\mathcal{F}$ be a coherent sheaf over $\Grass_{r, n, \Z_\ell}$. Suppose that, for each $\Q_\ell$-point $x \in \Grass_{r, n, \Z_\ell}(\Q_\ell)$, we have $x^* \calF \cong \calF_x =0$. Then, there exists an integer $e \geq 0$ such that $\ell^e \cdot x^* \calF=0$ for every $x \in \Grass_{r, n, \Z_\ell}(\Z_\ell)$.
 \end{cor}
 \begin{proof}
 The scheme $\Grass_{r, n, \Z_\ell}$ is a finite union of affine spaces over $\Z_\ell$, see Remark \ref{rmk:OpenCoverGrassmannian}. For each such affine space $S$, Lemma \ref{lemma:SheafOnAffineSpaceZell} gives an integer $e_S$ for which $\ell^{e_S} \cdot x^* \calF = 0$ for all $\Z_\ell$-points $x$ of $S$. The corollary follows by taking $e=  \max\{e_S\}$.
 \end{proof}
 }
 \subsection{Counting lemmas}
 
 First, for the lower bound on the numerator of \eqref{eqmtom1}, we used the following result.
 
 \begin{lemma}
 	\label{lemlowbound}
 	For any affine subvariety $\Xcal \subseteq \A^N_{\Z_\ell}$ such that $\Xcal (\Z_\ell)$ is a nonempty open subset of $\Xcal(\Q_\ell)$ and $\Xcal_{\Q_\ell}$ is equidimensional of dimension $d$, there is a positive constant $C(\Xcal)$ such that for all $m \geq 1$, 
 	\[
 	\left| \pi_m( \Xcal(\Z_\ell)) \right| \geq C(\Xcal) \ell^{d m}. 
 	\]
 \end{lemma}
 
 \begin{proof}
 	Considering $\Xcal(\Z_\ell)$ as a closed analytic subvariety of dimension $d$ of $\Z_\ell^N$, we can apply \cite[Théorème 2]{Oesterle82}. Note that the set $X_m$ of \cite{Oesterle82} is exactly our $\pi_m( \Xcal(\Z_\ell))$, and that we use the fact that $\pi_m( \Xcal(\Z_\ell))$ contains at least one element. The measure $\mu_d(\Xcal(\Z_\ell))$ is nonzero because $\Xcal(\Z_\ell)$ is of dimension $d$.
 \end{proof}

Any group scheme $\Gcal \subseteq \GL_{n,\Z_\ell}$ satisfies the hypotheses of Lemma \ref{lemlowbound} (embedding $\GL_{n}$ in the affine space $\A^{n^2+1}$) because $\Gcal(\Z_\ell) \neq \emptyset$ is open in $\Gcal(\Q_\ell)$ and $\Gcal_{\Q_\ell}$ is equidimensional.

Our next lemma is an ad hoc version of the implicit function theorem for schemes over $\Z_\ell$:

\begin{lemma}\label{lem:BijectionOnl^mPoints}
	Let $\ell$ be a prime number, $0 \leq d \leq n$ and let $A_{d+1}, \ldots, A_n$ be elements of the ring $\Z_\ell[X_1, \ldots, X_n]$ such that for every $d+1\leq i \leq n$, we have $A_i - X_i \in \ell \Z_\ell[X_1, \ldots, X_n]$. Denote by $S'$ the scheme $\Spec \Z_\ell[X_1, \ldots, X_n] / \langle A_{d+1}, \ldots, A_{n} \rangle$ and by $p\colon S' \ra \A^{d}_{\Z_\ell}$ the projection given by the first $d$ coordinates.
	
For all integers $m \geq 1$, the base-change of $p$ to $\Z/\ell^m\Z$ is an isomorphism, hence induces a bijection $p:S'(\Z/\ell^m\Z) \to \A^d(\Z/\ell^m\Z)$ on $\Z/\ell^m \Z$-points.	
	Consequently, for all integers $m \geq m' \geq 1$, the cardinality of any fiber of the natural map
	\[
	\pi_{m,m'} : S'(\ell^m) \ra S' (\ell^{m'})
	\]
	is $\ell^{d (m-m')}$, and $\pi_{m'} : S'(\Z_\ell) \ra S'(\ell^{m'})$ is surjective for all $m' \geq 1$.
\end{lemma}
\begin{proof}
Define $B=\Z/\ell^m\Z[X_1,\ldots,X_d]$ and $C=B[X_{d+1}, \ldots, X_n] / \langle A_{d+1}, \ldots, A_n \rangle$. So $\A^n_{\Z/\ell^m\Z} = \Spec B$ and $S'_{\Z/\ell^m\Z} = \Spec C$, and $p_{\Z/\ell^m \Z}$ is induced by the natural homomorphism $B \to C$. The hypothesis $A_i-X_i \in \ell\Z_\ell[X_1,\ldots,X_n]$ implies that the determinant of the Jacobian matrix $\left( \frac{\partial A_i}{\partial X_i} \right)_{i,j=d+1,\ldots,n}$ reduces to 1 modulo $\ell$. Since every element of $C$ congruent to $1$ modulo $\ell$ is a unit in $C$, \cite[Corollary 3.16]{Milne}, or equivalently \cite[\href{https://stacks.math.columbia.edu/tag/02GU}{Lemma 02GU}]{stacks-project}, yields that the map $\Spec(C) \to \Spec(B)$ is étale. Furthermore, it is of degree 1, because this can be tested after tensoring with $\F_\ell$, and the given extension $B \hookrightarrow C$ induces an isomorphism $B \otimes \F_\ell \cong C \otimes \F_\ell$ since $A_i \bmod \ell = X_i$ for all $i=d+1,\ldots,n$. As an étale map of degree 1 is an isomorphism, this concludes the proof of the first statement in the lemma. The other statements follow immediately from the properties of $\A^d$.
\end{proof}

Before stating and proving our main counting lemma we need one more fact that links the cardinality of the fibers of certain reduction maps with suitable derivations:
\begin{lemma}\label{lem:Kaehler}
Let $R$ be a ring, $X = \Spec A$ be an $R$-scheme and $\varepsilon : A \ra R$ be a section.
Let $I$ be an ideal of $R$ and $m, m'$ be positive integers with $m'<m \leq 2m'$. Let $\pi_{m'} : R \ra R/I^{m'}$ be the canonical projection.  
The set of points of $X(R/I^m)$ above $\pi_{m'} \circ \varepsilon$ is in bijection with 
$\Hom_{R}(\Omega^{1}_{A/R} \otimes_\varepsilon R,I^{m'}/I^m)$.
\end{lemma}

\begin{proof}
We write $\varepsilon_m$ for $\pi_{m} \circ \varepsilon$ and similarly for $\varepsilon_{m'}$, and denote by $\pi_{m,m'}$ the canonical map $R/I^m \to R/I^{m'}$.
A point of $X(R/I^m)$ above $\varepsilon_{m'}$ is a homomorphism of $R$-algebras $\varphi: A \ra R/I^m$ such that $\pi_{m,m'} \circ \varphi = \varepsilon_{m'}$. Consider the morphism of $R$-modules
\[
 \theta := \varphi - \varepsilon_{m} : A \ra I^{m'} / I^m.
\]
As $\varphi$ is a ring morphism, for all $a,b \in A$ we have 
\[
\begin{aligned}
 \theta(ab) & = \varphi(ab) - \varepsilon_m(ab) \\
 & = \varphi(a)\varphi(b) - \varepsilon_m(a)\varepsilon_m(b) \\
 & = (\theta(a)+\varepsilon_m(a))(\theta(b)+\varepsilon_m(b)) - \varepsilon_m(a)\varepsilon_m(b) \\
 & =
 \varepsilon_m(a) \theta(b) + \varepsilon_m(b) \theta(a), 
 \end{aligned}
\]
because $\theta(a) \theta(b)$ belongs to $I^{2m'} / I^m$, so it is 0 in $R/I^m$ by the assumption $2m' \geq m$.
In other words, $\theta$ is an $A$-linear derivation (with the $A$-module structure on $I^{m'}/I^m$ given by $\varepsilon_{m}$). Conversely, the same computation shows that every $A$-linear derivation $\theta : A \to I^{m'}/I^m$ provides a point $\varphi \in X(R/I^m)$ above $\varepsilon_m$. By the defining property of the Kähler differentials we have an isomorphism
\[
 \operatorname{Der}_A(A,I^{m'}/I^m) \cong \Hom_A(\Omega^1_{A/R},I^{m'}/I^m),
\]
but this latter space is isomorphic to
\[
 \Hom_R(\Omega^1_{A/R} \otimes_\varepsilon R,I^{m'}/I^m),
\]
as can be checked directly because $I^{m'}/I^m$ inherits its $A$-module structure from $\varepsilon$. 
\end{proof} 
 
 \begin{lemma}[Main counting lemma]\label{lem:MainCountingLemma}
 	Let $\ell$ be a prime number, $S$ be a closed subscheme of $\A^n_{\Z_\ell}$, and $\varepsilon : \Spec \Z_\ell \ra S$ be a section of the structure morphism. For positive integers $m \geq m'$, denote by $\pi_{m'} : S(\Z_\ell) \to S(\Z/\ell^{m'}\Z)$ and $\pi_{m, m'} : S(\Z/\ell^m\Z) \to S(\Z/\ell^{m'}\Z)$ the canonical reduction maps modulo $\ell^{m'}$, as in the beginning of §\ref{ss: reduction to point counting}. 	
 	Assume that, for some non-negative integers $d$ and $e$, the $\Z_\ell$-module $\ell^e \varepsilon^* \Omega^1_{S/\Z_\ell}$ is generated by a set of cardinality at most $d$. 	
\begin{romanenum}
\item \label{lem:MainCountingLemma i}
For all integers $m \geq m'> e$,
\[
	\left| \pi_{m,m'}^{-1} (\pi_{m'}(\varepsilon)) \right| \leq  \ell^{n(e+1)} \ell^{d(m-m')}.
\]
\item \label{lem:MainCountingLemma ii}
If $e=0$ and $0 <m' \leq m \leq 2m'$, then
\[
\left| \pi_{m,m'}^{-1}  (\pi_{m'}(\varepsilon)) \right| \leq \ell^{d(m-m')}.
\]
\end{romanenum} 
 \end{lemma}

\begin{proof} Let $I \subseteq \Z_\ell[X_1,\ldots,X_n]$ be the ideal defining the subscheme $S$ of $\A^n_{\Z_\ell}$.

(\ref{lem:MainCountingLemma i}) Assume first that $m>m'+e+1$. 
	To begin with, $\varepsilon$ is a point of $S(\Z_\ell) \subseteq \A^n_{\Z_\ell}(\Z_\ell)$, hence corresponds to a ring morphism $\Z_\ell [X_1, \ldots, X_n] \ra \Z_\ell$. Up to translating the subscheme $S$, one can assume that $\varepsilon$ corresponds to the evaluation of all the $X_i$ at $0$.
	In the following, for a polynomial $P \in \Z_\ell[X_1, \ldots, X_n]$, we write $\dd P$ for its differential at 0, i.e.
	\[
	\dd P = \sum_{i=1}^n \frac{\partial P}{\partial X_i}(0, \ldots, 0) \, \dd X_i.
	\]
If $I$ is generated by polynomials $P_1, \ldots, P_r$, by the fundamental exact sequences of the Kähler differentials we have
	\[
	M := \varepsilon^* \Omega^1_{S/\Z_\ell} \cong \left( \bigoplus_{i=1}^n \Z_\ell \dd X_i \right) / \langle \dd P_j, \, 1 \leq j \leq r \rangle.
	\]
By the Smith normal form over the DVR $\Z_\ell$, up to a $\Z_\ell$-linear change of variables we can choose $X_1, \ldots, X_n$ in such a way that 
	\[
	\langle \dd P_1, \ldots, \dd P_r \rangle = \bigoplus_{i=1}^{n} \ell^{e_i} \Z_\ell \, \dd X_i
	\]
for suitable $e_1 \geq \ldots \geq e_{n}$ (with $e_i \in \mathbb{N} \cup \{+\infty\}$, where we set $\ell^{+ \infty} = 0$ by convention). With this choice of coordinates we have 
\[
M = \bigoplus_{i=1}^n \frac{\Z_\ell}{\ell^{e_i}\Z_\ell} \, \dd X_i.
\]
As the $e_i$ are decreasing and $\dim _{\F_\ell} \ell^e M/\ell^{e+1}M \leq d$, we have $e \geq e_i$ for $i=d+1, \ldots, n$.
	Choose polynomials $Q_1, \ldots, Q_{n}$ in $I$ such that $\dd Q_i= \ell^{\max(e_i,e)} \, \dd X_i$ for all $i \in \{1, \ldots, n\}$.  Recall that $\varepsilon$ corresponds to the ring morphism $\Z_\ell[X_1,\ldots,X_n]\to \Z_\ell$ evaluating a polynomial at $(0, \ldots, 0)$. As $\varepsilon$ is a $\Z_\ell$-point of $S=\Spec \Z_\ell[X_1,\ldots,X_n]/I$, the kernel of $\varepsilon$ must contain $I$, so every element of $I$ vanishes at $0 = (0, \ldots, 0)$. Thus $Q_i(0)=0$ and we have $Q_i= \ell^{e} X_i + R_i$ with $R_i$ a sum of homogeneous polynomials of degree $\geq 2$. For every $i \in \{d+1, \ldots, n\}$ we now consider the polynomial
\[
 \widetilde{Q}_i := \frac{1}{\ell^{2e+1}} Q_i(\ell^{e+1} X_1, \ldots, \ell^{e+1} X_n) \in \Z_\ell[X_1, \ldots, X_n].
\]
Writing $Q_i= \ell^{e} X_i + R_i$ as above, one obtains that $\widetilde{Q}_i$ is of the form
\[
 \widetilde{Q}_i = X_i + \ell \widetilde{R}_i(X_1, \ldots, X_n), \quad \widetilde{R}_i \in \Z_\ell[X_1, \ldots, X_n].
\]

We apply Lemma \ref{lem:BijectionOnl^mPoints} with $A_i = \widetilde{Q}_i$ for every $i \in \{d+1, \ldots, n \}$. In particular, we let $S'$ denote the $\Z_\ell$-scheme defined in $\A^n_{\Z_\ell}$ by the ideal $\langle \widetilde{Q}_{d+1}, \ldots, \widetilde{Q}_n \rangle$.

Let $x = (x_1, \ldots, x_n) \in S(\ell^m) \subseteq \A^n_{\Z_\ell}(\ell^m) = (\Z/\ell^m\Z)^n$ be a point which is zero modulo $\ell^{m'}$. Since $m'>e$ by assumption, we can write $x = \ell^{e+1} y$ for some $y \in (\ell^{m'-e-1}\Z/\ell^{m} \Z)^n$, so that for every $i >d$ we have
\[
 0 \equiv Q_i(x) \equiv Q_i(\ell^{e+1}y) \equiv \ell^{2e+1} \widetilde{Q}_i(y) \pmod{\ell^m}.
\]
This implies that $\widetilde{Q}_i(y) \in \ell^{m-2e-1} \Z / \ell^m \Z$ for all $i=d+1, \ldots, n$.
Thus, $y' := \pi_{m,m-2e-1}(y) \in (\Z/\ell^{m-2e-1}\Z)^n$ gives a point in $S'(\Z/\ell^{m-2e-1} \Z)$ which is zero modulo $\ell^{m'-e-1}$. By Lemma \ref{lem:BijectionOnl^mPoints}, the cardinality of any fibre of the map
\[
\pi'_{m-2e-1, m'-e} : S'(\ell^{m-2e-1}) \to S'(\ell^{m'-e})
\]
is $\ell^{d(m-m'-e-1)}$; note that the lemma applies, since $m'-e$ is strictly positive. The image $y'' = \pi'_{m-2e-1, m'-e}(y')$ of $y'$ in $S'(\ell^{m'-e})$ is in particular a point of $\A^n(\ell^{m'-e})$ which is zero modulo $\ell^{m'-e-1}$, so there are at most $\ell^n$ possibilities for $y''$. The point $y'$ lies in the fibre of $\pi'_{m-2e-1, m'-e}$ over $y''$, so for each $y''$ there are at most $\ell^{d(m-m'-e-1)}$ such $y'$, hence at most $\ell^n \ell^{d(m-m'-e-1)}$ possibilities for $y'$ in total. Finally, given any such $y' \in (\Z/\ell^{m-2e-1}\Z)^n$, there are $\ell^{n(2e+1)}$ possible lifts $y \in (\Z/\ell^{m} \Z)^n$ of $y'$. Since $x=\ell^{e+1}y$, the value of $x$ is determined by $y \bmod \ell^{m-e-1}$, so for each $y'$ there are at most $\ell^{en}$ possibilities for $x$. This gives
\[
 \left| \pi_{m,m'}^{-1} (\varepsilon \! \! \mod m') \right| \leq \ell^{n(e+1)} \ell^{d(m-m'-e-1)}, 
\]
a slightly better bound than claimed in the statement. 
Finally, for the case $m \leq m' +e+1$, we simply consider the embedding of $S$ into $\A^n_{\Z_\ell}$. Via this embedding, a point in $S(\ell^m)$ that reduces to $\varepsilon \bmod m'$ in $S(\ell^{m'})$ is in particular a point of $\A^n(\ell^m)$ that reduces to $(0, \ldots, 0)$ in $\A^n(\ell^{m'})$. It is clear that there are at most $\ell^{n(m-m')}$ such points, and $\ell^{n(m-m')} \leq \ell^{(e+1)n}$, which proves the bound in this case.

(\ref{lem:MainCountingLemma ii}) This is a consequence of Lemma \ref{lem:Kaehler}. Indeed, we have $e=0$ by assumption, so we know that $\varepsilon^* \Omega^1_{S/\Z_\ell}$ is generated by at most $d$ elements over $\Z_\ell$, and therefore the cardinality of $\Hom_{\Z_\ell}(\varepsilon^* \Omega^1_{S/\Z_\ell},\ell^{m'} \Z / \ell^{m} \Z)$ is at most $\ell^{d(m-m')}$.
\end{proof}

We can now establish the following for group schemes over $\Z_\ell$:

\begin{prop}\label{prop:count}
Let $\ell$ be a prime number, $M$ be a free $\Z_\ell$-module of finite rank $n$, and $\Gcal$ be a linear subgroup scheme of $\GL_{M}$ with unit section $\varepsilon \in \Gcal(\Z_\ell)$. Suppose that $e, d \in \N$ are such that $\ell^e \varepsilon^* \Omega^1_{\Gcal/\Z_\ell}$ is generated by at most $d$ elements as a $\Z_\ell$-module, with $d \leq n^2$. There is a constant $C(n,d,e,\ell)$ such that for all integers $m \geq m' >0$ the following holds:

 \[
  \left| \Ker \left(\pi_{m,m'}: \Gcal(\ell^m ) \ra \Gcal (\ell^{m'}) \right) \right| \leq C(n,d,e,\ell) \ell^{d(m-m')}.
 \]
Moreover, if $e=0$ (e.g.,~if $\Gcal$ is smooth), the result holds with $C(n,d,0,\ell)=1$.
\end{prop}

\begin{proof}
Lemma \ref{lem:MainCountingLemma} directly implies the result for $m \geq m' >e$. In the general case, we use that given two group homomorphisms $f:G_1 \ra G_2$ and $g:G_2 \ra G_3$ between finite groups we have $\left| \Ker(g \circ f) \right| \leq \left| \Ker(g)\right|  \cdot \left| \Ker(f) \right|$. Hence, if $m > e \geq m'$ we have
 \begin{eqnarray*}
  \left| \Ker \pi_{m,m'} \right| & \leq & \left| \Ker \pi_{m,e+1} \right| \cdot  \left| \Ker \pi_{e+1,m'} \right|  \leq C(n,d,e,\ell) \ell^{d(m-e-1)} \ell^{(e+1-m')n^2}  \\
  & \leq & \left(C(n,d,e,\ell)\ell^{(e+1)n^2} \right) \ell^{d(m-m')},
 \end{eqnarray*}
 where we have used the trivial bound 
 \[
 \left| \Ker \pi_{e+1,m'} \right|  \leq \left| \{ M \in \operatorname{M}_{n}(\Z/\ell^{e+1}\Z) : M \equiv I_n \pmod{\ell^{m'}} \} \right| = \ell^{n^2(e+1-m')}
 \]
 and the fact that $d \leq n^2$.
 Finally, if $m \leq e$, we similarly bound $\left| \Gcal(\ell^m) \right|$ by $\ell^{e n^2}$, which is enough for our purposes.
  
In the case $e=0$, the result follows by induction from part (\ref{lem:MainCountingLemma ii}) of Lemma \ref{lem:MainCountingLemma}.

\end{proof}

\subsection{Conclusion of the proof of Proposition \ref{propkey}}
Let $\varepsilon$ be the unit section of $\Gcal$. 
By Proposition \ref{prop:er}, there exists an integer $e_r \geq 0$ such that, for every saturated submodule $\mathcal{W}$ of $M$ of rank $r$, the $\Z_\ell$-module $\ell^{e_r} \varepsilon^* \Omega^1_{\Gcal_{\Wcal}/\Z_\ell}$ is generated by at most $d_r(G) \leq \dim G \leq \dim \GL_n = n^2$ elements. For $m' \geq 1$, Proposition \ref{prop:count} immediately implies the desired inequality with constant given by $C(n, d_r(G), e_r, \ell)$. For $m'=0$ we then have
\[
\begin{aligned}
\left| \Fix_\Gcal(\Wcal)(\ell^m) \right| & \leq  \left| \Ker \left( \pi_{m,1}: \Fix_\Gcal(\Wcal)(\ell^m) \ra \Fix_\Gcal(\Wcal)(\ell)\right) \right| \cdot \left| \Fix_\Gcal(\Wcal)(\ell) \right|
  \\ & \leq C(n, d_r(G), e_r, \ell) \ell^{d_r(G) (m-1)} \cdot \left| \Gcal(\ell) \right|
    \\ & \leq C(n, d_r(G), e_r, \ell) \ell^{n^2-d_r(G)} \cdot \ell^{d_r(G) m}.
  \end{aligned}
  \]

\section{An expression for $\beta_A$} \label{S:beta bounds}

Fix a nonzero abelian variety $A$ over a number field $K$.    Recall that $\beta_A$ is the infimum of all real numbers $\beta$ for which there is a constant $C$, depending only on $A$ and $\beta$, such that the inequality $|A(L)_{\tors}|\leq C\cdot [L:K]^\beta$ holds for all finite extensions $L/K$.  

\begin{lemma} \label{L:same beta}
For any finite extension $K'/K$, we have $\beta_{A_{K'}}=\beta_A$.
\end{lemma}
\begin{proof}
The inequality $\beta_{A_{K'}} \leq \beta_{A}$ is trivial by considering extensions $L$ of $K'$ as extensions of $K$.  We now prove the opposite inequality.   For a finite extension $L/K$, set $L'=L\cdot K'$.  For any $\varepsilon>0$, we have 
\[
|A(L)_{\tors}|\leq |A(L')_{\tors}| \ll_{A,K'} [L':K']^{\beta_{A_{K'}}+\varepsilon} \leq [K':K]^{\beta_{A_{K'}}+\varepsilon} [L:K]^{\beta_{A_{K'}}+\varepsilon} \ll_{A,\varepsilon, K'} [L:K]^{\beta_{A_{K'}}+\varepsilon}.
\]  
Therefore, $\beta_A \leq \beta_{A_{K'}}$ by the minimality in the definition of $\beta_A$.  
\end{proof}

The following describes $\beta_A$ in terms of the $\ell$-adic monodromy groups of $G_{A,\ell}$.

\begin{thm} \label{T:main unconditional}
We have
\[
\beta_A= \max_\ell \gamma_{A,\ell},
\]
where the maximum is over all primes $\ell$.
\end{thm}
\begin{proof}
Define $\xi_A:=\max_\ell \gamma_{A,\ell}$; the maximum exists since the numerators and denominators of the $\gamma_{A,\ell}$ are bounded.  From Lemma~\ref{L:same beta}, $\beta_A$ is unchanged if we replace $A$ by a base extension by a finite extension of $K$.   The value $\xi_A$ is also unchanged if we replace $A$ by such a base extension.   So by Proposition~\ref{P:independence}, we may assume that the representations $\{\rho_{A,\ell^\infty}\}_\ell$ are independent.  We may also assume that the groups $G_{A,\ell}$ are connected by Proposition~\ref{P:connected}.

 We first prove $\beta_A \leq \xi_A$.  Take any finite extension $L/K$ and let $J$ be the set of primes that divide $|A(L)_{\tors}|$.   For each $\ell\in J$, define the field $L_\ell:=K(A(L)[\ell^\infty])$. By the independence of the representations $\rho_{A,\ell^\infty}$, we have $\prod_{\ell\in J} [L_\ell: K] \leq [L:K]$.   By Theorems ~\ref{T:ell-adic version, large ell} and \ref{T: small primes}, there is a constant $C>0$ depending only on $A$ such that  
  \[
  |A(L)[\ell^\infty]| = |A(L_\ell)[\ell^\infty]| \leq C\cdot [L_\ell:K]^{\gamma_{A,\ell}} \leq C\cdot [L_\ell:K]^{\xi_{A}}
  \] 
  for all $\ell \in J$.    Taking the product over all $\ell\in J$, we find that
\begin{align} \label{E:no epsilon}
|A(L)_{\tors}|=\prod_{\ell\in J} |A(L)[\ell^\infty]| \leq C^{|J|} \big(\prod_{\ell\in J} [L_\ell:K] \big)^{\xi_A} \leq C^{|J|} [L:K]^{\xi_A}.
\end{align}
Take any $\varepsilon>0$ and set $\delta:=1-1/(1+\varepsilon/\xi_A)$; we have $0<\delta<1$.   We have 
\[
C^{|J|}  = \prod_{\ell\in J} C \ll_{A,\varepsilon} \prod_{\ell \in J} \ell^\delta \leq |A(L)_{\tors}|^{\delta}, 
\]
where the first inequality uses that $C\leq \ell^\delta$ for all primes $\ell\gg_{A,\varepsilon} 1$. Using (\ref{E:no epsilon}), this implies that $|A(L)_{\tors}| \ll_{A,\varepsilon} |A(L)_{\tors}|^{\delta} \cdot [L:K]^{\xi_A}$.  Therefore,
\[
|A(L)_{\tors}| \ll_{A,\varepsilon} [L:K]^{\xi_A/(1-\delta)} = [L:K]^{\xi_A+\varepsilon},
\]
where the equality uses our choice of $\delta$.  Since $L/K$ and $\varepsilon>0$ were arbitrary, this implies that $\beta_A\leq\xi_A$.\\

We now prove $\xi_A\leq \beta_A$.   Fix a prime $\ell$ with $\gamma_{A,\ell}=\xi_A$ and set $G:=G_{A,\ell}$.  There is a nonzero subspace $W\subseteq V_\ell(A)$ such that $\gamma_{A,\ell} \cdot (\dim G - \dim G_W) = \dim W$.   By choosing a $\ZZ_\ell$-basis of $T_\ell(A)$, we can view $\calG_{A,\ell}$ as a subgroup of $\GL_{2g,\ZZ_\ell}$ and $W$ as a subspace of $\QQ_\ell^{2g}$.      Define the groups $\calH_0:=G(\QQ_\ell) \cap \GL_{2g}(\ZZ_\ell) =\calG_{A,\ell}(\ZZ_\ell)$ and $\calH:=G_W(\QQ_\ell) \cap \GL_{2g}(\ZZ_\ell) \subseteq \calG_{A,\ell}(\ZZ_\ell)$.

Take any integer $i\geq 1$.  By Th\'eor\`eme~9 of \cite{MR644559}, we have $|\calH_0(\ell^i)|  \asymp_{A,\ell} \ell^{i \dim G}$ and $|\calH(\ell^i)| \asymp_{A,W,\ell} \ell^{i \dim G_W}$ with notation as in \S\ref{SS:lie algebras and filtrations}.  Let $L_i$ be the subfield of $\Kbar$ fixed by the $\sigma \in \Gal_K$ for which $\rho_{A,\ell^i}(\sigma)$ lies in $\calH(\ell^i)$.    Using Theorem~\ref{T:finite index}, we find that $[L_i:K] \asymp_{A,W,\ell} \ell^{i (\dim G - \dim G_W)}= \ell^{i \dim W \cdot \gamma_{A,\ell}^{-1}}$.

Define $\calW:=W \cap \ZZ_\ell^{2g}$.  The $\ZZ_\ell$-module $\calW$ has rank equal to the dimension of $W$.  The group $\calW/\ell^i \calW \subseteq (\ZZ/\ell^i\ZZ)^{2g}$ is fixed by $\calH(\ell^i)$.  Therefore, $A(L_i)$ has a subgroup of order $|\calW/\ell^i \calW|=\ell^{i \dim W}$.   

Take any $\varepsilon>0$.  By the definition of $\beta_A$, we have $|A(L_i)_{\tors}| \ll_{A,\varepsilon} [L_i:K]^{\beta_A+\varepsilon}$ and hence
\[
\ell^{i \dim W} \ll_{A,\varepsilon} [L_i:K]^{\beta_A+\varepsilon} \ll_{A,W,\ell} \ell^{i \dim W \cdot \gamma_{A,\ell}^{-1}\cdot  (\beta_A+\varepsilon)}.
\]
Since this holds for all $i\geq 1$, we must have $\gamma_{A,\ell}^{-1} \cdot (\beta_A+\varepsilon) \geq 1$.  Since $\varepsilon>0$ was arbitrary, we have $\beta_A \geq \gamma_{A,\ell} = \xi_A$.
\end{proof}

\subsection{Proof of Theorem~\ref{T:main}}\label{SS:main proof i}

Since we are assuming the Mumford--Tate conjecture for $A$, the theorem follows from Theorem~\ref{T:main unconditional} and Lemma~\ref{L:gamma connection}(\ref{L:gamma connection ii}).

\subsection{Proof of Theorem~\ref{T:equivalent version}}\label{SS:main proof ii}

We have $\End(V)^{G_A} \cong \End(A_{\Kbar})\otimes_\ZZ \QQ$  by Lemma~\ref{L:MT basics}(\ref{L:MT basics ii}).  Since $A$ is geometrically simple, the ring $\End(V)^{G_A}$ is a division algebra.   Therefore, $V$ is an irreducible representation of $G_A$ and hence $\gamma_A = 2\dim A / \dim G_A$.    By Theorem~\ref{T:main unconditional} and  Lemma~\ref{L:gamma connection}(\ref{L:gamma connection i}), we have  inequalities
\[
\beta_A \geq \max_\ell   \frac{ 2 \dim A}{\dim G_{A,\ell}^\circ } \geq \gamma_A = \frac{ 2 \dim A}{\dim G_{A} }.
\] 

Suppose that $\beta_A=\gamma_A$.  By the above inequalities, we must have $\dim G_{A,\ell}^\circ = \dim G_A$ for some prime $\ell$.  By Proposition~\ref{P:partial MTC}(\ref{P:partial MTC i}) and the equality of dimensions, we deduce that $G_{A,\ell}^\circ= (G_A)_{\QQ_\ell}$.    The Mumford--Tate conjecture for $A$ then follows from Proposition~\ref{P:partial MTC}(\ref{P:partial MTC ii}).

The other implication follows directly from Theorem~\ref{T:main}.

\subsection{Proof of Theorem~\ref{T:equivalent version 2}}{\label{SS:main proof iii}

Theorem \ref{T:main} shows that (\ref{T:equivalent version 2 a}) implies (\ref{T:equivalent version 2 b}).  We trivially have that (\ref{T:equivalent version 2 b}) implies (\ref{T:equivalent version 2 c}).  So it remains to show that (\ref{T:equivalent version 2 c}) implies (\ref{T:equivalent version 2 a}).  

Assume that (\ref{T:equivalent version 2 c}) holds and let $A$ be a nonzero abelian variety defined over a number field $K$. Due to the invariance of the Mumford--Tate conjecture under finite extensions of the ground field and under isogeny, enlarging $K$ if needed, we may assume that $A$ is isomorphic over $K$ to a product $B_1^{n_1} \times \cdots \times B_r^{n_r}$, where each $B_i$ is defined over $K$ and is geometrically simple. By assumption, the equality $\beta_{B_i} = \gamma_{B_i}$ holds for each $i=1, \ldots, r$, hence by Theorem \ref{T:equivalent version} the Mumford--Tate conjecture holds for each $B_i$. By \cite{JC}, this implies that the Mumford--Tate conjecture holds for $A = B_1^{n_1} \times \ldots \times B_r^{n_r}$. 

\section{Some remarks on a version without Mumford--Tate groups} \label{S:some remarks}

Let $A$ be a nonzero abelian variety defined over a number field $K$.  In this section, we will formulate a conjectural expression for $\beta_A$ that does not involve the Mumford--Tate group.  By Lemma~\ref{L:same beta}, we may assume (after extending the number field and replacing by an isogenous abelian variety) that $A$ is of the form $\prod_{i=1}^n A_i^{m_i}$ such that the abelian varieties $A_i/K$ are geometrically simple, pairwise geometrically nonisogenous, and have all their endomorphisms defined over $K$.   For each subset $I \subseteq \{1,\ldots, n\}$, define the abelian variety $A_I:= \prod_{i\in I} A_i^{m_i}$ over $K$. 
 For each prime $\ell$, let $\gamma_{A,\ell}$ be the constant defined in \S\ref{S:prime power version}.  

\begin{conj} \label{C:main}
For each prime $\ell$, we have 
\[
\gamma_{A,\ell} = \max_{\emptyset \neq I \subseteq \{1,\ldots, n\}}  \frac{ 2 \dim A_I}{\dim G_{A_I,\ell} }.
\]
\end{conj}

Note that  both sides in Conjecture~\ref{C:main} equal $\gamma_A$ when the Mumford--Tate conjecture for $A$ holds; this uses Lemma~\ref{L:gamma connection}.

\begin{thm}
If Conjecture~\ref{C:main} holds for $A$, then 
\[
\beta_A = \max_{\substack{\ell \text{ prime}\\ \emptyset \neq I \subseteq \{1,\ldots, n\}}}  \frac{ 2 \dim A_I}{\dim G_{A_I,\ell} }
\]
\end{thm}
\begin{proof}
This is an immediate consequence of Theorem~\ref{T:main unconditional} and Conjecture~\ref{C:main} for $A$.
\end{proof}

\begin{prop}
Fix a prime $\ell$.   Suppose there is an algebraic subgroup $H \subseteq G_A$ such that $H_{\QQ_\ell}$ and $G_{A,\ell}^\circ$ are conjugate in $(G_A)_{\QQ_\ell}$.    Then Conjecture \ref{C:main} holds for the prime $\ell$.
\end{prop}
\begin{proof}
 By Proposition~\ref{P:partial MTC}(\ref{P:partial MTC i}), we have $G_{A,\ell}^\circ\subseteq (G_A)_{\QQ_\ell}$ and hence our assumption that $H_{\QQ_\ell}$ and $G_{A,\ell}^\circ$ are conjugate in $(G_A)_{\QQ_\ell}$ makes sense.  Therefore, $\alpha(G_{A,\ell}^\circ)=\alpha(H_{\QQ_\ell})=\alpha(H)$, where the last equality uses Proposition~\ref{P:newer approach}(\ref{P:newer approach iii}).

With notation as in \S\ref{SS:MT def}, the isotypic decomposition of $V$ as a representation of $G_A$ is given by $V=\bigoplus_{i=1}^n V_i$, see Lemma~\ref{L:GA isotypical}.
 For a nonempty subset $I\subseteq \{1,\ldots,n\}$, define $V_I:=\bigoplus_{i\in I} V_i$.  For any nonempty subset $I\subseteq \{1,\ldots,n\}$, we find that $(H_{V_I})_{\QQ_\ell}=(H_{\QQ_\ell})_{V_I\otimes_\QQ \QQ_\ell}$ is conjugate in $(G_A)_{\QQ_\ell}$ to $(G_{A,\ell}^\circ)_{V_I\otimes_\QQ \QQ_\ell}= (G_{A,\ell}^\circ)_{V_\ell(A_I)}$.  So
 \[
 \frac{\dim H - \dim H_{V_I}}{\dim V_I} = \frac{\dim G_{A,\ell}^\circ - \dim (G_{A,\ell}^\circ )_{V_\ell(A_I)}}{\dim V_\ell(A_I)} = \frac{\dim G_{A_I,\ell}^\circ }{2\dim A_I},
 \]
 where the last equality uses that the kernel of the projection $G_{A,\ell}\to G_{A_I,\ell}$ is $(G_{A,\ell})_{V_\ell(A_I)}$.  In particular, 
 \begin{align}\label{E:H comp}
\min_{\emptyset \neq I \subseteq \{1,\ldots, n\}}  \frac{\dim H - \dim H_{V_I}}{\dim V_I} = \min_{\emptyset \neq I \subseteq \{1,\ldots, n\}}  \frac{\dim G_{A_I,\ell} }{ 2 \dim A_I}.
\end{align}

We claim that $V=\bigoplus_{i=1}^n V_i$ is also the isotypic decomposition of $V$ as a representation of $H$.   If the claim holds, then $\alpha(H)=\min_{I}  {\dim G_{A_I,\ell} }/({ 2 \dim A_I})$, where $I$ varies over the nonempty subsets of $\{1,\ldots, n\}$, by Proposition~\ref{P:newer approach}(\ref{P:newer approach ii}) and (\ref{E:H comp}).   Thus the proposition will follow from the claim since $\alpha(H)=\alpha(G_{A,\ell}^\circ)=\gamma_{A,\ell}^{-1}$.

The group $H$ is connected and reductive since $G_{A,\ell}^\circ$ has these properties and $H_{\QQ_\ell}\cong G_{A,\ell}^\circ$.  Since $V=\bigoplus_{i=1}^n V_i$ is the isotypic decomposition as a representation of $G_A$ and $H\subseteq G_A$, to prove the claim it suffices to show that $\End_\QQ(V)^{G_A}=\End_\QQ(V)^H$.  We have $\End_\QQ(V)^{G_A}\subseteq \End_\QQ(V)^H$ since $H\subseteq G_A$, so it suffices to show that they have the same dimension as $\QQ$-vector spaces.  

By Lemma~\ref{L:MT basics}(\ref{L:MT basics ii}), we have $\End_{\QQ}(V)^{G_A}= \End(A)\otimes_\ZZ \QQ$, where we are using that all the endomorphisms of $A_{\Kbar}$ are defined over $K$.  By Proposition~\ref{P:Faltings}(\ref{P:Faltings ii}) and our assumption that all the endomorphisms of $A_{\Kbar}$ are defined over $K$, we have $\End_{\QQ_\ell}(V_\ell(A))^{G_{A,\ell}^\circ} =\End(A) \otimes_\ZZ \Q_\ell$.  Therefore, 
\[
\End_\QQ(V)^H \otimes_\QQ \QQ_\ell = \End_{\QQ_\ell}(V\otimes_\QQ \QQ_\ell)^{H_{\QQ_\ell}} \cong \End_{\QQ_\ell}(V_\ell(A))^{G_{A,\ell}^\circ}  =\End(A) \otimes_\ZZ \Q_\ell.
\] 
So $\End_{\QQ}(V)^{G_A}$ and $\End_{\QQ}(V)^{H}$ both have the same dimension as $\End(A) \otimes_\ZZ \QQ$ as a $\QQ$-vector space and thus are equal.  This completes the proof of the claim and the proposition.
\end{proof}

We now prove Conjecture~\ref{C:main} for several abelian varieties.  In particular,  Conjecture~\ref{C:main} will hold whenever $\End(A_{\Kbar})=\ZZ$; this includes many cases for which the Mumford--Tate conjecture is unknown.

\begin{prop} \label{P:unconditional prop}
Suppose that $A$ is geometrically simple and that the center of the ring $\End(A_{\Kbar})$ is isomorphic to $\ZZ$.   Then Conjecture~\ref{C:main} holds, i.e.,  $\gamma_{A,\ell}={ 2 \dim A}/{\dim G_{A,\ell} }$.
\end{prop}
\begin{proof}
Take any prime $\ell$.  After suitably increasing the field $K$, we may assume that $\End(A_{\Kbar})=\End(A)$ and that the group $G_{A,\ell}$ is connected.  The ring $D:=\End(A)\otimes_\ZZ \QQ$ is a division algebra since $A$ is geometrically simple.   From our assumption on $\End(A_{\Kbar})$, we find that the division algebra $D$ has center $\QQ$.   Therefore,  $D\otimes_\QQ \QQ_\ell$ is a central simple algebra over $\QQ_\ell$. 
 
By Proposition~\ref{P:Faltings}(\ref{P:Faltings ii}), the natural map $D\otimes_\QQ \QQ_\ell  \to \End_{\QQ_\ell[\Gal_K]}(V_\ell(A))$
is an isomorphism of $\QQ_\ell$-algebras.  Therefore, $\End_{\QQ_\ell[\Gal_K]}(V_\ell(A))$ is a central simple algebra over $\QQ_\ell$.   

Denote by $V_\ell(A)=\bigoplus_{i=1}^n V_i$ the decomposition of the representation $V_\ell(A)$ of $G_{A,\ell}$ into isotypical components.     We have $\End_{\QQ_\ell[\Gal_K]}(V_\ell(A)) = \prod_{i=1}^n \End_{\QQ_\ell[\Gal_K]}(V_i)$.    Since $\End_{\QQ_\ell[\Gal_K]}(V_\ell(A))$ is a simple $\QQ_\ell$-algebra, we deduce that $n=1$.  
Since there is only one isotypic component, by Proposition \ref{P:newer approach}(\ref{P:newer approach ii}) we find that 
\[
\gamma_{A,\ell} = \frac{\dim V_\ell(A)}{\dim G_{A,\ell} - \dim (G_{A,\ell})_{V_\ell(A)}} = \frac{2\dim A}{\dim G_{A,\ell}}.  \qedhere
\]
\end{proof}

\begin{remark}
Let us briefly sketch why we are currently unable to extend the proof of Proposition~\ref{P:unconditional prop} to arbitrary $A$.   For simplicity, assume that $A$ is geometrically simple, that $\End(A_{\Kbar})=\End(A)$ and that $G_{A,\ell}$ is connected for all $\ell$.

Denote the center of $\End(A)\otimes_\ZZ \QQ$ by $E$; it is a number field.   We have $E_\ell:=E\otimes_\QQ \QQ_\ell = \prod_{\lambda|\ell} E_\lambda$, where $\lambda$ runs over the places of $E$ that divide $\ell$.   The natural actions of $\Gal_K$ and $E_\ell$ on $V_\ell(A)$ commute.   Therefore,   $V_\lambda:=V_\ell(A)\otimes_{E_\ell} E_\lambda$ is a $\QQ_\ell[\Gal_K]$-module that we can identify with a submodule of $V_\ell(A)$.   We have $V_\ell(A)=\bigoplus_{\lambda|\ell} V_\lambda$ and using the work of Faltings, one can show that this is the isotypic decomposition of $V_\ell(A)$ as a representation of $G_{A,\ell}$ and that $G_{A,\ell}$ is reductive.    Using Proposition \ref{P:newer approach}(\ref{P:newer approach ii}), we have
\begin{align} \label{E:calL max}
\gamma_{A,\ell} = \max_{ \mathcal{L} \neq \emptyset} \frac{\dim V_{\mathcal{L}}}{\dim G_{A,\ell} - \dim (G_{A,\ell})_{V_{\mathcal{L}}}},
\end{align}
where $V_{\mathcal{L}}:=\bigoplus_{\lambda\in \mathcal{L}} V_\lambda$ and $\mathcal{L}$ runs over the nonempty sets of places $\lambda$ of $E$ that divide $\ell$.    Conjecture~\ref{C:main} is equivalent to showing that the maximum in (\ref{E:calL max}) is obtained with $\mathcal{L}=\{\lambda:\lambda|\ell\}$; this is obvious in the case of Proposition~\ref{P:unconditional prop} where $E=\QQ$.  
\end{remark}

\section{Improvements on Masser's bound} \label{S:Masser improvement}
\label{sec: Masser bound}
Consider a nonzero abelian variety $A$ defined over a number field $K$.   From Masser \cite{masser-lettre}, we always have the bound $\beta_A \leq \dim A$.   The goal of this section is to prove the following, which describes when Masser's bound can be improved upon.

\begin{thm} \label{T:Masser improvement}
We have $\beta_A \leq \dim A$, with equality holding if and only if $A$ is isogenous over $\Kbar$ to a power of a CM elliptic curve.
\end{thm}

Before beginning the proof,  we will need to prove a few lemmas.   Fix any prime $\ell$.   We first give some equivalent conditions for an abelian variety to have complex multiplication.

\begin{lemma} \label{L:CM criteria}
Let $A$ be a nonzero abelian variety over $K$ that is isogenous over $\Kbar$ to a power of a geometrically simple abelian variety.  Then the following are equivalent:
\begin{alphenum}
\item \label{I:CM cond a}
$A$ has complex multiplication,
\item \label{I:CM cond b}
$G_A$ is a torus,
\item \label{I:CM cond c}
$G_{A,\ell}^\circ$ is a torus {for some, equivalently for every, prime $\ell$},
\item \label{I:CM cond d}
there is a representation $W \subseteq V_\ell(A) \otimes_{\QQ_\ell} \Qbar_\ell$ of $(G_{A,\ell}^\circ)_{\Qbar_\ell}$ of degree $1$ { for some, equivalently for every, prime $\ell$}. 
\end{alphenum}
\end{lemma}
\begin{proof}
There is no harm in replacing $A$ by its base change by a finite extension of $K$.   So we may assume that $A$ is isogenous to a power of a geometrically simple abelian variety $A_1/K$.  Observe that if one of the conditions (\ref{I:CM cond a})--(\ref{I:CM cond d}) holds for $A$, then the corresponding condition also holds for $A_1$.   So without loss of generality, we may assume that $A$ is geometrically simple.   After again replacing the field $K$ by a finite extension, we may further assume that $\End(A)=\End(A_{\Kbar})$ and that $G_{A,\ell}$ is connected.    Define $g:=\dim A$.

Define $D:=\End(A)\otimes_\ZZ \QQ$.    Since $A$ is geometrically simple, $D$ is a division algebra.  Denote the center of $D$ by $F$ and define the natural numbers $d:=[F:\QQ]$ and $e:=[D:F]^{1/2}$.   With notation as in \S\ref{SS:MT def}, we have $G_A \subseteq \GL_V$.  By Lemma~\ref{L:MT basics}(\ref{L:MT basics ii}), $D$ agrees with the subalgebra of $\End_\QQ(V)$ that commutes with $G_A$.   Since $D$ is a division algebra, $V$ is an irreducible representation of the reductive group $G_A$.    

Set $L:=\Qbar_\ell$.  We have an isomorphism
\[
D \otimes_\QQ L = D \otimes_F (F\otimes_\QQ L) \cong D \otimes_F (\prod_{\sigma} L)\cong  \prod_{\sigma} (D \otimes_{F,\sigma} L ) \cong M_e(L)^d
\]
of $L$-algebras, where the products are over the $d$ embeddings $F\hookrightarrow L$.  In particular, the $L$-subalgebra of $\End_L(V\otimes_\QQ L)$ that commutes with $(G_A)_L$ is isomorphic to $M_e(L)^d$.  Since $(G_A)_L$ is reductive, we must have an isomorphism 
\begin{align} \label{E:decomp irred 1}
V\otimes_\QQ L \cong U_1^e \oplus \cdots \oplus U_d^e 
\end{align}
of representations of $(G_A)_L$, where the $U_i$ are irreducible and pairwise nonisomorphic.  Since $V$ is an irreducible representation of $G_A$, the $U_i$ must all have the same dimension over $L$.  In particular, $\dim_L U_i = (\dim V)/(de)= 2g/(de)$.

From Proposition~\ref{P:Faltings}(\ref{P:Faltings ii}), we find that the $L$-subalgebra of $\End_L(V_\ell(A)\otimes_{\QQ_\ell} L)$ that commutes with $(G_{A,\ell})_L$ is isomorphic to $D\otimes_\QQ L \cong M_e(L)^d$.   So we have an isomorphism 
\begin{align} \label{E:decomp irred 2}
V_\ell(A)\otimes_{\QQ_\ell} L \cong W_1^e \oplus \cdots \oplus W_d^e 
\end{align}
of representations of the reductive group $(G_{A,\ell})_L$, where the $W_i$ are irreducible and pairwise nonisomorphic.    Using the comparison isomorphism $V\otimes_\QQ \QQ_\ell = V_\ell(A)$ and the inclusion $G_{A,\ell} \subseteq (G_{A,\ell})_{\QQ_\ell}$ from Proposition~\ref{P:partial MTC}(\ref{P:partial MTC i}), we conclude that (\ref{E:decomp irred 1}) and (\ref{E:decomp irred 2}) give the same decomposition of $V_\ell(A)\otimes_{\QQ_\ell} L$ into irreducible representations of $(G_{A,\ell})_L$. 

We first show that (\ref{I:CM cond a}) and (\ref{I:CM cond b}) are equivalent.   
First suppose that $A$ has complex multiplication.  Since $A$ is geometrically simple, $D=F$ is a number field of degree $2g$.    So $e=1$ and $d=2g$, and hence each $U_i$ has dimension $(2g)/de=1$.  Since the representation of $(G_A)_L$ on $U_1\oplus \cdots \oplus  U_d$ is faithful and all the $U_i$ have dimension $1$, we deduce that $(G_A)_L$, and hence also $G_A$, is commutative.  Therefore, the reductive group $G_A$ must be a torus which completes the proof that (\ref{I:CM cond a}) implies (\ref{I:CM cond b}).  Now suppose instead that $G_A$ is a torus.  Since $V$ is an irreducible representation of the torus $G_A$ and $L$ is algebraically closed, $V\otimes_\QQ L$ will be a direct sum of irreducible representations of $(G_A)_L$ each having dimension $1$ and multiplicity $1$.   So in this case, we will have $e=1$ and $\dim U_i=1$ for all $1\leq i \leq d$.  So $2g/d=2g/(de)=\dim U_i =1$ and hence $d=2g$.   Therefore, $D=F$ is a number field of degree $2g$ and hence $A$ has complex multiplication.   This completes the proof that (\ref{I:CM cond a}) and (\ref{I:CM cond b}) are equivalent.

That (\ref{I:CM cond b}) implies (\ref{I:CM cond c}) {  for every prime $\ell$} is a direct consequence of the inclusion $G_{A,\ell} \subseteq (G_{A})_{\QQ_\ell}$ {  from Proposition \ref{P:partial MTC}(\ref{P:partial MTC i})}.  That (\ref{I:CM cond c}) {  for a given (resp.~every) prime $\ell$} implies (\ref{I:CM cond d}) {  for the same (resp.~every) prime $\ell$} is immediate since the irreducible representations of a torus over an algebraically closed field all have degree $1$.

Finally assume that (\ref{I:CM cond d}) holds {  for some prime $\ell$}.  It suffices to show that (\ref{I:CM cond b}) holds.  From our decomposition (\ref{E:decomp irred 1}) into irreducibles, one of the irreducible summands $U_i$ must be isomorphic to $W$ as a representation of $(G_{A,\ell})_L$.  Since all the $U_i$ have the same dimension, we deduce that $\dim U_i= \dim W=1$ for all $1\leq i \leq d$.    The group $(G_A)_L$ must be commutative since it acts faithfully on a direct sum of degree $1$ representations.  Therefore, the reductive group $(G_A)_L$, and hence also $G_A$, must be a torus.  
\end{proof}

We now describe when the $\ell$-adic monodromy group of a nonzero abelian variety has smallest possible dimension.

\begin{lemma} \label{L:group dimension lower bound}
Let $B$ be a nonzero abelian variety over $K$.   Then $\dim G_{B,\ell} \geq 2$, with equality holding if and only if $B$ is isogenous over $\Kbar$ to a power of a CM elliptic curve.
\end{lemma}
\begin{proof}
After replacing $B$ by its base extension by some finite extension of $K$, we may assume that the group $G_{B,\ell}$ is connected.    

The group $G_{B,\ell}$ is reductive by Proposition~\ref{P:reductive}(\ref{P:reductive i}), so the quotient $S:=G_{B,\ell}/Z(G_{B,\ell})$ is semisimple, where $Z(G_{B,\ell})$ is the center of $G_{B,\ell}$.   If $S\neq 1$, then $\dim S\geq 3$ since the smallest dimension of a non-trivial semisimple group is 3 (realised by forms of $\SL_2$ and $\SL_2/\{\pm I\}$).  In particular, we have $\dim G_{B,\ell} \geq 3$ whenever $S\neq 1$.   So for the rest of the proof, we may assume that $S=1$; equivalently, $G_{B,\ell}$ is a torus.    By Lemma~\ref{L:CM criteria}, $G_{B}$ is also a torus and $B$ has complex multiplication.  We have $(G_B)_{\QQ_\ell} \cong G_{B,\ell}$ by \cite[Corollary 2.11]{UllmoYafaev}.   So it suffices to prove that $\dim G_B \geq 2$, with equality holding if and only if $B$ is isogenous over $\Kbar$ to a power of a CM elliptic curve.

After replacing $B$ by its base change by some finite extension of $K$, we may assume that $B$ is isogenous to a product $\prod_{i=1}^n B_i^{m_i}$, where the $B_i$ are abelian varieties over $K$ that are geometrically simple and pairwise nonisogenous.   

Take any $1\leq i \leq n$.   Since $B$ has complex multiplication, $B_i$ has complex multiplication and $G_{B_i}$ is a torus.   Since $G_{B_i}$ is a quotient of $G_B$, we have 
\[
\dim G_B \geq \dim G_{B_i} \geq 2+\log_2 \dim(B_i) \geq 2
\] 
by \cite[Equation (3.5)]{MR608640}, where the last inequality is an equality only if $B_i$ is an elliptic curve.   So $\dim G_B \geq 3$ if at least one of the $B_i$ is not an elliptic curve.    We can thus assume that each $B_i$ is a CM elliptic curve.

Suppose that $n\geq 2$.   Since $G_{B_1 \times B_2}$ is a quotient of $G_B$ and $B_1$ and $B_2$ are nonisogenous CM elliptic curves, we have 
\[
\dim G_B \geq \dim G_{B_1\times B_2} = 3
\]
by \cite[Corollary 3.9]{MR1731466}.  Finally, in the case $n=1$, i.e., $B$ is isogenous (over $\Kbar$) to a power of a CM elliptic curve $B_1$, we have $\dim G_B = \dim G_{B_1}=2$.
\end{proof}

We now consider bounds for $\gamma_{A,\ell}$ when $A$ has complex multiplication.  This case is easy to study because we know that the Mumford--Tate conjecture holds.

\begin{lemma}  \label{L:CM special case of Masser}
Let $A$ be a nonzero abelian variety over $K$ that has complex multiplication.   Then $\gamma_{A,\ell} \leq \dim A$, with equality holding if and only if $A$ is isogenous over $\Kbar$ to a power of a CM elliptic curve.
\end{lemma}
\begin{proof}
After replacing $A$ by its base change by a finite extension of $K$, we may assume $A$ is isogenous to a product $\prod_{i=1}^n A_i^{m_i}$, where the $A_i$ are abelian varieties over $K$ that are geometrically simple and are pairwise nonisomorphic. By definition we have 
\[
\gamma_A = \max_{\emptyset \neq I \subseteq \{1,\ldots, n\}}  \frac{ 2 \dim A_I}{\dim G_{A_I} },
\]
where $A_I:=\prod_{i\in I} A_i^{m_i}$.  Since $A$ has complex multiplication, so do all the $A_I$.    Since the Mumford--Tate conjecture holds for all abelian varieties with complex multiplication \cite{PohlmannMumfordTateforCM}, Lemma~\ref{L:gamma connection}(\ref{L:gamma connection ii}) implies that 
\[
\gamma_{A,\ell} = \max_{\emptyset \neq I \subseteq \{1,\ldots, n\}}  \frac{ 2 \dim A_I}{\dim G_{A_I,\ell}^\circ }.
\]
Take any nonempty $I \subseteq \{1,\ldots, n\}$.  By Lemma~\ref{L:group dimension lower bound}, we have 
\[
2 \dim A_I / \dim G_{A_I,\ell}^\circ  \leq (2 \dim A_I)/2 = \dim A_I \leq \dim A,
\]
with $2 \dim A_I / \dim G_{A_I,\ell}^\circ = \dim A$ if and only if $I=\{1,\ldots,n\}$ and $A_I$ is isogenous over $\Kbar$ to a power of a CM elliptic curve.   The lemma is now immediate after taking the maximum over all $I$.
\end{proof}

We now consider powers of geometrically simple abelian varieties.

\begin{lemma} \label{L:Masser bound, simple powers}
Let $A$ be a nonzero abelian variety over $K$ that over $\Kbar$ is isomorphic to a power of a simple abelian variety.  We have $\gamma_{A,\ell} \leq \dim A$, with equality holding if and only if $A$ is isogenous over $\Kbar$ to a power of a CM elliptic curve.
\end{lemma}
\begin{proof}
After replacing $A$ by its base change by a finite extension of $K$, we may assume that the group $G_{A,\ell}$ is connected and that $A$ is isogenous to a power of a geometrically simple abelian variety $A_1/K$.  By Proposition~\ref{P:newer approach}(\ref{P:newer approach ii}), we have 
\[
\gamma_{A,\ell} = \frac{\dim W}{\dim G_{A,\ell} - \dim (G_{A,\ell})_W}
\]
for some nonzero $\QQ_\ell$-subspace $W\subseteq V_\ell(A)$ that is the direct sum of isotypic components of $V_\ell(A)$ as a representation of $G_{A,\ell}$.  Define $\delta:=\dim G_{A,\ell} - \dim (G_{A,\ell})_W$.   If $\delta \geq 3$, then $\gamma_{A,\ell} = (\dim W)/\delta \leq (2\dim A)/3 < \dim A$.  So it remains to consider the cases where $\delta$ is $1$ or $2$.

Suppose that $\delta=2$.  We have $\gamma_{A,\ell} = (\dim W)/\delta = (\dim W) / 2 \leq (2\dim A)/2 = \dim A$.  Now suppose we have an equality $\gamma_{A,\ell}=\dim A$ and hence $\dim W =2\dim A$.  Therefore, $W=V_\ell(A)$ and we then have $(G_{A,\ell})_W=1$.   So $\dim A=\gamma_{A,\ell} = (2\dim A)/ (\dim G_{A,\ell}-0)$ and hence $\dim G_{A,\ell}=2$.  By Lemma~\ref{L:group dimension lower bound}, we deduce that $A$ is isogenous over $\Kbar$ to a power of a CM elliptic curve. 

Suppose that $\delta=1$.   We know that $G_{A,\ell}$ contains the group $\GG_m$ of homotheties of $V_\ell(A)$.   Since $G_{A,\ell}$ is connected and $\delta=1$, we have $G_{A,\ell} =\GG_m\cdot (G_{A,\ell})_W$.    In particular, $G_{A,\ell}$ acts on $W$ via homotheties.   Therefore, any $1$-dimensional subspace of $W$ is a representation of $G_{A,\ell}$.  By the equivalence of (\ref{I:CM cond a}) and (\ref{I:CM cond d}) in Lemma~\ref{L:CM criteria}, we deduce that $A$ has complex multiplication.  The lemma now follows from Lemma~\ref{L:CM special case of Masser}.
\end{proof}

We can now consider upper bounds of $\gamma_{A,\ell}$ for an arbitrary abelian variety.

\begin{thm} \label{T:Masser improvement gamma}
Let $A$ be a nonzero abelian variety over $K$.  We have $\gamma_{A,\ell} \leq \dim A$, with equality holding if and only if $A$ is isogenous over $\Kbar$ to a power of a CM elliptic curve.
\end{thm}
\begin{proof}
After replacing $A$ by its base change by a finite extension of $K$, we may assume that there is an isogeny $\varphi\colon A \to \prod_{i=1}^n A_i$, where each $A_i$ is a power of a geometrically simple abelian variety $B_i$ over $K$ and the $B_i$ are pairwise nonisogenous over $\overline{K}$.     We further assume that the group $G_{A,\ell}$, and hence all the groups $G_{A_i,\ell}$, are connected.  By Proposition~\ref{P:newer approach}(\ref{P:newer approach ii}), we have 
\[
\gamma_{A,\ell} = \frac{\dim W}{\dim G_{A,\ell} - \dim (G_{A,\ell})_W}
\]
for some nonzero $\QQ_\ell$-subspace $W\subseteq V_\ell(A)$ that is the direct sum of isotypic components of $V_\ell(A)$ as a representation of $G_{A,\ell}$.
 Our fixed isogeny $\varphi$ induces an isomorphism $V_\ell(A)=\oplus_{i=1}^n V_\ell(A_i)$ of representations of $G_{A,\ell}$.    We have $\Hom_{G_{A,\ell}}(V_\ell(A_i),V_\ell(A_j))=\Hom(A_i,A_j)\otimes_\ZZ \QQ_\ell = 0$ for all distinct $1\leq i,j \leq n$ by Proposition~\ref{P:Faltings}(\ref{P:Faltings iii}).  Therefore, $W=\oplus_{i=1}^n W_i$, where $W_i$ is a direct sum of isotypic components of $V_\ell(A_i)$ as a representation of $G_{A_i,\ell}$. 
 
 Let $I$ be the set of $i\in \{1,\ldots, n\}$ with $W_i\neq 0$.  For each $i\in I$, we have a natural homomorphism of groups $G_{A,\ell}/(G_{A,\ell})_W \twoheadrightarrow G_{A_i,\ell}/(G_{A_i,\ell})_{W_i}$ (we have used that $W_i$ and $W$ are direct sum of isotypic components to guarantee that we are taking a quotient by a normal subgroup).   In particular, 
 \[
 \dim G_{A,\ell} - \dim (G_{A,\ell})_W \geq \dim G_{A_i,\ell} - \dim (G_{A_i,\ell})_{W_i}.
 \]
Therefore,
\begin{align*}
\gamma_{A,\ell} 
&=  \sum_{i\in I} \frac{\dim W_i}{\dim G_{A,\ell} - \dim (G_{A,\ell})_W} 
\leq  \sum_{i\in I} \frac{\dim W_i}{\dim G_{A_i,\ell} - \dim (G_{A_i,\ell})_W} 
\leq \sum_{i \in I} \gamma_{A_i,\ell}.
\end{align*}
By Lemma~\ref{L:Masser bound, simple powers}, we have 
\[
\sum_{i\in I} \gamma_{A_i,\ell} \leq \sum_{i\in I } \dim A_i  \leq \dim A,
\]
with equalities holding if and only if $I=\{1,\ldots, n\}$ and $A_i$ is isogenous over $\Kbar$ to a power of a CM elliptic curve for all $1\leq i \leq n$.  In particular, we have show that $\gamma_{A,\ell} < \dim A$ whenever $A$ does not have complex multiplication.  The case where $A$ has complex multiplication has already been dealt with in Lemma~\ref{L:CM special case of Masser}.
\end{proof}

\begin{proof}[Proof of Theorem~\ref{T:Masser improvement}]
By Theorem~\ref{T:main unconditional}, we have $\beta_A=\gamma_{A,\ell}$ for some prime $\ell$.  The theorem is then an immediate consequence of Theorem~\ref{T:Masser improvement gamma}.
\end{proof}

\section{The extension generated by a single torsion point}\label{S:one torsion point}

In this section we prove the following result concerning the degrees of the extensions generated by a single torsion point.  In \S\ref{SS:proof in main d1}, we will use it to prove Theorem~\ref{T:main d1}.  We make use of the notation $d_1$ from \S\ref{ss: reduction to point counting}.

\begin{thm} \label{T:new main d1}
Let $A$ be a nonzero abelian variety defined over a number field $K$.  Define $\delta$ to be the minimum value of $\dim G_{A,\ell}^\circ - d_1(G_{A,\ell}^\circ)$ as we vary over all primes $\ell$.  Take any real number $\varepsilon>0$.  Then for any integer $n\geq 1$ and any point $P \in A(\Kbar)$ of order $n$, we have 
\[
[K(P):K] \gg_{A,\varepsilon}  n^{\delta-\varepsilon}.
\]
Moreover, $\delta$ is the largest real number with this property. 
\end{thm}

The structure of the proof is similar to that of Theorem \ref{T:main}: we reduce to the case when the order of $P$ is the power of a single prime $\ell$ and we establish the required inequality separately for each prime, with a constant which is independent of $\ell$ for $\ell \gg_{A} 1$.

\subsection{Some lemmas}

\begin{lemma} \label{L:new d1}
Fix a Noetherian integral domain $R$ and let $F$ be its fraction field.  Let $G$ be a reductive subgroup of $\GL_{n,R}$.  Take any integer $m\geq 0$.  Then there  is a unique reduced closed subscheme $Z_m$
of $\AA^n_R$ such that for any field $k$ that is an $R$-algebra  and any $w\in k^n$, we have $w\in Z_m(k)$ if and only if 
$\dim (G_k)_W \geq m$ where $W:=k\cdot w \subseteq k^n$.
\end{lemma}
\begin{proof}
The group $G$ is a smooth group scheme over $S:=\Spec R$.    There is a natural left action of $G$ on the $S$-scheme $X:=\AA^n_R$.   Define the morphisms $f\colon G\times_S X \to X\times_S X$ and $\Delta\colon X \to X\times_S X$ by $(g,x)\mapsto  (g\cdot x, x)$ and $x\mapsto (x,x)$, respectively.  We define $Y$ to be the $X$-scheme given by the cartesian product 
\[
\xymatrix{
Y \ar[r] \ar[d] &  X \ar[d]^\Delta \\
G\times_S X \ar[r]^f & X\times_S X.
}
\]
This defines an $X$-group scheme $Y$ that we may view as a closed $X$-group subscheme of $G\times_S X$, and $Y$ is of finite type over $X$;  for details see \cite[V.10.2]{SGA3-I}. The function 
\[
X\to \ZZ,\quad x\mapsto \dim Y_x,
\]
where $Y_x$ is the fiber of $Y$ over $x$, 
is upper semicontinuous as a consequence of Chevalley's semicontinuity theorem, cf.~\cite[VI\,B.4.1]{SGA3-I}.  Equivalently, for each integer $m\geq 0$, the set $Z_m$ of points $x\in X$ that satisfy $\dim Y_x \geq m$ can be viewed as a reduced closed subscheme of $X$. 

Now consider any field $k$ that is an $R$-algebra.  Take any nonzero $w \in k^n =X(k^n)$ and define the subspace $W:=k\cdot w \subseteq k^n$.   The fiber $Y_w$ is group scheme over $k$ that is isomorphic to the algebraic subgroup of $G_k$ that fixes $w$.  In particular, $Y_w \cong (G_k)_W$ and hence $\dim Y_w = \dim (G_k)_W$.  Therefore,  $w\in Z_m(k)$ if and only if $\dim (G_k)_W \geq m$.   

Finally, $Z_m$ as given in the statement of the lemma is unique since it is a reduced closed subscheme of $\AA^n_R$ for which we know the set $Z_m(k)$ for all fields $k$ that are $R$-algebras.
\end{proof}

We now give some basic properties of $d_1(G)$ when $G$ is defined over a field.

\begin{lemma} 
\label{L:new d1 field case}
Let $F$ be a field and fix an algebraic closure $\Fbar$. Let $G$ be a reductive subgroup of $\GL_V$ where $V$ is a finite dimensional vector space over $F$.
\begin{romanenum}
\item \label{L:new d1 field case i}
For any field extension $L$ of $F$, we have $d_1(G_{L}) \leq d_1(G_{\Fbar})$.
\item \label{L:new d1 field case ii}
For any algebraically closed field $L$ containing $F$, we have $d_1(G_L)=d_1(G_{\Fbar})$.
\end{romanenum}
\end{lemma}
\begin{proof}
After choosing a basis for $V$, we may assume that $G$ is a subgroup of $\GL_{n,F}$.  Fix notation as in Lemma~\ref{L:new d1} with $R:=F$.  Let $m_1\geq 0$ be the maximal integer for which the $F$-variety $Z_{m_1}-\{0\}$ is nonempty (note that $Z_m$ is empty whenever $m> \dim G$).    For any field extension $L/F$, the set $Z_m(L)-\{0\}$ is empty for all $m>m_1$.  Therefore, $d_1(G_L) \leq m_1$ by the description of the sets $Z_m(L)$ in Lemma~\ref{L:new d1}.

Now take any algebraic closed field $L$ containing $F$.  Then $m_1\geq 0$ is also the minimal integer for which the set $Z_{m_1}(L)-\{0\}$ is nonempty.  From our description of the sets $Z_m(L)$ in Lemma~\ref{L:new d1}, we deduce that $d_1(G_L)=m_1$.   In particular, when $L=\Fbar$ we have $d_1(G_{\Fbar})=m_1$.  The lemma is now immediate.
\end{proof}

The following shows that for $\ell$ large enough, $d_1$ of the special and generic fibers of the $\ZZ_\ell$-scheme $\calG_{A,\ell}$ will agree.

\begin{lemma} \label{L:d_1 agreement for large ell}
Let $A$ be a nonzero abelian variety defined over a number field $K$.   Assume that the groups $G_{A,\ell}$ are connected for all $\ell$.  Then $d_1((\calG_{A,\ell})_{\FF_\ell}) = d_1(G_{A,\ell})$ for all sufficiently large primes $\ell$.
\end{lemma}
\begin{proof}

By taking $\ell$ large enough, we may assume that $\calG_{A,\ell}$ is reductive by Proposition~\ref{P:reductive}(\ref{P:reductive ii}).   After choosing a basis for the $\ell$-adic Tate module, we shall view $\calG_{A,\ell}$ as a closed subgroup of $\GL_{2g,\ZZ_\ell}$.   Also by taking $\ell$ sufficiently large, we may assume that $(\calG_{A,\ell})_{\ZZ_\ell^\un}$ is conjugate in $\GL_{2g,\ZZ_\ell^\un}$ to $\calG_{\ZZ_\ell^\un}$ for one of the finitely many reductive groups $\calG \subseteq \GL_{2g,\ZZ}$ from Proposition~\ref{P:redux finiteness}.  Take any integer $0\leq m\leq (2g)^2$.  

With $R=\ZZ$ and $G=\calG$, let $Z_m$ be the subscheme of $\AA^{2g}_{\ZZ}$ from Lemma~\ref{L:new d1}.   Define $C_0:=Z_m$.  For $i\geq 0$, we recursively define $C_{i+1}$ to be the singular locus of $C_i$.  Note that the $C_i$ are reduced closed subschemes of $\AA^{2g}_\ZZ$.   Let $r\geq 0$ be the minimal integer for which the generic fiber of the $\ZZ$-scheme $C_{r}$ is empty.  There is a positive integer $N$ so that $(C_r)_{\ZZ[1/N]}$ is empty.   By taking $\ell$ large enough, we may assume that $\ell\nmid N$ (this uses that there are only finitely many $m$ and $\calG$ we are considering).  

For each $0\leq i < r$, define the $\ZZ$-scheme $U_i:=C_i-C_{i+1}$; it is clearly a regular scheme.   The generic fiber of $U_i$ is a variety over $\QQ$ that is regular and hence is smooth.   So there is a positive integer $N$ such that $(U_i)_{\ZZ[1/N]}$ is a smooth scheme over $\ZZ[1/N]$.     By taking $\ell$ large enough, we may assume that $(U_i)_{\ZZ_\ell}$ is smooth for all $0\leq i <r$.

With $R=\ZZ_\ell$ and $G=\calG_{A,\ell}$, let $Z_{m,\ell}$ be the reduced closed subscheme $Z_m$ of $\AA^{2g}_{\ZZ_\ell}$ from Lemma~\ref{L:new d1}.   Define $C_{0,\ell}:=Z_{m,\ell}$.  For $i\geq 0$, we recursively define $C_{i+1,\ell}$ to be the singular locus of $C_{i,\ell}$; they are closed subschemes of $\AA^{2g}_{\ZZ_\ell}$.    Using that singular loci and the construction of the scheme from Lemma~\ref{L:new d1} are both stable under base change,  we deduce that $(C_{i,\ell})_{\ZZ_\ell^\un}$ and $(C_i)_{\ZZ_\ell^\un}$ are isomorphic $\ZZ_\ell^\un$-schemes.  Since $(C_{r,\ell})_{\ZZ_\ell^\un}\cong (C_r)_{\ZZ_\ell^\un}$ is empty and the inclusion $\ZZ_\ell \subseteq \ZZ_\ell^\un$ is a faithfully flat ring map, we find that $C_{r,\ell}$ is also empty.   For $0\leq i < r$, define the $\ZZ_\ell$-scheme $U_{i,\ell}:=C_{i,\ell}-C_{i+1,\ell}$.  We have $U_{i,\ell}\cong (U_i)_{\ZZ_\ell}$ and hence $U_{i,\ell}$ is smooth.  By Hensel's lemma and the smoothness of $U_{i,\ell}$, the reduction modulo $\ell$ map $U_{i,\ell}(\ZZ_\ell)\to U_{i,\ell}(\FF_\ell)$ is surjective for all $1\leq i < r$.  Since $C_{r,\ell}$ is empty, we find that $Z_{m,\ell}(\FF_\ell)=\cup_{0\leq i < r} U_{i,\ell}(\FF_\ell)$ and hence $Z_{m,\ell}(\ZZ_\ell)\to Z_{m,\ell}(\FF_\ell)$ is surjective.  

So after taking $\ell$ sufficiently large, we have shown that  for all integers $m\geq 0$, the reduction modulo $\ell$ map $Z_{m,\ell}(\ZZ_\ell) \to Z_{m,\ell}(\FF_\ell)$ is surjective;  this is trivial for the excluded integers $m>(2g)^2$ since $Z_{m,\ell}$ will be empty.  In particular, if $Z_{m,\ell}(\FF_\ell)-\{0\}$ is nonempty, then $Z_{m,\ell}(\QQ_\ell)-\{0\}$ is nonempty as well.  Now suppose that  $Z_{m,\ell}(\QQ_\ell)$ contains a nonzero element $w$. We can scale $w$ by any nonzero element of $\QQ_\ell$ and it will still lie in $Z_{m,\ell}(\QQ_\ell)$.    So without loss of generality, we may assume that $w\in \ZZ_\ell^{2g}-\ell \ZZ_\ell^{2g}$.   We have $w\in Z_m(\ZZ_\ell)$ and its reduction modulo $\ell$ gives a nonzero element of $Z_m(\FF_\ell)$.   So by taking $\ell$ sufficiently large, we have shown that for all $m\geq 0$, $Z_{m,\ell}(\FF_\ell)-\{0\}$ is nonempty if and only if $Z_{m,\ell}(\QQ_\ell)-\{0\}$ is nonempty.   The lemma is now an immediate consequence of Lemma~\ref{L:new d1}.
\end{proof}

\begin{lemma} \label{L:generic smoothness of GW}
Let $G$ be a reductive subgroup of $\GL_{n,\ZZ[1/N]}$ for some positive integer $N$.   Then for all sufficiently large primes $\ell \nmid N$ and all one-dimensional subspaces $W$ of $\FFbar_\ell^n$, the group $(G_{\FFbar_\ell})_W$ is smooth.
\end{lemma}
\begin{proof}
Fix a generator $w$ of $W$.
The group $(G_{\FFbar_\ell})_W$ is defined inside $\GL_{n, \FFbar_\ell}$ by the equations of $G_{\FFbar_\ell}$ (whose number and degrees are bounded uniformly in $\ell$), together with the $n$ linear equations that encode the condition $g \cdot w = w$. In particular, the number of equations defining $(G_{\FFbar_\ell})_W$ and their degrees are bounded uniformly in $\ell$. The result then follows immediately from \cite[Proposition 1]{MR3649016} (notice that this result is stated for finite fields of characteristic $\ell$, but the proof makes it clear that it also applies to $\FFbar_\ell$).
\end{proof}

\subsection{Prime power case}

Throughout this section, we let $A$ be a nonzero abelian variety defined over a number field $K$ for which the groups $G_{A,\ell}$ are connected for all $\ell$.   We shall prove that Theorem~\ref{T:new main d1} holds when we restrict to the case where $n$ is a power of a prime $\ell$.   We first give an argument that will work for all sufficiently large $\ell$ and then one that will work for any remaining $\ell$.

\begin{prop}\label{P:large prime version, cyclic case}
For any prime $\ell \gg_A 1$, if $P \in A(\Kbar)$ is a point of order $\ell^e$, then
\[
[K(P):K] \gg_A  (\ell^e)^{\dim G_{A,\ell} - d_1(G_{A,\ell})}.
\]
\end{prop}
\begin{proof}
By taking $\ell\gg_A 1$, we may assume by Proposition~\ref{P:reductive} that the $\ZZ_\ell$-group scheme $\calG_{A,\ell}\subseteq \GL_{T_\ell(A)}$ is reductive and that $\ell$ is odd. By choosing a basis for $T_\ell(A)$ as a $\ZZ_\ell$-module, we identify $\calG_{A,\ell}$ with an algebraic subgroup of $\GL_{2g,\ZZ_\ell}$.   In particular, we have $(\calG_{A,\ell})_{\FF_\ell} \subseteq \GL_{2g,\FF_\ell}$.     Define $\gG:=\gG_\ell \subseteq M_{2g}(\FF_\ell)$ as in \S\ref{SS:lie algebras and filtrations}; it is the Lie algebra of $(\calG_{A,\ell})_{\FF_\ell}$ by Lemma~\ref{L:Lie algebra gG}.

Define the number field $L:=K(P)$ and the group $U:=A(L)[\ell^\infty]$. Following the proof of Theorem~\ref{T:ell-adic version, large ell} given in \S\ref{SS:proof of large ell}, we find that
\[
[L:K] \gg_A \ell^{\sum_{j=0}^{\infty}(\dim \gG - \dim \gG_{W_j})},
\]
where $W_j$ is a certain subspace of $\FF_\ell^{2g}$ that satisfies $W_j \cong U[\ell^{i+1}]/U[\ell^i]$ .  Note that the sum occuring in the exponent is actually finite since for $j$ large enough we have $W_j=0$ and hence $\gG_{W_j}=\gG$.

Since $U$ contains the point $P$ of order $\ell^e$, we have subspaces $W_j' \subseteq W_j$ satisfying $\dim W_j' = 1$ for $0\leq j <e$ and $W_j'=0$ for $j\geq e$. Since clearly $\dim \gG_{W_j'} \geq \dim \gG_{W_j}$, we obtain
\[
[K(P):K]=[L:K] \gg_A \ell^{\sum_{j=0}^{e-1}(\dim \gG - \dim \gG_{W_j'})}.
\]

Set $G:=(\calG_{A,\ell})_{\FF_\ell}$.   We claim that for $\ell \gg_A 1$, the group variety $G_W$ is smooth for all one-dimensional subspaces $W\subseteq \FF_\ell^{2g}$.  It suffices to show that $(G_{\FFbar_\ell})_W$ is smooth for all one-dimensional subspaces $W\subseteq \FFbar_\ell^{2g}$. Let $\{\calG_i\}_{i\in I}$ be the finite collection of reductive subgroups of $\GL_{2g,\ZZ}$ from Proposition~\ref{P:redux finiteness}.    By taking $\ell \gg_A 1$, Proposition~\ref{P:redux finiteness} implies that $G_{\FFbar_\ell}$ and $\calG_{\FFbar_\ell}$ are conjugate subgroups of $\GL_{2g,\FFbar_\ell}$, where $\calG=\calG_i$ for some $i\in I$.   The claim then follows from Lemma~\ref{L:generic smoothness of GW} and the finiteness of $I$.

By taking $\ell \gg_A 1$, the above claim proves that $G_{W_j'}$ is smooth for all $j$.  We thus have $\dim \gG_{W_j'} = \dim G_{W_j'}$ since $\gG_{W_j'}$ is the Lie algebra of the smooth group variety $G_{W_j'}$.   Therefore, $\dim \gG_{W_j'} \leq d_1(G)$ for all $j$.  We have $\dim \gG= \dim G$ since $G$ is reductive and hence smooth.  So we obtain the inequality
\[
[K(P):K] \gg_A (\ell^e)^{\dim G - d_1(G)}.
\]
We have $\dim G = \dim G_{A,\ell}$ since $\calG_{A,\ell}$ is reductive.   We have $d_1(G)=d_1(G_{A,\ell})$ for $\ell\gg_A 1$ by Lemma~\ref{L:d_1 agreement for large ell}.    The proposition is now immediate.
\end{proof}

\begin{prop}\label{P:small prime version, cyclic case}
For every prime $\ell$ and every point $P \in A(\overline{K})$ of order $\ell^e$, we have
\[
[K(P) : K] \gg_{A, \ell} (\ell^e)^{\dim G_{A,\ell} - d_1(G_{A,\ell})}.
\]
\end{prop}
\begin{proof}
Set $G=G_{A, \ell}$.    We begin by showing the following analogue of Proposition \ref{propmain}. Let $\Gcal \subseteq \GL_{n,\Z_\ell}$ be a linear group scheme such that $G = \Gcal_{\mathbb{Q}_\ell}$ has finite slope. There is a positive constant $C(\Gcal)$ such that for every $m \geq 1$ and every \textit{cyclic} subgroup $H \subseteq (\Z/\ell^m \Z)^n = \A^n(\ell^m)$ we have
 \begin{equation}\label{eq: small primes, cyclic case}
  [\Gcal(\Z_\ell) : \fix_{\Gcal(\Z_\ell)}(H)] \geq \frac{1}{C(\Gcal)} |H|^{\dim G - d_1(G)},
 \end{equation}
where
\[
 \fix_{\Gcal(\Z_\ell)}(H) := \{ M \in \Gcal(\Z_\ell) \, | \, \pi_m(M) \textrm{ is the identity on }H\}.
\]
  The proposition can be deduced from this statement (applied to the subgroup $H$ generated by $P$) exactly as Theorem \ref{T: small primes} is deduced from Proposition \ref{propmain}.

To prove \eqref{eq: small primes, cyclic case} we proceed as in the proof of Proposition \ref{propmain}, taking in particular the same notation. For the convenience of the reader, we repeat here the main steps, which in the present setting are much simpler than in the general case.

Since $H$ is cyclic, generated by a point of order $\ell^e$, we can take $r=1$ and $m_1=m=e$. We also set $\calW_1 := \Z_\ell e_1$ and $\calG_1 := \Fix_{\calG_{\Z_\ell}}(\calW_1)$, where $e_1 \in \Z_\ell^{2g}$ projects to a generator of $H$ modulo $\ell^m$.
We then obtain
\[
[\Gcal(\Z_\ell) : \fix_{\Gcal(\Z_\ell)}(H)] \geq \frac{|\pi_{m_1}(\Gcal(\Z_\ell))|}{|\pi_{m_1}( \fix_{\Gcal(\Z_\ell)}(H))|}.
\]
The numerator is lower-bounded by $C_1(\mathcal{G}) \ell^{m \dim G}$, see Lemma \ref{lemlowbound}. The denominator is at most the cardinality of $\calG_{1}(\ell^{m})$, which by Proposition \ref{propkey} (applied to $\calG_1$, $r=1$ and $m'=0$) is at most $C(\calG) \ell^{d_1(G) m}$. Hence, the quantity above is lower-bounded by $\frac{C_1(\calG)}{C(\calG)} \ell^{m(\dim G-d_1(G))}$, as desired.
\end{proof}

\subsection{Proof of Theorem \ref{T:new main d1}}
After replacing $K$ by a finite extension, 
and $A$ by its base extension to this field, we may assume that all the groups $G_{A,\ell}$ are connected by Proposition~\ref{P:connected}.  This is allowable  since it will only increase the degree of $K(P)$ by a uniformly bounded quantity and will not change the groups $G_{A,\ell}^\circ$. 

Let $P \in A(\overline{K})$ be any torsion point and denote its order by $n$.  We have a factorization $n=\prod_{\ell|n} \ell^{e_\ell}$ and $P=\sum_{\ell | n} P_\ell$ with $P_\ell\in A(\overline{K})$ a torsion point of order $\ell^{e_\ell}$.  By Proposition \ref{P:independence} there exists a finite extension $K'$ of $K$ such that the extensions $K'(A[\ell^\infty])/K'$ are all linearly disjoint as $\ell$ ranges over the prime numbers. Thus, we have
\begin{equation}\label{eq: decomposition, cyclic case}
[K(P):K] \gg_A [K'(P) : K'] = \prod_{\ell|n} [K'(P_\ell) : K'].
\end{equation}
By Propositions \ref{P:large prime version, cyclic case} and \ref{P:small prime version, cyclic case}, we obtain the inequality
\begin{equation}\label{eq: lower bound for one point, prime power version}
[K'(P_\ell) : K'] \geq C\cdot (\ell^{e_\ell})^{\dim G_{A, \ell} - d_1(G_{A, \ell})} \geq C \ell^{\varepsilon} \cdot (\ell^{e_\ell})^{\dim G_{A, \ell} - d_1(G_{A, \ell}) - \varepsilon}
\end{equation}
for all $\ell|n$, where $C$ is a positive constant depending only on $A$.  Combining \eqref{eq: decomposition, cyclic case} and \eqref{eq: lower bound for one point, prime power version}, and using that $C \ell^\varepsilon \geq 1$ for all $\ell\gg_{A,\varepsilon} 1$, we obtain
\begin{align*}
[K(P):K] & \gg_{A,\varepsilon}  \prod_{\ell|n} (\ell^{e_\ell})^{\dim G_{A, \ell} - d_1(G_{A, \ell})-\varepsilon}
\end{align*}
and hence $[K(P):K] \gg_{A,\varepsilon}  n^{\delta-\varepsilon}$.

It remains to show that $\delta$ is the maximal value for which $[K(P):K] \gg_{A,\varepsilon}  n^{\delta-\varepsilon}$ holds for all $\varepsilon>0$, all positive integers $n$ and all torsion points $P\in A(\Kbar)$ of order $n$.    Choose a prime $\ell$ for which $\delta=\dim G_{A, \ell} - d_1(G_{A, \ell})$. There is a one-dimensional subspace $W$ of $V_\ell(A)$ for which $d_1(G_{A, \ell})=\dim (G_{A,\ell})_W$.   Define $\calW := W \cap T_\ell(A)$; it is a rank one $\ZZ_\ell$-module.  Let $w$ be a $\Z_\ell$-basis of $\mathcal{W}$. For every $e\geq 1$, the image of $w$ in $T_\ell(A)/\ell^eT_\ell(A)=A[\ell^e]$ defines a torsion point $P_e$ of exact order $\ell^e$. Arguing as in the proof of Theorem \ref{T:main unconditional}, we obtain
\[
[K(P_e):K] = \frac{[K(A[\ell^e]) : K]}{[K(A[\ell^e]) : K(P_e)]} \asymp_{A, \ell} (\ell^e)^{\dim G_{A,\ell} - \dim((G_{A,\ell})_W)}=(\ell^e)^\delta.
\]
Letting $e\to \infty$, we see that $\delta$ in the statement of the theorem cannot be replaced with a larger value.

\subsection{Proof of Theorem \ref{T:main d1}}  \label{SS:proof in main d1}

After replacing $K$ by a finite extension, 
and $A$ by its base extension to this field, we may assume that all the groups $G_{A,\ell}$ are connected by Proposition~\ref{P:connected}.  Note that this only changes the degree $[K(P):K]$ by a uniformly bounded amount.  Let $\delta$ be the minimal value $\dim G_{A,\ell} - d_1(G_{A,\ell}))$ as we vary over all primes $\ell$.  The theorem will then be a direct consequence of Theorem \ref{T:new main d1} once we show that $\delta=d$.  

We claim that $\delta=\dim G_A - d_1((G_A)_\Qbar)$.   We have $(G_A)_{\QQ_\ell}=G_{A,\ell}$ since we are assuming the Mumford--Tate conjecture for $A$, and hence $\dim G_A = \dim G_{A,\ell}$ and $d_1((G_A)_{\QQ_\ell})=d_1(G_{A,\ell})$.  Therefore, $\delta=\dim G_A - \max_\ell d_1((G_{A})_{\QQ_\ell})$ and we need only prove that $d_1((G_A)_{\Qbar}) = \max_\ell d_1((G_{A})_{\QQ_\ell})$.  For all $\ell$, we have $d_1((G_A)_{\Qbar}) \geq d_1((G_A)_{\QQ_\ell})$ by Lemma~\ref{L:new d1 field case}(\ref{L:new d1 field case i}).  So it suffices to show that $d_1((G_A)_{\Qbar}) \leq d_1((G_A)_{\QQ_\ell})$ for some prime $\ell$.  Set $m:=d_1((G_A)_{\Qbar})$.  By choosing a basis, we may assume $G_A$ is a subgroup of $\GL_{2g,\QQ}$.  With $R=\QQ$ and $G=G_A$, let $Z_m$ be the subscheme of $\AA_\QQ^{2g}$ from Lemma~\ref{L:new d1}.   By our choice of $m$, there is a nonzero $w\in Z_m(\Qbar)$.   Let $L \subseteq \Qbar$ be a number field for which $w\in Z_m(L)$.   Take any prime $\ell$ that splits completely in $L$.   There is an embedding $L\hookrightarrow \QQ_\ell$ and hence $Z_m(\Q_\ell)$ contains a nonzero element.  From our description of the set $Z_m(\QQ_\ell)$ from Lemma~\ref{L:new d1}, we deduce that $d_1((G_A)_{\QQ_\ell}) \geq m$ which completes our proof of the claim.

By the above claim and Lemma~\ref{L:new d1 field case}(\ref{L:new d1 field case ii}), we have $\delta=\dim G_A - d_1((G_A)_\CC)$.  By choosing a basis for $V_A$, we may view $G_A$ as a subgroup of $\GL_{2g,\QQ}$.   Take any nonzero $w\in \CC^{2g}$.  With $W:=\CC w$, define $H:= ((G_A)_\CC)_W$; it is the subgroup of $(G_A)_\CC$ that fixes $w$.   Let $(G_A)_\CC \to \AA_\CC^{2g}$ be the morphism $g\mapsto g \cdot w$.   This morphism factors through an immersion $(G_A)_\CC/H \hookrightarrow \AA_\CC^{2g}$, cf.~\cite[Proposition~7.17]{MilneAlgebraicGroups}.   Therefore, the orbit of $w$ under the action of $(G_A)_\CC$ has dimension $\dim G_A - \dim ((G_A)_\CC)_W$.   By varying over all nonzero $w\in \CC^{2g}$, we deduce that $d= \dim G_A - d_1((G_A)_\CC)$.  Therefore, $\delta=d$.

\bibliographystyle{plain}
%\bibliography{../bib/master}
% \bib, bibdiv, biblist are defined by the amsrefs package.

\end{document}